\documentclass[12pt]{amsart}%{scrartcl}%{amsart}

\usepackage{euscript}
\usepackage{amsmath}
\usepackage{amsthm}
\usepackage{epsfig}
\usepackage{amssymb}

\usepackage{epic}

\DeclareMathOperator{\supp}{{\bf supp}}

\DeclareMathOperator{\coker}{{\rm coker}}

\numberwithin{equation}{section}
\numberwithin{equation}{subsection}

\theoremstyle{plain}
\newtheorem{theorem}[equation]{Theorem}
\newtheorem{lemma}[equation]{Lemma}
\newtheorem{proposition}[equation]{Proposition}
\newtheorem{corollary}[equation]{Corollary}

\theoremstyle{definition}
\newtheorem{example}[equation]{Example}
\newtheorem{remark}[equation]{Remark}

\newtheorem{definition}[equation]{Definition}

\newtheorem{goal}[equation]{Goal}

\newtheorem{bekezdes}[equation]{}

\numberwithin{equation}{section}
\numberwithin{equation}{subsection}

\title{Generalized Monodromy Conjecture in dimension two}

\author{Andr\'as N\'emethi}
\address{A. R\'enyi Institute of Mathematics, 1053 Budapest,  Re\'altanoda u. 13-15,  Hungary.}
\email{nemethi@renyi.hu}
\thanks{The first author is partially supported by OTKA Grants.
The second author is partially supported by FWO--Flanders project G.0318.06}

\author{Willem Veys}
\address{K.U.Leuven, Dept. Wiskunde, Celestijnenlaan 200B, 3001 Leuven, Belgium}
\email{wim.veys@wis.kuleuven.be
}

\keywords{Monodromy conjecture, topological zeta function,  monodromy,
surface singularity, plane curve singularities, resolution graphs,
semigroup condition, splice diagrams, splice decomposition}

\subjclass[2000]{Primary. 14B05, 32S40, 14H20;
Secondary. 32S05, 14H50, 14J17, 32S25}

\date{}

\begin{document}

\maketitle

%\begin{center}
% {\em Dedicated to }
%\end{center}

\pagestyle{myheadings} \markboth{{\normalsize A. N\'emethi and W.
Veys}}{ {\normalsize Generalized Monodromy Conjecture}}

\renewcommand{\div}{{\rm div}}

\newcommand{\et}{\mathcal{T}}
\newcommand{\bS}{{\mathbb S}}
\newcommand{\bma}{\mbox{\boldmath$a$}}
\newcommand{\bmb}{\mbox{\boldmath$b$}}
\newcommand{\bmc}{\mbox{\boldmath$c$}}
\newcommand{\bme}{\mbox{\boldmath$e$}}
\newcommand{\bmi}{\mbox{\boldmath$i$}}
\newcommand{\bmj}{\mbox{\boldmath$j$}}
\newcommand{\bmv}{\mbox{\boldmath$v$}}
\newcommand{\bmk}{\mbox{\boldmath$k$}}
\newcommand{\bmm}{\mbox{\boldmath$m$}}
\newcommand{\bms}{\mbox{\boldmath$s$}}
\newcommand{\bSW}{\mbox{\boldmath$SW$}}
\newcommand{\bmf}{\mbox{\boldmath$f$}}
\newcommand{\bmg}{\mbox{\boldmath$g$}}
\newcommand{\gq}{{\mathfrak q}}
\newcommand{\xo}{o}
\newcommand{\veeK}{{\tiny\vee}}
\newcommand{\gh}{g}

\newcommand{\lp}{{l}}
\newcommand{\ev}{\varepsilon}
\newcommand{\tx}{\tilde{X}}
\newcommand{\tz}{\tilde{Z}}
\newcommand{\calL}{{\mathcal L}}
\newcommand{\calm}{{\mathcal M}}
\newcommand{\calx}{{\mathcal X}}
\newcommand{\calo}{{\mathcal O}}
\newcommand{\calt}{{\mathcal T}}
\newcommand{\cali}{{\mathcal I}}
\newcommand{\calj}{{\mathcal J}}
\newcommand{\calC}{{\mathcal C}}
\newcommand{\calS}{{\mathcal S}}
\newcommand{\calQ}{{\mathcal Q}}
\newcommand{\calF}{{\mathcal F}}

\newcommand{\cs}{\langle \chi_0\rangle}

\newcommand{\cc}{\bar{C}}
\newcommand{\vp}{\varphi}

\def\mmod{\mbox{mod}}
\let\d\partial
\def\EE{\mathcal E}
\newcommand{\cC}{\EuScript{C}}
\def\C{\mathbb C}
\def\Q{\mathbb Q}
\def\R{\mathbb R}
\def\bS{\mathbb S}
\def\bH{\mathbb H}
\def\bB{\mathbb B}\def\bC{\mathbb C}\def\bA{\mathbb A}
\def\Z{\mathbb Z}
\def\N{\mathbb N}
\def\bn{\mathbb N}
\def\bp{\mathbb P}\def\bt{\mathbb T}
\def\eop{$\hfill\square$}
\def\bif{(\, , \,)}
\def\coker{\mbox{coker}}
\def\im{{\rm Im}}

\newcommand{\Gammma}{{G}}
\newcommand{\no}{\noindent}
\newcommand{\bfc}{{\mathbb C}}
\newcommand{\bfq}{{\mathbb Q}}
\newcommand{\calE}{{\mathcal E}}
\newcommand{\calW}{{\mathcal W}}
\newcommand{\calV}{{\mathcal V}}
\newcommand{\calP}{{\mathcal P}}
\newcommand{\calI}{{\mathcal I}}\newcommand{\calJ}{{\mathcal J}}
\newcommand{\calA}{{\mathcal A}}\newcommand{\CalA}{{\calA_F\cup\calA_W}}
\newcommand{\calAA}{{\mathcal A}'}
\newcommand{\calB}{{\mathcal B}}
\newcommand{\calR}{{\mathcal R}}
\newcommand{\calG}{{\mathcal G}}\newcommand{\calN}{{\mathcal N}}
\newcommand{\bc}{{\mathbb C}}
\newcommand{\bez}{B_{\epsilon_0}}
\newcommand{\br}{{\mathbb R}}
\newcommand{\bq}{{\mathbb Q}}
\newcommand{\sez}{S_{\epsilon_0}}
\newcommand{\ep}{\epsilon}
\newcommand{\vs}{\vspace{3mm}}
\newcommand{\si}{\sigma}
\newcommand{\Gammas}{S}
\newcommand{\Gx}{G_\pi(X)}
\newcommand{\Gxf}{G_\pi(X,f)}
\newcommand{\GxF}{G_\pi(X,F)}
\newcommand{\Gax}{\Gamma_\pi(X)}
\newcommand{\Gaxf}{\Gamma_\pi(X,f)}
\newcommand{\GaxF}{\Gamma_\pi(X,F)}
\newcommand{\GaxFW}{\Gamma_\pi(X,F,W)}

\newcommand{\q}{w}

\newcommand{\labelpar}{\label}

\begin{abstract}
The aim of the article is an extension  of the Monodromy Conjecture of Denef and Loeser in dimension two, incorporating zeta functions with differential forms and targeting {\em all}  monodromy eigenvalues, and also considering singular ambient spaces. That is, we treat in a conceptual unity  the poles of the (generalized) topological zeta function
and the monodromy eigenvalues associated with an analytic germ $f:(X,0)\to (\C,0)$ defined on a normal surface
singularity $(X,0)$.  The article targets the `right' extension
in the case when the link of $(X,0)$ is a homology sphere.
As a first step, we prove a splice decomposition formula for the topological zeta function $Z(f,\omega;s)$
for any $f$ and analytic differential form $\omega$, which will play the key technical localization tool in the later
definitions and proofs.

Then, we define a set of `allowed' differential forms via a local restriction
along each splice component. For plane curves we show the following three guiding  properties:
(1) if $s_0$ is any pole of $Z(f,\omega;s)$ with $\omega$ allowed, then $\exp(2\pi is_0)$ is a monodromy
eigenvalue of $f$, (2)  the `standard' form is allowed, (3) every monodromy eigenvalue of $f$ is obtained as in (1)
for some allowed $\omega$ and some $s_0$.

For general $(X,0)$ we prove (1) unconditionally, and (2)--(3) under an additional (necessary) assumption, which generalizes the semigroup condition of Neumann--Wahl. Several examples illustrate the definitions and support the
basic assumptions.
\end{abstract}

%\tableofcontents{}

{\small

\section{Introduction}

\subsection{}
{\bf The Monodromy Conjecture} of Igusa, Denef and Loeser \cite{DL1,DL2}
is one of the most fertilizing
conjectures in singularity theory. It relates poles of
Igusa/motivic/topological zeta functions to monodromy eigenvalues.
For instance,   for a local analytic {\em isolated}  singularity
$f:(\C^n,0)\to (\C,0)$ it predicts that if $s_0$ is a pole of the
local topological zeta function of $f$, then $\exp(2\pi is_0)$ is
an eigenvalue of the local monodromy operator acting on
$H^*(F_0,\C)$, where $F_0$ is the Milnor fiber of $f$.
In the definition of the topological zeta function not only
some invariants of the local germ $f$ are codified, but in a
subtle way also some numerical data of the standard differential form of
 $(\C^n,0)$ lifted to an embedded resolution of $f$.

 The conjecture was proved for $n=2$ by Loeser (originally in the
context of $p$-adic Igusa zeta functions) in \cite{Lo2}. There are
by now various other partial results, e.g.
\cite{ACLM1,ACLM2,BMT,LP,LV,Lo3,Ve1,Ve6}, nevertheless the conjecture resists
to all attacks (even for $n=3$).
%\marginpar{OTHER CASES???}
The main obstacle is the lack of a conceptual bridge connecting the two
invariants, the topological zeta function and the monodromy operator;
the existent proofs of the particular cases basically compute both sides independently
(using their special properties) and compare the two final data.

A possible way to find a more conceptual understanding and tools is to
extend the conjecture to a larger class. This leads us
to the replacement of $(\C^n,0)$ with a singular space, and of the standard
differential form with some generalization of it.
Although both types of generalizations are obstructed (see the next subsections),
{\em the main target  of the  present article is to find the right such extension
when the ambient space is 2--dimensional}.
Since the two types of generalizations are independent, and have rather different effects,
in order to understand their nature,
at the first discussion we separate them.

\subsection{} {\bf Extending the differential form.} There is a more direct
motivation for the generalization of the standard form.
It is easy to see on explicit examples that for any fixed germ $f:(\C^n,0)\to (\C,0)$,
and by considering the `classical' topological zeta function,   not all
the eigenvalues of the monodromy operator are realized; actually
quite few eigenvalues are obtained this way (in general). Hence, for any fixed $f$, it
is natural to try to extend in some way this set of poles, such
that the same procedure would yield {\em all} eigenvalues of $f$.
We expect that such a construction could reveal the conceptual bridge mentioned above.
A natural way
to extend poles is using the local topological zeta functions
associated with the original germ $f$ and with {\em a set of analytic
differential $n$-forms $\omega$ living in $(\C^n,0)$}.

We now describe these zeta functions; they are defined in
terms of an embedded resolution $\pi:\tilde{X}\to\C^n$ of $f^{-1}
( 0 ) \cup \mbox{div}(\omega)$. We denote by $E_i,\ i \in S$, the
irreducible components (exceptional divisors and strict transforms)
of the inverse image $\pi^{-1} (f^{-1}(0)
 \cup \mbox{div}(\omega))$ and by $N_i$ and $\nu_i - 1$ the
multiplicities of $E_i$ in the divisor of $\pi^\ast f$ and
$\pi^\ast \omega$, respectively. We put $E^\circ_I := (\cap_{i \in
I} E_i) \setminus (\cup_{j \notin I} E_j)$ for $I \subset S$.
%in
%particular $E^\circ_\emptyset = X \setminus (\cup_{j \in S} E_j)$.
Hence,  the $E^\circ_I$ constitute  a stratification of $\tilde{X}$ in locally
closed subsets.

\begin{definition}\label{I:1}  The  {\sl (local) topological zeta function} of
$(f,\omega)$ at $0 \in \C^n$ is
$$Z (f, \omega; s) := \sum_{I \subset S} \chi (E^\circ_I \cap \pi^{-1} ( 0 ))
\prod_{i \in I} \frac{1}{\nu_i + sN_i},$$  where $s$ is a
variable.
\end{definition}

This definition extends naturally the definition of Denef and Loeser from  \cite{DL1} valid
for the standard form $\omega = dx_1 \wedge
\cdots \wedge dx_n$. Their original proof that the corresponding  expression
does not depend on the chosen resolution is by describing it as a
kind of limit of $p$-adic Igusa zeta functions. Later they
obtained the statement as a specialization of the intrinsically
defined motivic zeta functions \cite{DL2}. Another technique is
applying the Weak Factorization Theorem \cite{AKMW,Wl} to compare
two different resolutions. For arbitrary $\omega$ one can proceed
analogously. Hence $Z (f, \omega; s)$ is a well--defined invariant of the
pair $(f,\omega)$.

In the literature similar generalizations are already present,
see for example \cite{ACLM1,ACLM2,Ve4}; however they are subject to the
restriction  $\supp(\mbox{div}(\omega)) \subset f^{-1}(0)$.  In the present
article  we release this condition.  (In the original context of $p$-adic Igusa zeta functions,
see e.g. \cite[III 3.5]{Lo2}.)

Although $Z (f, \omega; s)$ is a sum of `local' contributions, in this sum many local
candidate poles cancel, and usually it is hard to characterize those which survive.

\bekezdes{}
We recall that the \lq classical\rq\ monodromy conjecture predicts
the implication
$$s_0\ \text{is a pole of}\ Z(f,\omega;s) \Rightarrow \exp(2\pi is_0)\ \text{is
a monodromy eigenvalue of}\ f,
$$
where $\omega$ is the standard form $dx_1 \wedge \cdots
\wedge dx_n$. The point is that for arbitrary analytic
differential forms $\omega$ this  implication is in general
false. Even more, in \cite{VeNew} the second author showed that every given
monodromy eigenvalue of $f$ is induced by a pole $s_0$ of some
$Z(f,\omega;s)$, but in general that zeta function has other \lq
bad\rq\ poles, not inducing eigenvalues.
This shows, that one can indeed generate a lot of poles, but their
relationship with the eigenvalues is uncontrolled. The next program
targets exactly this uncertainty via the selection of forms with compatibility properties with
the monodromy operator.

Partly initiated in \cite{VeNew} (see also \cite{NV}) we propose the following
program.

\begin{goal}\label{goal:in} \
Define/identify a collection of {\it allowed} analytic forms $\omega$ (depending on $f$) such that

\begin{enumerate}
\item   if $s_0$ is any pole of $Z(f,\omega;s)$, where $\omega$ is allowed, then $\exp(2\pi is_0)$ is
a monodromy eigenvalue of $f$,
\item   the standard form $\omega=dx_1 \wedge \cdots
\wedge dx_n$ is allowed,
\item   every monodromy eigenvalue of $f$ is obtained as in (1)
for some allowed $\omega$ and some $s_0$.
\end{enumerate}
\end{goal}

A few remarks are in order. First, note that (1) and (2)
imply the classical Monodromy Conjecture. Furthermore, (1) and (3) combined
show that the set of eigenvalues of $f$ coincides with the set $\exp(2\pi iP)$,
where $P$ runs over all the poles of the zeta functions of $f$ and all allowed forms.

Note also that the `size' of the wished allowed forms
 is obstructed by both conditions (1) and (3). A
larger set is obstructed more by (1), while if this set  is too small
then it may not realize in (3) all eigenvalues. In particular, its construction
really requires a conceptual understanding of the geometry of the
pole--eigenvalue bridge mentioned above.

\smallskip
Let us briefly support  our goal by comparing with a more classical
context. Recall that the topological zeta function is a kind of
avatar of the $p$-adic Igusa zeta function, which is the
meromorphic continuation of a $p$-adic integral associated to a
$p$-adic function germ $f$ (with complex parameter/variable
$s$). We could rephrase the above also for that zeta function, again
with complex parameter/variable $s$. For
the analogous {\em complex integral associated to a complex function
germ $f$}, and involving {\em compactly supported
$C^\infty$ forms $\omega$}, there are general \emph{theorems} by
Malgrange \cite{Ma2}, Kashiwara \cite{Ka} and Barlet \cite{Ba2},
claiming the analogous statements of the goals above. Roughly, if
$s_0$ is any pole of the zeta function of $f$ and {\it any}
compactly supported $C^\infty$ form $\omega$, then $\exp(2\pi is_0)$ is a monodromy
eigenvalue of $f$, and all eigenvalues are obtained this way. Our
\lq standard form\rq\ can be compared with a $C^\infty$ form that
is non-vanishing.
 For a detailed explanation and comparison, we refer to the introduction of
\cite{VeNew}. However, the `exact comparison' in general fails, and it is still
hidden what the analogue of compactly supported   $C^\infty$ forms is in the
holomorphic category.

\smallskip One of the  main results of this
paper is an identification of a
%possible
set of
allowed forms realizing the goals (1)--(2)--(3)
above for $n=2$, that is for  an \emph{arbitrary} plane curve germ $f$.

Our technique is to consider the so--called splice diagram of $f$,
and its splice decomposition in star--shaped pieces. It is not
difficult (and reasonably conceptual) to define the allowed forms and
realize our goal when the
diagram of $f$ itself is star--shaped; we then use this as
inspiration for the general case, identifying allowedness \lq
locally\rq, that is, on all star--shaped subdiagrams, cf.   section \ref{s:3}.

\subsection{} {\bf Extension to a singular ambient space.}
An important new feature in this paper is generalizing
(\ref{goal:in}) to a singular setting. More precisely, we will consider an analytic
function germ $f$ defined on a normal surface germ $(X,0)$
and study the analogue of the goals (1)--(2)--(3)  for this $f$.
First we must identify the relevant (generalized) topological zeta function
for such a pair $(X,f)$ and for an analytic form $\omega$.
In fact, in all our combinatorial arguments, we will replace $(f,\omega)$ by two
Weil divisors (for details and motivation see subsection (\ref{bek:INTR1})).

Of course, in order to have a well--defined analogue  of part (2) of Goal \ref{goal:in}, we need to
consider Gorenstein germs, which guarantees the existence of a `standard form'.
If the ambient space is two-dimensional,
 the Gorenstein condition
simplifies at topological/combinatorial level
to the numerically Gorenstein condition, which is {\it automatically  guaranteed},
for example, if the link is an integral homology sphere.

Nevertheless, we will need some further combinatorial restrictions.
There is an example of Rodrigues \cite{Ro} indicating  that the `naive'
extension of the Monodromy Conjecture to
the Gorenstein singular setting might be obstructed.
This example produces a set of integers $\{\nu_i\}_i$ and  $\{N_i\}_i$  associated with
the exceptional divisors (or, with vertices of the plumbing graph)
which topologically are not obstructed to be the multiplicities of the
standard form and of an analytic germ $f$, and they produce a counterexample
to the Monodromy Conjecture.
In (\ref{ex:unimod}) we even construct another such example
involving a Gorenstein surface singularity {\em with unimodular dual graph}.
One of the following two possibilities can solve this situation in order to have a chance
for a positive continuation:
either we impose some  additional  topological restrictions which eliminate any such
counterexample, or we try to show that the analytic realization of the analytic germ
(in the presence of the Gorenstein structure) guarantees these additional needed
topological restrictions. The second possibility
looks very difficult and is hopeless with the present tools of the theory, and it is not the goal
of the present article to attack it. Therefore, in order to have an extended version,
we stay with the first possibility.

The additional restriction we impose, in fact,  is very natural;
it  is a modification  of the {\it semigroup condition}
of Neumann and Wahl \cite{NW1}, adapted to the present situation and
to `divisors supported on a graph'.
This condition is automatically satisfied if the ambient space is smooth (and in several other cases too).
This also  emphasizes a subtle connection between
the semigroup condition and the Monodromy Conjecture.

Our second main result extends the combinatorial definition of the allowed forms
to the singular surface case when the link of the ambient space is
an integral homology sphere, and establishes
for them  Goal  (1) unconditionally, and (2) and (3) under the semigroup condition.

Allowed forms are again defined via the same local picture of the splice components.
We emphasize that the definition of the
allowed forms, the generalized semigroup condition, and the whole proof of Goal \ref{goal:in} is
combinatorial: one uses only the Weil divisors of the functions and forms
(and their analytic realizations will be not involved).
In subsection (\ref{bek:INTR1}) we give some details about the formalism of
zeta functions associated with Weil divisors,  and in (\ref{bek:INTR2}) we motivate the definition
of allowed forms. An ambient germ with integral homology sphere link will be
 abbreviated by IHS germ.

\subsection{}
Here is the {\bf plan of the paper} and the list of the most important {\bf new results}.

In the next section we recall the classical notion of splice diagrams for surface germs and for functions/divisors on them, and we incorporate in the picture also differential forms (and their generalizing divisors). Here we also introduce the extension of the semigroup condition of Neumann and Wahl in the presence of a divisor. Section 3 treats the concept of splicing of these diagrams. The relevant \lq splice formulae\rq\ are well known for functions on surface germs; we develop them for differential forms. Then we use these formulae to derive a splice formula for the topological zeta function, see Theorem \ref{prop:splicezeta}. In section 4 we define our allowed forms/divisors and investigate their crucial properties concerning restriction and extension along (sub)diagrams. In section 5 we show that any pole of the zeta function associated to a function $f$ and any allowed form/divisor induces a monodromy eigenvalue of $f$ (first goal). The second goal (standard form is allowed) and third goal (any eigenvalue is induced by a pole of a zeta function of an allowed form) are proved for plane curve germs
 (unconditionally) in section 6. Finally their generalized versions for functions on IHS germs
 (under the semigroup condition) are proven in section 7.

Additionally, we list several examples in order to make the manuscript more readable, and in order to
emphasize the role of several key points in definitions or about needed restrictions.
For example, Examples \ref{ex:EXAMPLE4} and  \ref{ex:EXAMPLE7} show that
the semigroup condition is necessary to have Goal (2) and (3), respectively, (at least in any topological
treatment), while
the discussion from  \ref{ex:unimod} shows that if we drop the IHS assumption about
the link of the ambient space we need to treat  a much stronger (and presumably more technical)
notion replacing the semigroup condition.

\subsection{} {\bf Some more details and motivations.} Here we present the key motivations for the
major restrictions and constructions of the article as a separated guide.

\bekezdes{}\labelpar{bek:INTR1}
 We first explain what the {\bf topological zeta function is on a
{\em  singular} ambient surface} associated with two Weil divisors.

There have been various generalizations of topological and motivic
zeta functions to singular ambient varieties $X$ (instead of
$\C^n$), see for example \cite{Ve4}, and specifically for surfaces
\cite{Ro,RoVe,Ve3}. Before introducing the ones we will use,
note that the zeta function $Z(f, \omega; s)$ of
Definition \ref{I:1} depends in fact only on the effective
divisors $F:=\div(f)$ and $W:=\div(\omega)$ on $(\C^n,0)$, and not
on the actual function $f$ and form $\omega$. In terms of these
divisors, the numerical data $N_i$ and $\nu_i-1$ above are given
as the multiplicities of $E_i$ in the divisors $\pi^*F$ and $K_\pi
+ \pi^*W$, respectively, where $K_\pi$ is the relative canonical
divisor of $\pi$.

Therefore, it is natural to associate in a singular setting a
topological zeta function to two {\it Weil  divisors} on $X$, in
terms of an embedded resolution $\pi$ of the union of their
supports. For this we should be in a situation where there is a
natural notion of pullback of Weil divisors, and where the
relative canonical divisor $K_\pi$ exists.
Both conditions are satisfied when $(X,0)$ is an arbitrary normal
surface germ (for the pullback see (\ref{2.2.2}), while for
$K_\pi$ see (\ref{eq:adj})). Let $F:=\sum_{j\in J} N_j E_j$ be an effective
non-zero Weil divisor, and $W:=\sum_{j\in J} (\nu_j-1) E_j$ an
arbitrary Weil divisor on $X$, where $E_j,\ j\in J,$ are (finitely
many) irreducible Weil divisors. We only require that $(N_j,\nu_j)
\neq (0,0)$ for $j\in J$, that is, a component $E_j$ that appears
in $W$ with multiplicity $-1$ must appear in the support of $F$.

Let $\pi:\tilde{X}\to X$ be an embedded resolution of $\supp(F)
\cup \supp(W)$. We denote again by $E_i,\ i \in S,$ the
irreducible components of its inverse image, and by $N_i$ and
$\nu_i-1$ the multiplicities of $E_i$ in the divisors $\pi^*F$ and
$K_\pi + \pi^*W$, respectively. Note that the $N_i$ and $\nu_i$ of
exceptional components $E_i$ are in general rational numbers.
  The  {\sl (local) topological zeta function} of
$(F,W)$ at $0 \in X$, denoted as $Z (F, W; s)$, is defined by the
same formula as in Definition \ref{I:1}. It is straightforward to
verify that this expression does not depend on the chosen
resolution $\pi$.

In our study below we will assume that the link of $X$ is an integral homology sphere; then in particular all $N_i$ and $\nu_i$ are integers.
In such a context,
more useful for us will be a formula for $Z (F, W; s)$ in terms of
the {\em splice diagram associated to $(F,W)$}, see
(\ref{ss:zetafunctions}). Roughly, the splice diagram  is obtained from
the dual minimal embedded resolution graph of  $\supp(F) \cup \supp(W)$ by collapsing the
strings to edges and modifying the decorations  by a system of data which describes more
trustworthily the needed linking numbers.

This description allows us to determine a  splice formula showing the `almost additivity'
of   $Z (F, W; s)$ with respect to the splice decomposition of the diagrams. This is
another novelty of the article, which becomes a  crucial tool in the main proofs.

\bekezdes{}\labelpar{bek:INTR3} ({\bf Restrictions regarding $W$.})
It is an easy fact that in the resolution graph
exceptional components of valency 1 or 2
do \emph{not} contribute to the actual poles of the zeta function,
and that those of valency at least 3 (corresponding to the nodes
in the splice diagram) in general do contribute to the poles.
Since precisely those last components contribute to the monodromy
eigenvalues, it is very reasonable to restrict from the start the
support of the desired allowed $W$ as follows. The map $\pi$
should be also an embedded resolution of $f^{-1}(0) \cup
\supp(W)$, and more precisely $\supp(W)$ should consist only of
components coinciding with components of $f$ and components whose
strict transform intersects the exceptional locus in a component
of valency 1; moreover such a component of valency 1 must
intersect at most one component of $W$. In this way, one does not create new
exceptional components of valency at least 3 and does not transform
those of valency 1 or 2 into components of valency 3, what would
probably create undesired new poles. This restriction has a similar formulation in
terms of splice diagrams too.

\bekezdes{}\labelpar{bek:INTR2}
 Next, we give the {\bf idea of the definition of
allowed forms/divisors} associated with
a fixed function $f$ (or Weil divisor $F$), with some motivation.
It will be illustrated via  the plane curve germ
given by the function $f =(y^{d_1}-x^{d_2})(y^{d_1}+x^{d_2})$,
where $d_1>1$, $d_2>1$ and $\gcd(d_1,d_2)=1$. Below is its
star-shaped splice diagram with node $E$, where the dashed
arrows indicate the support of the strict transform of
$W=\div(\omega)$, and the decorations along these arrows are the multiplicities of
its components (for details, see section \ref{s:1}).
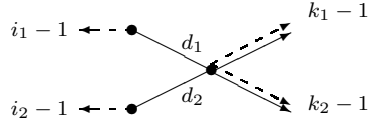
\smallskip

\begin{picture}(200,55)(265,0)
\put(440,20){\circle*{4}}\put(410,35){\circle*{4}}\put(410,5){\circle*{4}}
\put(440,19){\vector(2,1){30}}\put(440,19){\vector(2,-1){30}}
\put(440,20){\line(-2,-1){30}} \put(440,20){\line(-2, 1){30}}
%\put(440,-10){\makebox(0,0){$E(N,\nu)$}}
\put(433,30){\makebox(0,0){\tiny{$d_1$}}} \put(433,10){\makebox(0,0){\tiny{$d_2$}}}
\put(488,42){\makebox(0,0){\tiny{$k_1-1$}}}
\put(488,7){\makebox(0,0){\tiny{$k_2-1$}}}
\put(375,5){\makebox(0,0){\tiny{$i_2-1$}}}
\put(375,35){\makebox(0,0){\tiny{$i_1-1$}}}
%\put(560,20){\makebox(0,0){($r$ arrowheads)}}
\dashline[3]{3}(440,22)(465,35)\put(465,35){\vector(2,1){5}}
\dashline[3]{3}(440,22)(465,9)\put(465,9){\vector(2,-1){5}}
\dashline[3]{3}(410,35)(395,35)\put(395,35){\vector(-1,0){5}}
\dashline[3]{3}(410,5)(395,5)\put(395,5){\vector(-1,0){5}}
\end{picture}

%We now (i) describe the Alexander polynomial of $f$, see
%(\ref{ss:splicingalex}), whose roots are (essentially) the
%monodromy eigenvalues of $f$, (ii) give a concise formula for
%$Z(s)=Z(F,W;s)$, and (iii) derive from this a natural choice for
%allowed $W$ yielding our goal.

\medskip
\noindent
Then  $N=2d_1d_2$ is the vanishing order of $f$ along the node $E$.
%One gets the next facts:

\smallskip

(i) By A'Campo's formula,
 the monodromy eigenvalues of $f$ are, besides the trivial eigenvalue $1$,
 precisely the roots of the polynomial
$$
\Lambda(t)=\frac {(t^N-1)^2}{(t^{N/d_1}-1)(t^{N/d_2}-1)}.
$$
 These are all $N$-th roots of unity that are {\it not}
simultaneously $(N/d_1)$-th and $(N/d_2)$-th roots of unity; in
other words all $\exp(2\pi i \frac uN)$ for which $d_1 \nmid u$ or
$d_2 \nmid u$.

\medskip
(ii) Using (\ref{ss:zetafunctions}) one has
$$
Z(f,\omega;s) = \frac 1{\nu + sN} \left( -2+ \frac {d_1}{i_1} +
\frac {d_2}{i_2} + \frac 1{k_1+s} + \frac 1{k_2+s} \right),
$$
 where (see (\ref{eq:nu}))
\begin{equation}\label{eq:nu-intro}
\nu = d_1d_2 (k_1+k_2-2) + d_2i_1+d_1i_2.
\end{equation}

We now investigate the candidate pole $s_0:= -\nu/N$ of the
zeta function. If  $s_0$ is a  pole of order one,
%its residue is (up to a factor $N$) equal to
%$$
%\calR:=-2+ \frac {d_1}{i_1} + \frac {d_2}{i_2} + \frac
%1{k_1-\nu/N} + \frac 1{k_2-\nu/N}.
%$$
one easily verifies that its residue $\calR$ is not identically
zero as function in the four variables $i_1,i_2,k_1,k_2$. Hence
$s_0$ is a pole of $Z(s)$ as soon as the algebraic equation
$\calR=0$ is not satisfied. It is also straightforward to compute
that $\calR$ is identically zero in $k_1$ and $k_2$ if $i_1=d_1$
and $i_2=d_2$, and that generally $\calR$ is not identically zero
otherwise.

\medskip
(iii) With respect to Goal (1),  if we wish to put only \lq
necessary\rq\ restrictions to realize it, the following is
a very natural choice for allowed $W$. Note first
that
 $d_\ell \mid \nu$ if and only if $d_\ell \mid i_\ell$ (by (\ref{eq:nu-intro})). We call $W$ {\it
allowed} if the following condition on $i_1$ and $i_2$ is
satisfied: \emph{if $d_1 | i_1$ and $d_2 | i_2$, then $i_1=d_1$
and $i_2=d_2$}. (There are no conditions on $k_1$ and $k_2$.)
Therefore, for $W$ allowed, if $\exp(2\pi is_0)$ is not a root of $\Lambda$,
then $d_1\mid i_1$ and $d_1\mid i_2$, hence $\calR=0$.

Moreover with this definition the divisor $W=0$, corresponding to
$i_1=i_2=k_1=k_2=1$, is clearly allowed (Goal (2)), and
Goal (3) is also satisfied. Indeed, fix a root
$\exp(2\pi i \frac uN)$ of $\Lambda(t)$. Since the numbers
$d_1$ and $d_2$ are coprime, there exist integers
$k_1,k_2,i_1,i_2$ (all positive if we desire so) such that $\nu$
in (\ref{eq:nu-intro}) satisfies $\nu \equiv u \mod N$. The
restrictions on the given $u$ imply that $d_1 \nmid i_1$ or $d_2
\nmid i_2$. Hence the constructed $W$ is allowed.

\smallskip

One can carry out without too much effort a similar analysis for
any plane curve germ $f$, or, more generally, for
any function $f$ on an IHS
germ for which the splice diagram is star-shaped, identifying allowed forms
$\omega$/divisors $W$ satisfying our Goal \ref{goal:in},
see (\ref{def:alloweddiagram}).
For arbitrary $f$ the situation is at first sight combinatorially
hopeless. Nevertheless, for them we use the concept of
{\it splicing} of a general splice diagram into star-shaped
building blocks, and we ask that the \lq restriction\rq\ of the
desired $W$ to any such star-shaped building block satisfies the
\lq natural\rq\ (already identified) conditions of allowedness.

Though conceptually appealing, it is not clear from the start that
such a \lq local\rq\ definition of allowedness will do the job, a
priori it is even not obvious that allowed forms/divisors exist on
arbitrary diagrams. It will turn out that there are plenty of
them, and at the end the proof of the first goal will be
(combinatorially) quite conceptual.

We still want to mention one important point regarding the introduction of Weil
divisors discussed in (\ref{bek:INTR1}).    In the singular setting,
in order to deal with our goals concerning the zeta functions associated to a given \emph{function} $f$,
by our inductive splicing strategy, we consider the restriction of the divisor of $f$ to the
star-shaped building blocks; the point is that usually the analytic realization of them
is a difficult issue deviating from the original main objective. This shows that
% we need to consider zeta functions associated to \emph{Weil divisors} $F$.
%(The analytic realization of the restrictions is not guaranteed; this makes
the introduction of Weil divisors is even necessary.

\section{Splice diagrams associated with normal surface singularities}\labelpar{s:1}
\subsection{Splice diagrams of surface--germs}\labelpar{ss:1}
Let $(X,0)$ be the germ of a complex normal surface singularity,
and $M$ be its link. It is well--known that $M$ is an oriented
plumbed 3--manifold, and any dual resolution graph might serve as
a plumbing graph for $M$. Let $\pi:\tilde{X}\to X$ be a good
resolution, that is, the exceptional divisor $E:=\pi^{-1}(0)$ is a
normal crossings divisor on the smooth complex surface $\tilde{X}$.
The topology of $\pi$ is codified in the dual graph $G=\Gx$
associated with the irreducible components $\{E_i\}_i$ of $E$:
each $E_i$ determines a vertex of $G$ with genus decoration
$[g(E_i)]$ and self--intersection $(E_i,E_i)$, while the edges of
$G$ correspond to intersection points $E_i\cap E_j$, cf.
\cite{La}. Since $X$ is normal, $G$ is connected. Let $I(G)$ be
the negative definite intersection form $(E_i,E_j)_{i,j}$. By
plumbing construction one recovers from $G$ both $M$ and (the
$C^\infty$--type of) $(\tilde{X},E)$.

We recall that $M$ is a rational homology sphere if and only if
$G$ is a tree and $g(E_i)=0$ for all $i$. Moreover, $M$ is an
integral homology sphere if additionally $\det(-I(G))=1$.
%In this last case we say that $(X,0)$ is an IHS germ.

\medskip
If $M$ is an {\it integral homology sphere}, then $G$ can
equivalently be  codified in a more condensed form via its splice
diagram $\Gamma=\Gax$, cf. \cite{EN}. The diagram $\Gamma$ is the
tree obtained from $G$ by replacing each maximal string of $G$  by
a single edge. Hence, $\Gamma$ is homeomorphic to $G$, but it has
no vertices of valency 2. Its vertices are either {\it nodes} (of
valency $\geq 3$) or {\it ends/boundary vertices} (of valency 1);
they correspond to the {\it nodes/rupture vertices} and {\it
boundary vertices}  of $G$ with valencies $\geq 3$ and $1$, respectively.
The decorations of $\Gamma$ are as follows. At each node $v$ of
$\Gamma$ one inserts a weight $d_{ve}$ on each incident edge $e$.
Let $G_{ve}$ be the connected component of $G\setminus \{v\}$  `in
the direction of $e$', then $d_{ve}:=\det(-I(G_{ve}))$. It is
proved in \cite[Ch. V]{EN} that the decorated graphs $G$ and
$\Gamma$ determine each other.

The decorations $\{d_{ve}\}_{v,e}$ of $\Gamma$ satisfy the next
compatibility conditions:
\begin{equation}\label{eq:compcon}
\left\{\begin{array}{l}
(a) \ \ \ \mbox{$d_{ve}\geq 1$},\\
(b) \ \ \ \mbox{$\{d_{ve}\}_e$ are pairwise coprime integers for any fixed node $v$},\\
(c) \ \ \ \mbox{any   `edge determinant' $q_e$ is positive.}
\end{array}\right.\end{equation}
Part (c) means the following:
for any fixed edge $e$ with end--nodes $v$ and $w$,
let the decorations at $v$  be  $d_{ve}$ and $\{d_{ve_i}\}_i$, and
similarly  $d_{we}$ and $\{d_{we'_j}\}_j$ at $w$. Then
(see \cite[\S 24]{EN}):
\begin{equation}\label{eq:edgedet}
q_e := d_{ve}d_{we}-\prod_id_{ve_i}\prod_jd_{we'_j}>0.
\end{equation}
If $\pi$ is the {\it minimal} good resolution, then $\Gamma$ also
is minimal, in the sense that all the decorations $d_{ve}$ are
strictly greater than 1, provided that  $e$ connects $v$ with a
{\it boundary vertex}. If $G$, or $\Gamma$, is not minimal, then
such a restriction does not hold. By `splice calculus', one can
delete such an edge with decoration 1 and the supported boundary
vertex, getting a new equivalent diagram. All these equivalent
diagrams represent the same 3--manifold $M$. The nodes (and the
corresponding star--shaped subgraphs around them) correspond
exactly to the (minimal or non--minimal) Jaco--Shalen--Johannson
decomposition of $M$ (depending on the minimality of $\Gamma$),
each star--shaped subgraph describing a Seifert piece.

%We will use the notation $\calV$ for the vertices, $\calN$ for
%nodes, $\calB$ for boundary vertices and $\calE$ for edges of
%$\Gamma$. Moreover, we call {\it special edges} those connecting
%two nodes, denoted by $\calE^s$.

\begin{bekezdes}\label{be:11}
The diagram $\Gamma$ (or $G$) contains the same amount of
information as the link $M$, hence working with it  we  disregard
completely the analytic structure of $(X,0)$. In the sequel we
regard  $\Gamma$ as an abstract splice diagram which
satisfies (\ref{eq:compcon}). In fact, any such diagram can be
realized by some singularity link, but the corresponding analytic
structure(s) can be hard to determine and are irrelevant from the
point of view of many invariants.

In this correspondence, in fact, there is an `easy case', namely
when $(X,0)$ is smooth and $M$ is the 3--sphere $S^3$: this is
happening  if and only if in any (maybe non--minimal) splice
diagram which represents them the following fact holds: for any
node $v$ at most two of the integers $\{d_{ve}\}_e$ can be
strictly greater than 1 (that is,  any Seifert piece is an $S^3$).
\end{bekezdes}
\begin{bekezdes}\labelpar{be:semi}{\bf Semigroup condition for $\Gamma$.}
It is convenient to introduce some other combinatorial invariants
of a splice diagram $\Gamma$ as well. If $v$ and $w$ are two
vertices of $\Gamma$, we set $\ell_{vw}$ for the product of the
edge weights that are adjacent to, but not on, the path  from $v$
to $w$. Furthermore, for each node $v$, let $d_v$ be the product
of edge weights adjacent to $v$.

For any node $v$ and adjacent edge $e$, let $\Gamma_{ve}$ be the
connected component of $\Gamma\setminus \{v\}$ in the direction of
$e$. We say, following \cite{NW1}, that the node $v$ and adjacent
edge $e$ satisfy the {\it semigroup condition} if
\begin{equation*}\label{eq:semcond}
d_v \ \mbox{is in the semigroup generated by the $\ell_{vw}$,
where $w$ is a boundary vertex of $\Gamma$ in $\Gamma_{ve}$}.
\end{equation*}
By definition, a {\it splice diagram $\Gamma$ satisfies the semigroup
condition}, if all pairs $(v,e)$ as above satisfy the semigroup condition.
\end{bekezdes}

\subsection{Splice diagrams of function--germs/divisors}\labelpar{ss:2} Assume
that $f:(X,0)\to (\C,0)$ is the germ of an analytic function on an
IHS germ $(X,0)$. If $\pi: \tilde{X}\to X$ is an embedded
resolution of the pair $(X, f^{-1}(0))$, then the topology of $f$
is described by the embedded resolution graph of $f$. This
consists of the dual graph $\Gxf$ of the exceptional divisors
decorated by the self--intersections, and supplemented by the
following data: each irreducible component of the strict transform
intersecting an irreducible exceptional divisor $E_i$ is codified
by an arrowhead supported by that vertex of $G$ which corresponds
to $E_i$. Additionally, each vertex and arrowhead inherits a
multiplicity decoration, the vanishing order of $f\circ \pi $
along the corresponding irreducible divisor.

Clearly, the divisor  $\div(f\circ \pi)$ on $\tilde{X}$ is a
principal divisor, hence $(\div(f\circ \pi),E_i)=0$ for any $i$.
This (and the fact that $\det(I(G))\not=0$) shows that all the
multiplicities of the strict transforms of $f=0$ determine
$\div(f\circ \pi)$ completely (compare with (\ref{eq:mult})). In
fact, this property identifies  $\div(f\circ \pi)$ as the pullback
of the divisor $f=0$ on $X$.

More generally, a divisor $F$ on $\tilde{X}$ supported on $E$ and
on some noncompact transversal slices of $E$ is called {\it
$P$--divisor} of $\tilde{X}$ if $(F,E_i)=0$ for any $i$. If we
start with an arbitrary Weil divisor $F'$ on $X$ and $\pi$ is an
embedded resolution of the pair $(X,F')$,  then there is an unique
$P$--divisor $F$ on $\tilde{X}$ whose arrow-multiplicities agree
with the multiplicities of the components of $F'$. This $F$ will
be called the pullback of $F'$. On the other hand, if $F$ is a
$P$-divisor on $\tilde{X}$, projecting down its noncompact
components (by keeping their multiplicities)
 we get a Weil divisor
$F'$ on $X$ such that $F$ is the pullback of $F'$.  Hence, $F$ and
$F'$  determine each other (thus we will sometimes write $F$ for
both of them).

A $P$--divisor is  codified in the graph $G$ similarly as the
principal divisors via its arrowheads and multiplicity system, and
it is  uniquely determined by the arrowheads and their
multiplicities. This pair is denoted by $\GxF$.

In all of our topological--combinatorial discussions, we regard
$\GxF$ as a {\it combinatorial} object, a plumbing representation
of a pair $(M,M\cap F')$.  We do not ask the analytic realization
of $F'$ as a principal divisor (and even if we do in some
discussions, we always consider analytic realizations and not
algebraic ones).

\begin{bekezdes}
The splice diagram $\Gamma=\Gaxf$ associated with $(X,f)$
(or more generally, $\Gamma_\pi(X,F)$ associated with $G_\pi(X,F)$)
is constructed similarly as above, but now the nodes are those
vertices which have valency $\geq 3$ including the edges supporting arrowheads.
Moreover,  the new splice diagram contains arrowheads and multiplicity decorations as well (see \cite{EN}).

For the arrowheads of the {\it minimal} splice diagram we use the
following principle. If an arrowhead $a$ of $G$ is supported by a
vertex $v$ of $\GxF$ with valency $\geq 3$ (including the edges
supporting arrowheads) then $v$ becomes a node of $\Gamma$ and $a$
becomes an arrowhead of $\Gamma$ supported by $v$. The weight of
the edge at $v$ supporting such an arrowhead is either 1 or is
missing. Next, assume that the arrowhead $a$ of $\GxF$ is
supported by a vertex $v$ of $G$ of valency $ 2$. This means that
$v$ is a boundary vertex of $\Gx$ with an arrowhead. If we forget
about the arrowheads and we determine $\Gamma$ from $\Gx$, then
$v$ becomes a boundary vertex of $\Gax$. Then, reconsidering the
arrowheads, this boundary vertex in $\GaxF$ is replaced by an
arrowhead.
%(Note that if there is an arrowhead of $\GxF$ supported
%by a vertex with valency 1, then this vertex is the only vertex of
%$G$, and it must have self--intersection $-1$. If the arrowhead is
%its only arrowhead then
%the splice diagram has a unique vertex, with one arrowhead whose weight is 1, or is missing. In fact then
%$(X,F')$ determines a smooth plane curve germ.)
Summed up: in minimal diagrams, all the arrowheads of $\GaxF$ are
supported by nodes, nevertheless they have two different
interpretations: if the weight of an edge (at the node $v$)
supporting an arrowhead is $\geq 2$ then the corresponding strict
transform intersects the corresponding boundary curve of $G$,
while if the weight is 1 (or it is missing) then the strict
transform intersects that exceptional component which corresponds
to the node. (In the language of knots: weight $\geq2$ gives a
special Seifert fiber, while weight 1 a generic Seifert fiber in
the corresponding Seifert piece.)

Nevertheless, sometimes we also allow non--minimal representations
(which appear naturally when we splice the diagrams). Namely, the
following calculus provides equivalent diagrams:

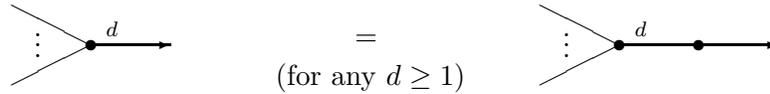
\begin{picture}(300,55)(-50,-10)
\put(240,20){\circle*{4}} \put(270,20){\circle*{4}}
\put(240,20){\vector(1,0){60}} \put(240,20){\line(-2,-1){30}}
\put(240,20){\line(-2, 1){30}}
\put(220,23){\makebox(0,0){$\vdots$}}
\put(248,26){\makebox(0,0){\tiny{$d$}}}

\put(140,20){\makebox{$=$}} \put(110,5){\makebox{(for any $d\geq
1)$}} \put(40,20){\circle*{4}} \put(40,20){\vector(1,0){30}}
\put(40,20){\line(-2,-1){30}} \put(40,20){\line(-2, 1){30}}
\put(20,23){\makebox(0,0){$\vdots$}}
\put(48,26){\makebox(0,0){\tiny{$d$}}}
\end{picture}

Also, $\GaxF$ inherits the multiplicity of each arrowhead and node
from $\GxF$ (with the same geometric interpretation). In the case of principal divisors,
$f$ defines an isolated singularity if and only if all
arrowhead--multiplicities are 1.

We will use the notation $\calV$ for the vertices, $\calN$ for the
nodes, $\calB$ for the boundary vertices and $\calE$ for the edges of
$\GaxF$. Moreover, we call {\it special edges} those connecting
two nodes, denoted by $\calE^s$. The arrowheads will be denoted by
$\calA_F$, the multiplicities by  $(N_w)$, $w\in
\calV\cup\calA_F$.

Again, we can regard  $\GaxF$ as an abstract graph, we do not ask
about  the analytic realization of the pair $(X,F)$ (although, if
the graph satisfies (\ref{eq:compcon}) and $N_a> 0$ for all
$a\in\calA_F$, then some analytic realization exists, cf. \cite[\S\,24]{EN}.)
\end{bekezdes}
\begin{bekezdes}\labelpar{2.2.2}
Recall that for any $P$--divisor the multiplicities of the arrowheads (and the
combinatorics of the splice diagram without the other
multiplicities) determine all the multiplicities of the vertices.
(In the case of the graph $G$, this is done via $I(G)^{-1}$.) In the language of
$\Gamma$ one has the following.
Let $v$ be a fixed vertex, and let $a$ be an arrowhead.
Then define $\ell_{va}$ as the product of the edge weights that
are adjacent to, but not on, the path  from $v$ to $a$. Then the
multiplicity $N_v$ of any vertex $v$ is given by, cf. \cite[\S 10]{EN},
\begin{equation}\label{eq:mult}
N_v=\sum_{a\in\calA_F} N_a\ell_{va}.
\end{equation}
In particular, if
 $F'=\sum_a N_aF'_a$ is a Weil divisor on $X$
with the $F'_a$ irreducible, and $\{F_a\}_{a\in\calA_F}$ are  the
strict transforms and $F_v=E_v$ the exceptional curves, then the
pullback $F=\pi^*(F')$  of $F'$ is represented in the splice
diagram by $\sum_{v\in \calA_F\cup\calV}N_vF_v$.
\end{bekezdes}
\begin{bekezdes}\labelpar{be:decorat} The splice diagram $\Gamma_\pi(X,F)$ also satisfies the {\bf compatibility
conditions} (\ref{eq:compcon}).  If one deletes all the arrowheads
and decorations of $\Gamma_\pi(X,F)$ associated with $F$ we
recover a possible (maybe non--minimal) splice diagram of $X$.
Nevertheless, by this simplification, some of the nodes might
disappear.
%hence $\Gamma_\pi(X,F)$ might contain more decorations.
\end{bekezdes}
\begin{bekezdes}\labelpar{be:semigro}{\bf Semigroup condition for $\Gamma_\pi(X,F)$.}
In the presence of a divisor, the semigroup condition
(\ref{be:semi}) will be modified as follows. We say that   $\GaxF$
satisfies  the semigroup condition if all  pairs $(v,e)$ satisfy
(\ref{be:semi}), provided that $v$ is a node
 with adjacent edge $e$, and the connected part of
$\GaxF \setminus \{v\}$ in the direction of $e$ contains {\it no
arrowheads}.

\begin{remark}\labelpar{re:1234}
(1) The splice diagram associated with the minimal embedded
resolution  of a plane curve singularity always satisfies the
semigroup condition. Indeed, there is only one sub--diagram (see
below), where the condition is not satisfied trivially: the nodes
of this sub--diagram are not sitting on geodesic paths connecting
two arrowheads of the diagram.

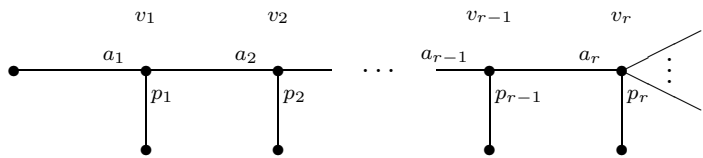
\begin{picture}(400,75)(0,0)
%\put(20,40){\circle*{4}}
\put(70,40){\circle*{4}} \put(120,40){\circle*{4}}
\put(170,40){\circle*{4}} \put(250,40){\circle*{4}}
\put(300,40){\circle*{4}}
%\put(70,10){\circle*{4}}
 \put(120,10){\circle*{4}}
\put(170,10){\circle*{4}} \put(250,10){\circle*{4}}
\put(300,10){\circle*{4}}

\put(70,40){\line(1,0){120}}\put(230,40){\line(1,0){70}}
\put(300,40){\line(2,1){30}}\put(300,40){\line(2,-1){30}}
%\put(70,40){\line(0,-1){30}}
\put(170,40){\line(0,-1){30}} \put(120,40){\line(0,-1){30}}
\put(250,40){\line(0,-1){30}} \put(300,40){\line(0,-1){30}}

%\put(58,45){\makebox(0,0){$a_1$}}
\put(108,45){\makebox(0,0){\tiny{$a_1$}}}
\put(158,45){\makebox(0,0){\tiny{$a_2$}}}
\put(233,45){\makebox(0,0){\tiny{$a_{r-1}$}}}
\put(288,45){\makebox(0,0){\tiny{$a_r$}}}

%\put(72,30){\makebox(0,0)[l]{$p_1$}}
\put(122,30){\makebox(0,0)[l]{\tiny{$p_1$}}}
\put(172,30){\makebox(0,0)[l]{\tiny{$p_2$}}}
\put(252,30){\makebox(0,0)[l]{\tiny{$p_{r-1}$}}}
\put(302,30){\makebox(0,0)[l]{\tiny{$p_r$}}}

\put(208,40){\makebox(0,0){$\ldots$}}
\put(318,43){\makebox(0,0){$\vdots$}}
%\put(70,55){\makebox(0,0){$v_1$}}
\put(120,60){\makebox(0,0){\tiny{$v_1$}}}
\put(170,60){\makebox(0,0){\tiny{$v_2$}}}
\put(250,60){\makebox(0,0){\tiny{$v_{r-1}$}}}
\put(300,60){\makebox(0,0){\tiny{$v_r$}}}

\end{picture}

For this part,  the semigroup condition follows from the
positivity of the edge determinants and the fact (used literately
several times)  that for coprime positive integers $a$ and $p$,
each integer larger than $ap$ belongs to the semigroup generated
by $a$ and $p$. (For a more general argument, see (\ref{ss:SC}).)

(2) If $F=0$, we recover the semigroup condition of
(\ref{be:semi}). On the other hand, the semigroup condition of
$\GaxF$, and of the diagram  obtained from $\GaxF$  by deleting the information
regarding $F$, are independent. This fact is exemplified next.
Note that in the semigroup condition the position of the arrowheads is important,
while the multiplicity system of $F$ is irrelevant,
hence we will omit the multiplicities  from the next diagrams.

(3) It is possible that $\Gax$ does not satisfy the
semigroup condition, but  $\GaxF$ for some $F$ does. Take for
instance any $P$-divisor $F$ with enough arrowheads such that
$\Gamma\setminus \{v\}$ contains an arrowhead in the direction of
$e$ for each pair $(v,e)$.

(4) On the other hand, it is possible that $\Gax$ satisfies the
semigroup condition, but $\GaxF$ does not, due to the appearance
of the new nodes which support the arrowheads. Take for example
the following  resolution graph and the corresponding splice
diagram.

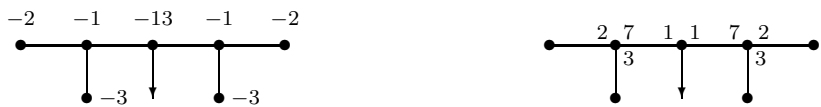
\begin{picture}(400,55)(80,-5)

\put(125,25){\circle*{4}} \put(150,25){\circle*{4}}
\put(175,25){\circle*{4}} \put(200,25){\circle*{4}}
\put(225,25){\circle*{4}} \put(150,5){\circle*{4}}
\put(200,5){\circle*{4}} \put(125,25){\line(1,0){100}}
\put(150,25){\line(0,-1){20}} \put(175,25){\vector(0,-1){20}}
%\put(370,5){\circle*{4}}
\put(200,25){\line(0,-1){20}} \put(125,35){\makebox(0,0){\tiny{$-2$}}}
\put(150,35){\makebox(0,0){\tiny{$-1$}}}
\put(175,35){\makebox(0,0){\tiny{$-13$}}}%\put(175,-3){\makebox(0,0){\tiny{$(N)$}}}
\put(200,35){\makebox(0,0){\tiny{$-1$}}}
\put(225,35){\makebox(0,0){\tiny{$-2$}}} \put(160,5){\makebox(0,0){\tiny{$-3$}}}
\put(210,5){\makebox(0,0){\tiny{$-3$}}}%\put(322,5){\makebox(0,0){$-7$}}

\put(325,25){\circle*{4}} \put(350,25){\circle*{4}}
\put(375,25){\circle*{4}} \put(400,25){\circle*{4}}
\put(425,25){\circle*{4}} \put(350,5){\circle*{4}}
\put(400,5){\circle*{4}} \put(325,25){\line(1,0){100}}
\put(350,25){\line(0,-1){20}} \put(375,25){\vector(0,-1){20}}
%\put(370,5){\circle*{4}}
\put(400,25){\line(0,-1){20}} \put(345,30){\makebox(0,0){\tiny{$2$}}}
\put(355,30){\makebox(0,0){\tiny{$7$}}}\put(370,30){\makebox(0,0){\tiny{$1$}}}
\put(395,30){\makebox(0,0){\tiny{$7$}}}\put(380,30){\makebox(0,0){\tiny{$1$}}}
%\put(375,35){\makebox(0,0){\tiny{$-13$}}}
%\put(375,-3){\makebox(0,0){\tiny{$(N)$}}}
\put(406,30){\makebox(0,0){\tiny{$2$}}}
%\put(425,35){\makebox(0,0){\tiny{$-2$}}}
\put(355,20){\makebox(0,0){\tiny{$3$}}}
\put(405,20){\makebox(0,0){\tiny{$3$}}}%\put(322,5){\makebox(0,0){$-7$}}
\end{picture}

\vspace{3mm}

\noindent Then $\GaxF$ does not satisfy the semigroup condition at
the central node, although if we delete the arrowhead then
$\Gamma_\pi(X)$ does since that node disappears.
%(For analytic realization, see e.g. \cite[(5.3)]{SI}.)
\end{remark}
\end{bekezdes}

\subsection{Splice diagrams and differential forms}\labelpar{ss:3}
We still consider an IHS singularity $(X,0)$, and a function germ
$f$ or a nonzero effective Weil divisor $F'$ on $X$. We fix an embedded
resolution $\pi:\tilde{X}\to X$ of $(X,f^{-1}(0))$ or $(X,F')$ as
in (\ref{ss:2}).

Next, we also wish to incorporate in the picture  a differential
(meromorphic) 2--form $\omega$ or a Weil divisor $W$. The basic
models for us are the following situations.

\vspace{2mm}

\noindent $\bullet$ \ Assume that $(X,0)$ is smooth, hence $f$
determines a plane curve singularity. Then classically one
considers $\omega_0=dx\wedge dy$ (for some local coordinates
$(x,y)$ of $(X,0)$), or,  more generally, $\omega =g\omega_0$ for
some local analytic germ $g$ on $(X,0)$. In this case the pullback
$\tilde{\omega} :=\pi^*(\omega)$ is clearly holomorphic on
$\tilde{X}$.

\vspace{1mm}

\noindent $\bullet$ \ Generalizing to the singular setting, we
assume that $(X,0)$ is Gorenstein and that $\omega_0$ is a nowhere
vanishing holomorphic 2--form on $X\setminus \{0\}$. Then, for any
holomorphic germ $g$ on $X$, the pullback
$\tilde{\omega}:=\pi^*(g\omega_0)$ has a meromorphic extension
over the exceptional curve. If $(X,0)$ is rational then it is
holomorphic, otherwise it might have poles.

\vspace{2mm}

We will enrich the diagrams with the vanishing orders of
$\div(\tilde{\omega})$ along the corresponding irreducible
divisors. Again, this only depends on $\div(\tilde{\omega})$, not
on $\tilde{\omega}$ itself; so we rather incorporate from the
start a Weil divisor $W'$ on $X$.
In the special cases above $\div(\omega)=\div(g)=W'$ and
$\div(\tilde{\omega}) = K_\pi + \pi^*W'$.
 Also, it will be natural in our context to restrict the
possible support of $W'$; we assume that the following facts hold:

\begin{enumerate}
\item
$\pi$ is also an embedded resolution of $\supp(F') \cup
\supp(W')$;
\item
if  $W':=\sum_{k}(i_k-1)W'_k$ is the irreducible decomposition of
$W'$, then the strict transform of each $W'_k$ either is identical
with the strict transform $F_a$ of one of the $F'_a$, or, it
intersects a boundary component of $\GxF$. Moreover, each boundary
component can intersect at most one of the strict transforms of
the $W'_k$.
\end{enumerate}
We denote the $P$--divisor $\pi^*(W')$ by $W$. In our diagrams,
the strict transform $W_k$ of a $W'_k$ that intersects $E_i$ will
be denoted by a dashed arrowhead, attached to the vertex
corresponding to $E_i$. If $W_k$ agrees with one of the $F_a$,
then the dashed arrowhead doubles the associated ordinary
arrowhead, while if a boundary vertex does not support any
ordinary arrowhead but supports a $W_k$, then its dashed arrow is
attached to this vertex. We will denote by $\calA_W$ the dashed
arrowheads of $W$. The  \lq double\rq\ arrowheads are given by
$\calA_F \cap \calA_W$.
%$\calA$ will stay for $\CalA$.

Again, we will identify the following diagrams.

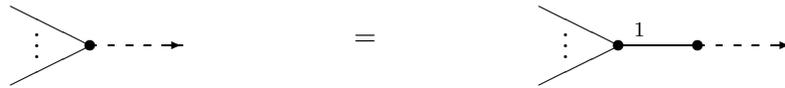
\begin{picture}(300,60)(-50,-10)
\put(240,20){\circle*{4}} \put(270,20){\circle*{4}}
\put(240,20){\line(1,0){30}} \put(240,20){\line(-2,-1){30}}
\put(240,20){\line(-2, 1){30}}
\put(220,23){\makebox(0,0){$\vdots$}}
\put(248,26){\makebox(0,0){\tiny{$1$}}}

\dashline[3]{3}(270,20)(300,20)\put(300,20){\vector(1,0){5}}
\dashline[3]{3}(40,20)(70,20)\put(70,20){\vector(1,0){5}}
\put(140,20){\makebox{$=$}}
%\put(110,5){\makebox{(for any $d\geq 1)$}}
\put(40,20){\circle*{4}} %\put(40,20){\vector(1,0){30}}
\put(40,20){\line(-2,-1){30}} \put(40,20){\line(-2, 1){30}}
\put(20,23){\makebox(0,0){$\vdots$}}
%\put(50,26){\makebox(0,0){$1$}}
\end{picture}

Finally, we have to add to the decorations the multiplicities of
the components of $K_\pi + W$. By technical (and traditional)
reasons, the multiplicity of a  dashed arrowhead  is denoted as
above by $i_a-1$, while the multiplicity of a vertex $v$ by
$\nu_v-1$. Recall that, in general, we do not impose for the
integers $i_a$ and $\nu_v$ to be positive. We denote this enriched
diagram as $\GaxFW$.

Note that in the three different levels $\Gax$, $\GaxF$ and
$\GaxFW$, the valencies of a fixed vertex $v$ of $\Gamma$ are not
the same; these three valencies will be denoted by $\delta_v$,
$\delta_v'$ and $\delta_v''$ respectively.

\begin{bekezdes} \label{be:KK}
In order to identify the coefficients $\nu_v-1$ of $K_\pi+W$,
we only must describe the coefficients of the $E_v$ in $K_\pi$,
since those of $W$ are as in (\ref{2.2.2}).

The divisor $K=K_\pi$ is determined in the  graph
$\Gx$ by the {\it adjunction relations}
\begin{equation}\label{eq:adj}
(K+E,E_i)=\delta_i-2,
\end{equation}
where here $\delta_i$ is the  valency of the vertex $i$ in $\Gx$.
Let $L$ be the lattice $H_2(\tilde{X},\Z)$ with the intersection
form $I(G)$. Then for each vertex $i$ of $G$ one can define
$E^*_i\in L$ with $(E^*_i,E_j)=-\delta_{ij}$ (the {\it negative }
of the Kronecker delta), i.e. the sign--modified dual basis of
$L$. (Since $\det(-I(G))=1$, they are well--defined.) Therefore,
(\ref{eq:adj}) reads as
%\begin{equation}\labelpar{eq:adj2}
$K+E=\sum_{i}\, (2-\delta_i)E^*_i$.
%\end{equation}
Since all the valency 2 vertices of $G$ are irrelevant,  the
relation descends naturally to the level of $\Gamma=\Gax$ (this means
that the multiplicities of $K+E$ along the vertices of $\Gamma$  are those
given on the right):
\begin{equation}\label{eq:adj2}
K+E=\sum_{w\in\calV}\, (2-\delta_w)E^*_w.
\end{equation}
Each $E^*_w$, considered as divisor supported on $E$, together
with a non--compact  irreducible divisor
%$\tilde{F}_w$
intersecting $E_w$ in a smooth point of $E$, form a $P$--divisor.
%the pullback by $\pi$ of $\pi(\tilde{F}_w)$,
Hence, we can use (\ref{eq:mult}) and we obtain that the multiplicity of
$K+E$ along any {\it node} $v$ of $\Gamma$ (or in the presence of a
divisor $F$, along any node of
$\Gamma_\pi(X,F)$) is
\begin{equation}\label{eq:can}
\sum_{w\in\calV}\, (2-\delta_w)\ell_{vw}.
\end{equation}
In the presence of the divisor $W'=\sum_{a\in \calA_W} (i_a-1)W'_a$
on $X$, again (\ref{eq:mult}) then yields
\begin{equation}\label{eq:nu}
\nu_v=\sum_{w\in\calV}\, (2-\delta_w)\ell_{vw} + \sum_{a\in
\calA_W} (i_a-1)\ell_{va}.
\end{equation}
 We warn the reader about the following
fact. The sum (\ref{eq:can}) is associated with  $\Gamma_\pi(X)$.
Therefore, in the applications later, even if we start with some
$\GaxF$, those ordinary arrowheads associated with $F$ with weight
of their supporting edge greater than 1 should be replaced by
boundary vertices (hence the summation index $\calV$ should be the
set of vertices of $\Gax$). This is valid for the left sum of
(\ref{eq:nu}) too.
\end{bekezdes}
\begin{bekezdes}\labelpar{be:notation}{\bf Notation.}
In all the next combinatorial formulas associated with a
splice diagram   $\Gamma_\pi(X,F,W)$, the resolution $\pi$
or the geometric source of the diagram is irrelevant.
%Moreover, the information in a Weil divisor $F$ of $X$ and
%in its pullback $P$--divisor $\pi^*F$ is the same, hence we will disregard the
%divisor on $X$ and we just write $F$ for its pullback.

In particular, in the sequel $\Gamma(F,W) $ means a splice diagram with two
$P$--divisors $F$ and $W$. If $W=0$ then we just write $\Gamma(F)$.
The divisor $W$ will always be linked with $K$ (determined in $\Gamma$ by (\ref{be:KK}))
in the expression $K+W$.
\end{bekezdes}

\subsection{Topological zeta functions of diagrams}\labelpar{ss:zetafunctions}
In \cite{Ve2} the second author derived a formula for the
topological zeta function of a plane curve germ $f$ in terms of
its so-called relative log canonical model; this can be
interpreted as being in terms of $\Gaxf$ (where $\pi$ is minimal).
The same proof yields a similar formula in our more general
context of (\ref{ss:3}), with the divisors $F'$ and $W'$ on the germ
$(X,0)$. In fact one associates in this way a zeta function to a
decorated diagram $\Gamma(F,W)$, cf. (\ref{be:notation}). We want
to formalize this, since our technique to study the topological
zeta function is in fact a \lq splicing formula\rq\ for zeta
functions of decorated splice diagrams.

Therefore, let us consider a decorated splice diagram
$\Gamma(F,W)$ as in (\ref{be:notation}). Moreover, it is
convenient to associate multiplicities $N_a$ and $i_a-1$ to {\it
all} $a\in \CalA$, that is, we put $N_a=0$ for $a\in\calA_W
\setminus \calA_F$ and $i_a=1$ for $a\in\calA_F \setminus
\calA_W$.

%\begin{definition}\label{I:zetadiagram} Let $\Gamma=\Gamma(F,W)$ be such a diagram.
%We require for each $a\in\CalA$ that $(N_a,i_a)\neq(0,0)$.
%For a node $v$, let $(\CalA)_v$ and $\calB_v$ be the
%(ordinary and/or dashed) arrowheads and boundary vertices,
%respectively, attached at $v$. Denote the weight at $v$ on the
%incident edge in the direction of such an arrowhead $a$ or
%boundary vertex $w$ by $d_{va}$ and $d_{vw}$, respectively. For a
%special edge $e$, let $v$ and $w$ denote its end vertices, and $q_e$ the edge determinant
%(\ref{eq:edgedet}). Then the
%zeta function $Z(\Gamma)$ of the diagram $\Gamma$ is
%\begin{equation*}\begin{split}
%Z(\Gamma)=Z(\Gamma;s):= \sum_{v\in\calN} \frac1{\nu_v+sN_v}\left(
%2-\delta_v'' + \sum_{w\in \calB_v} d_{vw} +\sum_{a \in
%(\CalA)_v} \frac{d_{va}}{i_a+sN_a}\right)\\ +
%\sum_{e\in\calE^s} \frac{q_e}{(\nu_v+sN_v)(\nu_w+sN_w)}.
%\end{split}\end{equation*}
%\end{definition}

%CAN WE REPLACE THE ABOVE DEFINITION 2.4.1 BY THE NEXT ONE??
%(if yes, cut the old one and delete 2 from the label of the second one)\\
%\marginpar{?????????????????}

\begin{definition}\label{I:zetadiagram} Let $\Gamma=\Gamma(F,W)$ be such a diagram.
We require for each $a\in\CalA$ that $(N_a,i_a)\neq(0,0)$. For a
node $v$, let $(\CalA)_v$ and $\calB_v$ be the (ordinary and/or
dashed) arrowheads and boundary vertices, respectively, attached
at $v$. Denote the weight at $v$ on the incident edge in the
direction of such an arrowhead $a$ or boundary vertex $w$ by
$d_{va}$ and $d_{vw}$, respectively. For any $w\in \calB_v$, let
$i_w-1$ be the decoration of the dashed arrowhead supported by
$w$; if such an arrowhead does not exist then set  $i_w=1$. For a
special edge $e$, let $v$ and $w$ denote its end vertices, and
$q_e$ the edge determinant (\ref{eq:edgedet}). Then the zeta
function $Z(\Gamma)$ of the diagram $\Gamma$ is
\begin{equation*}\begin{split}
Z(\Gamma)=Z(\Gamma;s):= \sum_{v\in\calN} \frac1{\nu_v+sN_v}\left(
2-\delta_v'' + \sum_{w\in \calB_v} \frac{d_{vw}}{i_w}
 +\sum_{a \in
(\CalA)_v} \frac{d_{va}}{i_a+sN_a}\right)\\ +
\sum_{e\in\calE^s} \frac{q_e}{(\nu_v+sN_v)(\nu_w+sN_w)}.
\end{split}\end{equation*}
\end{definition}

%\begin{definition}\label{I:zetadiagram3} Let $\Gamma=\Gamma(F,W)$ be such a diagram.
%We require for each $a\in\CalA$ that $(N_a,i_a)\neq(0,0)$.
%
%For a node $v$, let $\calL_v$ be the legs attached at $v$,
%abutting at an (ordinary and/or dashed) arrowhead or a boundary
%vertex. The multiplicity of the ordinary and dashed arrowhead at
%such a leg $\ell$ is $N_\ell$ and $i_\ell -1$, respectively. (When
%it does not occur, we have $N_\ell=0$ and $i_\ell=1$,
%respectively.) Denote the weight at $v$ on $\ell$ by $d_{v\ell}$.
%
%For a special edge $e$, let $v$ and $w$ denote its end vertices,
%and $q_e$ the edge determinant (\ref{eq:edgedet}). Then the zeta
%function $Z(\Gamma)$ of the diagram $\Gamma$ is
%\begin{equation*}
%Z(\Gamma)=Z(\Gamma;s):= \sum_{v\in\calN} \frac1{\nu_v+sN_v}\left(
%2-\delta_v'' +\sum_{\ell \in \calL_v}
%\frac{d_{v\ell}}{i_\ell+sN_\ell}\right) + \sum_{e\in\calE^s}
%\frac{q_e}{(\nu_v+sN_v)(\nu_w+sN_w)}.
%\end{equation*}
%\end{definition}

\section{Splicing the diagrams and their invariants}\labelpar{s:2}

\subsection{Splicing the diagrams}\labelpar{ss:splicingdiagrams} The main advantage of the splice diagrams is that they
describe in an ideal way the splice (non--minimal JSJ--)
decomposition of the 3--manifold $M$ into its Seifert pieces:
while doing this operation, the decorations follow rather simple
rules. The behavior of the multiplicities $\{N_w\}_w$ associated
with a principal divisor $f$  is classical, it was developed in \cite{EN}.

It is easy to see that any $P$--divisor follows the same formula.
On the other hand,  the rules for the
 numbers $\{\nu_v\}_v$  are slightly more involved, and we were not
able to find them in the literature (though, see  the `simpler'
situation considered in \cite{NV}).  In this subsection, we will
present these splice formulae.

Subsection (\ref{be:d}) treats the case $\Gamma=\Gax$,
(\ref{be:f}) the case $\Gaxf$ and its generalization $\Gamma(F)$,
while (\ref{be:fo}) the case $\Gamma(F,W)$.

\begin{bekezdes}\labelpar{be:d} {\bf Splicing $\Gamma$.}
First, recall that splicing along the edge $e$ with end--nodes
$v_L$ and $v_R$ (left/right) is the operation which replaces the
left diagram $\Gamma$ into the two diagrams $\Gamma_L$  and
$\Gamma_R$ (containing $v_L$ and $v_R$ respectively), with two new
end vertices $\bar{v}_L$ and $\bar{v}_R$. (All the other parts of
the diagrams are kept unmodified.) If either $d$ or $d'$ is 1,
then the corresponding leg in $\Gamma_L$ or $\Gamma_R$ can be
deleted in the minimal representation, but we prefer to keep it.
In this way, the valency of the vertices $v_L$ and $v_R$ stays
unmodified. Moreover, when we equip such a leg later with an
arrowhead, it cannot be deleted.
\end{bekezdes}

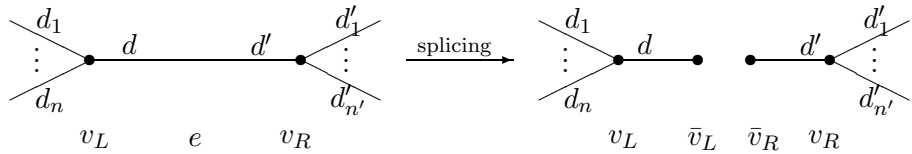
\begin{picture}(100,75)(-25,-20)

\put(40,20){\circle*{4}} %\put(80,20){\circle*{4}}
\put(120,20){\circle*{4}} \put(40,20){\line(1,0){80}}
\put(40,20){\line(-2,-1){30}} \put(40,20){\line(-2, 1){30}}
\put(20,23){\makebox(0,0){$\vdots$}} \put(120,20){\line(2,-1){30}}
\put(120,20){\line(2, 1){30}}
\put(137,23){\makebox(0,0){$\vdots$}}
\put(42,-10){\makebox(0,0){$v_L$}}
\put(118,-10){\makebox(0,0){$v_R$}}
\put(55,26){\makebox(0,0){$d$}} \put(105,26){\makebox(0,0){$d'$}}
\put(138,35) {\makebox(0,0){$d'_1$}}
\put(138,5){\makebox(0,0){$d'_{n'}$}} \put(25,35)
{\makebox(0,0){$d_1$}} \put(25,5){\makebox(0,0){$d_n$}}
\put(80,-10){\makebox(0,0){$e$}}

\put(160,20){\vector(1,0){40}}
\put(164,25){\makebox(0,0)[l]{\tiny{splicing}}}

\put(240,20){\circle*{4}}
\put(270,20){\circle*{4}}\put(290,20){\circle*{4}}
\put(320,20){\circle*{4}}
\put(240,20){\line(1,0){30}}\put(290,20){\line(1,0){30}}
\put(240,20){\line(-2,-1){30}} \put(240,20){\line(-2, 1){30}}
\put(220,23){\makebox(0,0){$\vdots$}}
\put(320,20){\line(2,-1){30}} \put(320,20){\line(2, 1){30}}
\put(337,23){\makebox(0,0){$\vdots$}}
\put(242,-10){\makebox(0,0){$v_L$}}
\put(318,-10){\makebox(0,0){$v_R$}}
\put(272,-10){\makebox(0,0){$\bar{v}_L$}}
\put(295,-10){\makebox(0,0){$\bar{v}_R$}}
\put(250,26){\makebox(0,0){$d$}} \put(313,26){\makebox(0,0){$d'$}}
\put(338,35) {\makebox(0,0){$d'_1$}}
\put(338,5){\makebox(0,0){$d'_{n'}$}} \put(225,35)
{\makebox(0,0){$d_1$}} \put(225,5){\makebox(0,0){$d_n$}}

\end{picture}

\noindent (The diagrams $\Gamma_L$ and $\Gamma_R$ correspond again
to dual graphs of certain  IHS normal surface singularities. If $\Gamma$ represents $M=S^3$, then both
$\Gamma_L$ and $\Gamma_R$ represent $S^3$, see (\ref{be:11}).)

\begin{bekezdes}\labelpar{be:f} {\bf Splicing $\Gamma(F)$.}
Next, we analyze the behavior of the multiplicity system
determined by a function $f$ or a $P$--divisor $F$. Let
$\calA:=\calA_F$ be the index set of arrowheads;
%(irreducible components of $F$);
it can be
written as a disjoint union $\calA_L\cup \calA_R$, according to
the position of the arrowheads. First, assume that both $\calA_L$
and $\calA_R$ are non--empty. If $a\in \calA_L$, then let
$\ell_{ea}$ be the product of the edge weights, all of them in
$\Gamma_L$, that are adjacent to, but not on, the path from $v_R$
to $a$. Symmetrically, one defines the integers $\ell_{ea}$ for
$a\in\calA_R$. Then $\Gamma(F)$  has the following splice
decomposition:

\begin{picture}(100,75)(5,-20)

\put(40,20){\circle*{4}} %\put(80,20){\circle*{4}}
\put(120,20){\circle*{4}} \put(40,20){\line(1,0){80}}
\put(40,20){\line(-2,-1){30}} \put(40,20){\line(-2, 1){30}}
\put(20,23){\makebox(0,0){$\vdots$}} \put(120,20){\line(2,-1){30}}
\put(120,20){\line(2, 1){30}}
\put(137,23){\makebox(0,0){$\vdots$}}
\put(42,-10){\makebox(0,0){$v_L$}}
\put(118,-10){\makebox(0,0){$v_R$}}
\put(55,26){\makebox(0,0){$d$}} \put(105,26){\makebox(0,0){$d'$}}
\put(138,35) {\makebox(0,0){$d'_1$}}
\put(138,5){\makebox(0,0){$d'_{n'}$}} \put(25,35)
{\makebox(0,0){$d_1$}} \put(25,5){\makebox(0,0){$d_n$}}
\put(80,-10){\makebox(0,0){$e$}}
\put(42,10){\makebox(0,0){\tiny{$(N)$}}}
\put(118,10){\makebox(0,0){\tiny{$(N')$}}}

\put(160,20){\vector(1,0){40}}
\put(163,25){\makebox(0,0)[l]{\tiny{splicing}}}

\put(240,20){\circle*{4}} %\put(270,20){\circle*{4}}\put(290,20){\circle*{4}}
\put(370,20){\circle*{4}}
\put(240,20){\vector(1,0){30}}\put(370,20){\vector(-1,0){30}}
\put(240,20){\line(-2,-1){30}} \put(240,20){\line(-2, 1){30}}
\put(220,23){\makebox(0,0){$\vdots$}}
\put(370,20){\line(2,-1){30}} \put(370,20){\line(2, 1){30}}
\put(387,23){\makebox(0,0){$\vdots$}}
\put(242,-10){\makebox(0,0){$v_L$}}
\put(368,-10){\makebox(0,0){$v_R$}}
\put(250,26){\makebox(0,0){$d$}} \put(363,26){\makebox(0,0){$d'$}}
\put(388,35) {\makebox(0,0){$d'_1$}}
\put(388,5){\makebox(0,0){$d'_{n'}$}} \put(225,35)
{\makebox(0,0){$d_1$}} \put(225,5){\makebox(0,0){$d_n$}}
\put(242,10){\makebox(0,0){\tiny{$(N)$}}}
\put(368,10){\makebox(0,0){\tiny{$(N')$}}}
\put(285,20){\makebox(0,0){\tiny{$(M)$}}}
\put(325,20){\makebox(0,0){\tiny{$(M')$}}}
\end{picture}

\noindent where (cf. \cite[(10.6)]{EN})
\begin{equation}\label{eq:mul}
M=\sum_{a\in\calA_R}\, N_a\ell_{ea}\ \ \ \ \mbox{and} \ \ \ \ M'=\sum_{a\in\calA_L}\, N_a\ell_{ea}.
\end{equation}

\noindent If all the arrowheads of $\Gamma$ are in one side, say
$\calA_L=\emptyset$, then one has the new situation

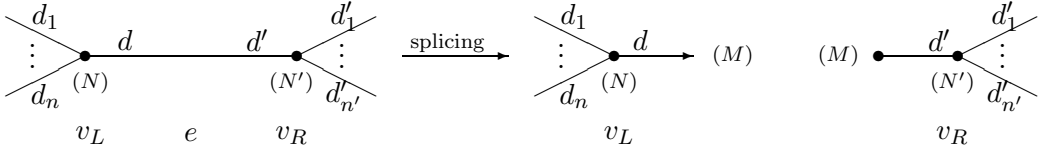
\begin{picture}(100,75)(5,-20)

\put(40,20){\circle*{4}} %\put(80,20){\circle*{4}}
\put(120,20){\circle*{4}} \put(40,20){\line(1,0){80}}
\put(40,20){\line(-2,-1){30}} \put(40,20){\line(-2, 1){30}}
\put(20,23){\makebox(0,0){$\vdots$}} \put(120,20){\line(2,-1){30}}
\put(120,20){\line(2, 1){30}}
\put(137,23){\makebox(0,0){$\vdots$}}
\put(42,-10){\makebox(0,0){$v_L$}}
\put(118,-10){\makebox(0,0){$v_R$}}
\put(55,26){\makebox(0,0){$d$}} \put(105,26){\makebox(0,0){$d'$}}
\put(138,35) {\makebox(0,0){$d'_1$}}
\put(138,5){\makebox(0,0){$d'_{n'}$}} \put(25,35)
{\makebox(0,0){$d_1$}} \put(25,5){\makebox(0,0){$d_n$}}
\put(80,-10){\makebox(0,0){$e$}}
\put(42,10){\makebox(0,0){\tiny{$(N)$}}}
\put(118,10){\makebox(0,0){\tiny{$(N')$}}}

\put(160,20){\vector(1,0){40}}
\put(163,25){\makebox(0,0)[l]{\tiny{splicing}}}

\put(240,20){\circle*{4}} \put(340,20){\circle*{4}} %\put(290,20){\circle*{4}}
\put(370,20){\circle*{4}}
\put(240,20){\vector(1,0){30}}\put(370,20){\line(-1,0){30}}
\put(240,20){\line(-2,-1){30}} \put(240,20){\line(-2, 1){30}}
\put(220,23){\makebox(0,0){$\vdots$}}
\put(370,20){\line(2,-1){30}} \put(370,20){\line(2, 1){30}}
\put(387,23){\makebox(0,0){$\vdots$}}
\put(242,-10){\makebox(0,0){$v_L$}}
\put(368,-10){\makebox(0,0){$v_R$}}
\put(250,26){\makebox(0,0){$d$}} \put(363,26){\makebox(0,0){$d'$}}
\put(388,35) {\makebox(0,0){$d'_1$}}
\put(388,5){\makebox(0,0){$d'_{n'}$}} \put(225,35)
{\makebox(0,0){$d_1$}} \put(225,5){\makebox(0,0){$d_n$}}
\put(242,10){\makebox(0,0){\tiny{$(N)$}}}
\put(368,10){\makebox(0,0){\tiny{$(N')$}}}
\put(285,20){\makebox(0,0){\tiny{$(M)$}}}
\put(325,20){\makebox(0,0){\tiny{$(M)$}}}
\end{picture}

\noindent where $M$ is computed by the same formula as in
(\ref{eq:mul}).\end{bekezdes}

 We denote the \lq
total\rq\ inherited divisors on $\Gamma_L$ and $\Gamma_R$ by $F_L$
and $F_R$, respectively. They can be identified with $P$--divisors
of the diagrams $\Gamma_L$ and $\Gamma_R$, respectively.
 \medskip
\begin{remark}\label{re:splicediv}
When $F=\div(f)$ is a plane curve germ, the left and right
graphs above correspond again to dual graphs associated to plane
curve germs. In particular, the arrowheads with multiplicities $M$ and
$M'$ correspond to components of these new germs.
%(2) In general however, even if $F$ is the divisor of a function
%germ $f$, these new arrowheads in $\Gamma_L$ and $\Gamma_R$
%correspond to (irreducible) Weil divisors, that are {\it not
%necessarily} divisors of (analytic) functions.
\end{remark}

\begin{bekezdes}\labelpar{be:fo} {\bf Splicing $\Gamma(F,W)$.}
Finally, let us analyze the behavior of the divisor of a 2--form
$\omega$ or, more generally, $K+W$ for some $P$--divisor $W$. Since the splicing
of $W$ is covered by the previous step (valid for any $P$--divisor), we have to understand
what happens to  $K$ only.

%The divisor of $\pi^*(\omega_0)$ is the {\it canonical divisor} of $\tilde{X}$,  usually denoted by $K$, and it is
%determined in the dual resolution graph $\Gx$ by the {\it adjunction relations:}
%\begin{equation}\label{eq:adj}
%(K+E,E_i)=\delta_i-2,
%\end{equation}
%where here $\delta_i$ is the valency of the vertex $i$ in $\Gx$.
%Let $L$ be the lattice $H_2(\tilde{X},\Z)$ with the intersection form $I(G)$. Then for each vertex $i$ of $G$
%one can define $E^*_i\in L$ with $(E^*_i,E_j)=-\delta_{ij}$ (the {\it negative } of the Kronecker delta), i.e. the (sign--modified) dual basis of $L$.
%(Since $\det(-I(G))=1$, they are well--defined.)
%Therefore, (\ref{eq:adj}) reads as
%\begin{equation}\labelpar{eq:adj2}
%$K+E=\sum_{i}\, (2-\delta_i)E^*_i$.
%\end{equation}
%Since all the valency 2 vertices of $G$ are irrelevant,  the
%relation descends naturally to the level of $\Gamma=\Gax$:
%\begin{equation}\label{eq:adj2}
%K+E=\sum_{v\in\calV(\Gamma)}\, (2-\delta_v)E^*_v.
%\end{equation}
%Since the divisor of $E^*_v$ is the same as the divisor along $E$ of the
%function represented by an arrowhead attached to $E_v$,  $E_v^*$ can be treated as a functions.

Let us consider the splicing of $\Gamma$ along $e$ as in
(\ref{be:d}). Any invariant associated with $\Gamma$ has its
analogue for $\Gamma_L$ and $\Gamma_R$. We wish to compare the
divisors of the pullbacks of the forms $\omega_{0,\Gamma}$ with
those of $\omega_{0,\Gamma_L}$ and $\omega_{0,\Gamma_R}$ --- if
they exist analytically; and, more generally  $K_\Gamma$ with
$K_{\Gamma_L}$ and $K_{\Gamma_R}$ (a combinatorial, always
well--posed question). For any $v\in\calV(\Gamma_L)$, we denote
the dual basis element computed in $\Gamma_L$ by $E^*_{v,\Gamma_L}$.
Moreover, we separate the vertices of $\Gamma_L$ inherited from
$\Gamma$: we set $\bar{\calV}(\Gamma_L):=\calV(\Gamma_L)\setminus
\{\bar{v}_L\}$, cf. the notation of (\ref{be:d}), and similarly
for $\Gamma_R$. Let us rewrite (\ref{eq:adj2}) into
\begin{equation}\label{eq:adj3}
(K+E)_\Gamma=\sum_{v\in\bar{\calV}(\Gamma_L)}\,
(2-\delta_v)E^*_v+\sum_{v\in\bar{\calV}(\Gamma_R)}\,
(2-\delta_v)E^*_v.
\end{equation}
Recall that $E^*_v$ behaves as a $P$--divisor associated with  one
arrowhead supported on $v$. So, by (\ref{eq:mult}), for any $v\in
\bar{\calV}(\Gamma_L)$ the restrictions satisfy
$E^*_v|_{\bar{\calV}(\Gamma_L)}=E^*_{v,\Gamma_L}|_{\bar{\calV}(\Gamma_L)}$
(i.e., the multiplicities along  $\bar{\calV}(\Gamma_L)$ agree).
Hence
\begin{equation*}%\label{eq:adj4}
(K+E)_\Gamma|_{\bar{\calV}(\Gamma_L)}=\Big(\
(K+E)_{\Gamma_L}-E^*_{\bar{v}_L,\Gamma_L}+
\sum_{v\in\bar{\calV}(\Gamma_R)}\, (2-\delta_v)E^*_v\
\Big)|_{\bar{\calV}(\Gamma_L)}.
\end{equation*}
Clearly, all multiplicities of both $E_\Gamma$ and $E_{\Gamma_L}$
along $\bar{\calV}(\Gamma_L)$ are one, hence they cancel:
\begin{equation}\label{eq:adj4}
K_\Gamma|_{\bar{\calV}(\Gamma_L)}=\Big(\ K_{\Gamma_L}-E^*_{\bar{v}_L,\Gamma_L}+
\sum_{v\in\bar{\calV}(\Gamma_R)}\, (2-\delta_v)E^*_v\ \Big)|_{\bar{\calV}(\Gamma_L)}.
\end{equation}
Using again (\ref{eq:mult}), the sum  can
be  replaced by  a $P$--divisor of $\Gamma_L$. Indeed,
set
\begin{equation}\label{eq:i}
i=i_{e,L}:=\sum_{v\in\bar{\calV}(\Gamma_R)}\, (2-\delta_v)\ell_{ev},
\end{equation}
where $\ell_{ev}$ (for any $v\in \bar{\calV}(\Gamma_R)$) is
the product of the edge weights of $\Gamma$, all of them in $\Gamma_R$,
that are adjacent to, but not on, the path  from $v_L$ to $v$.
Furthermore, let $G_L$ be the $P$--divisor on $\Gamma_L$
determined by one arrowhead with multiplicity one supported
on $\bar{v}_L$.
Then (\ref{eq:mult}) and (\ref{eq:adj4}) imply
\begin{equation}\label{eq:adj5}
K_\Gamma|_{\bar{\calV}(\Gamma_L)}= \left(K_{\Gamma_L} +
(i-1)G_L\right)|_{\bar{\calV}(\Gamma_L)}.
\end{equation}
If the forms above exist, and if $G_L$ is the pullback divisor of a
function $g_L$, then we have
 \begin{equation}\label{eq:adj6}
\mbox{div}(\pi^*\omega_{0,\Gamma})|_{\bar{\calV}(\Gamma_L)}=
\mbox{div}\left(\pi_L^*(g^{i-1}_L\cdot
\omega_{0,\Gamma_L})\right)|_{\bar{\calV}(\Gamma_L)}.
\end{equation}
Obviously, there is a symmetric identity for the restriction on
$\bar{\calV}(\Gamma_R)$. On diagrams we have

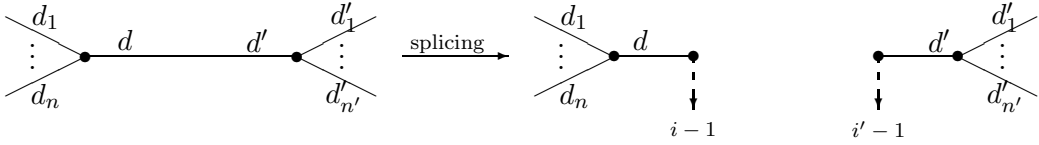
\begin{picture}(100,72)(5,-20)

\put(40,20){\circle*{4}} %\put(80,20){\circle*{4}}
\put(120,20){\circle*{4}}
\put(40,20){\line(1,0){80}}
\put(40,20){\line(-2,-1){30}}
\put(40,20){\line(-2, 1){30}}
\put(20,23){\makebox(0,0){$\vdots$}}
\put(120,20){\line(2,-1){30}}
\put(120,20){\line(2, 1){30}}
\put(137,23){\makebox(0,0){$\vdots$}}
%\put(42,-10){\makebox(0,0){$v_L$}}
%\put(118,-10){\makebox(0,0){$v_R$}}
\put(55,26){\makebox(0,0){$d$}} \put(105,26){\makebox(0,0){$d'$}}
\put(138,35) {\makebox(0,0){$d'_1$}}
\put(138,5){\makebox(0,0){$d'_{n'}$}} \put(25,35)
{\makebox(0,0){$d_1$}} \put(25,5){\makebox(0,0){$d_n$}}
%\put(80,-10){\makebox(0,0){$\omega_{0,\Gamma}$}}

%\put(0,20){\makebox(0,0){$(\dagger)$}}
%\put(118,10){\makebox(0,0){$(N')$}}

\put(160,20){\vector(1,0){40}}
\put(163,25){\makebox(0,0)[l]{\tiny{splicing}}}

\put(240,20){\circle*{4}} \put(270,20){\circle*{4}} \put(340,20){\circle*{4}}
\put(370,20){\circle*{4}}
\put(240,20){\line(1,0){30}}\put(370,20){\line(-1,0){30}}
\put(240,20){\line(-2,-1){30}}
\put(240,20){\line(-2, 1){30}}
\put(220,23){\makebox(0,0){$\vdots$}}
\put(370,20){\line(2,-1){30}}
\put(370,20){\line(2, 1){30}}
\put(387,23){\makebox(0,0){$\vdots$}}
%\put(242,-10){\makebox(0,0){$v_L$}}
%\put(368,-10){\makebox(0,0){$v_R$}}
\put(250,26){\makebox(0,0){$d$}} \put(363,26){\makebox(0,0){$d'$}}
\put(388,35) {\makebox(0,0){$d'_1$}}
\put(388,5){\makebox(0,0){$d'_{n'}$}} \put(225,35)
{\makebox(0,0){$d_1$}} \put(225,5){\makebox(0,0){$d_n$}}
%\put(242,10){\makebox(0,0){$(N)$}}
%\put(368,10){\makebox(0,0){$(N')$}}
\put(270,-8){\makebox(0,0){\tiny{$i-1$}}}
\put(340,-8){\makebox(0,0){\tiny{$i'-1$}}}

\dashline[3]{3}(270,20)(270,5)\put(270,5){\vector(0,-1){5}}
\dashline[3]{3}(340,20)(340,5)\put(340,5){\vector(0,-1){5}}

\end{picture}

\noindent When we incorporate also the divisor $W$, the equations (\ref{eq:i}) and (\ref{eq:adj5})
respectively extend to
\begin{equation}\label{eq:totali}
i:=\sum_{v\in\bar{\calV}(\Gamma_R)}\, (2-\delta_v)\ell_{ev}+
\sum_{a\in \calA_{W,R}} (i_a-1)\ell_{ea},
\end{equation}
where $\ell_{ea}$  is the product of the edge weights of $\Gamma$,
all of them in $\Gamma_R$,
 that are adjacent to, but not on, the path  from $v_L$ to the corresponding dashed arrow in $\Gamma_R$, and
\begin{equation}\label{eq:totaladj5}
(K_\Gamma+W)|_{\bar{\calV}(\Gamma_L)}= \left(K_{\Gamma_L}+
(i-1)G_L +W^+_L\right)|_{\bar{\calV}(\Gamma_L)},
\end{equation}
where $W^+_L$ is induced by $W$ on $\Gamma_L$, as in (\ref{be:f}).
%Remark \ref{re:splicediv}(2).
We will reserve the notation $W_L$ rather for $(i-1)G_L +W^+_L$,
in order to have the expression
\begin{equation}\label{eq:totaladj6}
(K_\Gamma+W)|_{\bar{\calV}(\Gamma_L)}= \left(K_{\Gamma_L}+
W_L\right)|_{\bar{\calV}(\Gamma_L)}.
\end{equation}

Note that (usually) the vertices $\bar{v}_L$ and $\bar{v}_R$ have
valency 2 (counting all the arrowheads), hence in the formulas
considered in the next sections they will be irrelevant.

In the presence of  a $P$--divisor $F$, with both
$\calA_{F,L}$ and $\calA_{F,R}$ non--empty, the above diagram modifies
into

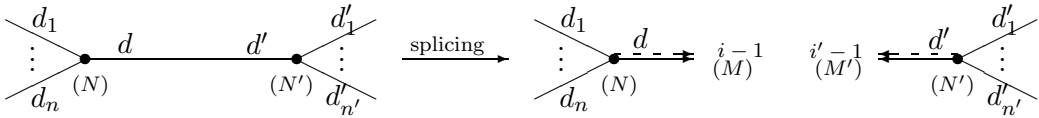
\begin{picture}(100,72)(5,-20)

\put(40,20){\circle*{4}} %\put(80,20){\circle*{4}}
\put(120,20){\circle*{4}}
\put(40,20){\line(1,0){80}}
\put(40,20){\line(-2,-1){30}}
\put(40,20){\line(-2, 1){30}}
\put(20,23){\makebox(0,0){$\vdots$}}
\put(120,20){\line(2,-1){30}}
\put(120,20){\line(2, 1){30}}
\put(137,23){\makebox(0,0){$\vdots$}}
%\put(42,-10){\makebox(0,0){$v_L$}}
%\put(118,-10){\makebox(0,0){$v_R$}}
\put(55,26){\makebox(0,0){$d$}} \put(105,26){\makebox(0,0){$d'$}}
\put(138,35) {\makebox(0,0){$d'_1$}}
\put(138,5){\makebox(0,0){$d'_{n'}$}} \put(25,35)
{\makebox(0,0){$d_1$}} \put(25,5){\makebox(0,0){$d_n$}}
%\put(80,-10){\makebox(0,0){$\omega_{0,\Gamma}$}}
\put(42,10){\makebox(0,0){\tiny{$(N)$}}}
\put(118,10){\makebox(0,0){\tiny{$(N')$}}}

\put(160,20){\vector(1,0){40}}
\put(163,25){\makebox(0,0)[l]{\tiny{splicing}}}

\put(240,20){\circle*{4}} %\put(270,20){\circle*{4}} \put(340,20){\circle*{4}}
\put(370,20){\circle*{4}}
\put(240,20){\vector(1,0){30}}\put(370,20){\vector(-1,0){30}}
\put(240,20){\line(-2,-1){30}}
\put(240,20){\line(-2, 1){30}}
\put(220,23){\makebox(0,0){$\vdots$}}
\put(370,20){\line(2,-1){30}}
\put(370,20){\line(2, 1){30}}
\put(387,23){\makebox(0,0){$\vdots$}}
%\put(242,-10){\makebox(0,0){$v_L$}}
%\put(368,-10){\makebox(0,0){$v_R$}}
\put(250,28){\makebox(0,0){$d$}} \put(363,28){\makebox(0,0){$d'$}}
\put(388,35) {\makebox(0,0){$d'_1$}}
\put(388,5){\makebox(0,0){$d'_{n'}$}} \put(225,35)
{\makebox(0,0){$d_1$}} \put(225,5){\makebox(0,0){$d_n$}}
\put(242,10){\makebox(0,0){\tiny{$(N)$}}}
\put(368,10){\makebox(0,0){\tiny{$(N')$}}}
\put(288,23){\makebox(0,0){\tiny{$i-1$}}}
\put(324,23){\makebox(0,0){\tiny{$i'-1$}}}
\put(285,16){\makebox(0,0){\tiny{$(M)$}}}
\put(325,16){\makebox(0,0){\tiny{$(M')$}}}

\dashline[3]{3}(240,22)(265,22)\put(265,22){\vector(1,0){5}}
\dashline[3]{3}(370,22)(345,22)\put(345,22){\vector(-1,0){5}}

\end{picture}

\noindent where $M$ and $M'$ are determined as in (\ref{eq:mul}).
If $\calA_{F,L}=\emptyset$, then one has
%the diagram of $\Gamma_L$ is replaced by
%$\Gamma_L$ from the previous splice decomposition ($\dagger$);
%(compare with the two cases from (\ref{be:f}):

\begin{picture}(100,66)(5,-10)

\put(40,20){\circle*{4}} %\put(80,20){\circle*{4}}
\put(120,20){\circle*{4}}
\put(40,20){\line(1,0){80}}
\put(40,20){\line(-2,-1){30}}
\put(40,20){\line(-2, 1){30}}
\put(20,23){\makebox(0,0){$\vdots$}}
\put(120,20){\line(2,-1){30}}
\put(120,20){\line(2, 1){30}}
\put(137,23){\makebox(0,0){$\vdots$}}
%\put(42,-10){\makebox(0,0){$v_L$}}
%\put(118,-10){\makebox(0,0){$v_R$}}
\put(55,26){\makebox(0,0){$d$}} \put(105,26){\makebox(0,0){$d'$}}
\put(138,35) {\makebox(0,0){$d'_1$}}
\put(138,5){\makebox(0,0){$d'_{n'}$}} \put(25,35)
{\makebox(0,0){$d_1$}} \put(25,5){\makebox(0,0){$d_n$}}
%\put(80,-10){\makebox(0,0){$\omega_{0,\Gamma}$}}
\put(42,10){\makebox(0,0){\tiny{$(N)$}}}
\put(118,10){\makebox(0,0){\tiny{$(N')$}}}

\put(160,20){\vector(1,0){40}}
\put(163,25){\makebox(0,0)[l]{\tiny{splicing}}}

\put(240,20){\circle*{4}} %\put(340,20){\circle*{4}} \put(340,20){\circle*{4}}
\put(370,20){\circle*{4}}\put(340,20){\circle*{4}}
\put(240,20){\vector(1,0){30}}
\put(370,20){\line(-1,0){30}}
\put(240,20){\line(-2,-1){30}}
\put(240,20){\line(-2, 1){30}}
\put(220,23){\makebox(0,0){$\vdots$}}
\put(370,20){\line(2,-1){30}}
\put(370,20){\line(2, 1){30}}
\put(387,23){\makebox(0,0){$\vdots$}}
%\put(242,-10){\makebox(0,0){$v_L$}}
%\put(368,-10){\makebox(0,0){$v_R$}}
\put(250,28){\makebox(0,0){$d$}} \put(363,28){\makebox(0,0){$d'$}}
\put(388,35) {\makebox(0,0){$d'_1$}}
\put(388,5){\makebox(0,0){$d'_{n'}$}} \put(225,35)
{\makebox(0,0){$d_1$}} \put(225,5){\makebox(0,0){$d_n$}}
\put(242,10){\makebox(0,0){\tiny{$(N)$}}}
\put(368,10){\makebox(0,0){\tiny{$(N')$}}}
%\put(288,23){\makebox(0,0){\tiny{$i_{e,L}-1$}}}
\put(340,-5){\makebox(0,0){\tiny{$i'-1$}}}
\put(285,16){\makebox(0,0){\tiny{$(M)$}}}
\put(325,16){\makebox(0,0){\tiny{$(M)$}}}
\put(283,23){\makebox(0,0){\tiny{$i-1$}}}

\dashline[3]{3}(340,20)(340,5)\put(340,5){\vector(0,-1){5}}
\dashline[3]{3}(240,22)(265,22)\put(265,22){\vector(1,0){5}}

\end{picture}

\end{bekezdes}
\begin{remark}\labelpar{re:forms}
(1) Assume that we start with $W=0$. Then both $W^+_L$ and $W^+_R$
are zero, but in general $W_L$ and $W_R$ are not. Furthermore,
even  if $W$ is {\it effective}  (in particular, if $W=0$), the
induced divisors $W_L$ and/or $W_R$ are in general not effective.

In a different language: even
if we start with the `standard 2--from' $\omega_0$ (instead of the
more general $g\omega_0$) --- like in the traditional framework
of, say, topological or motivic zeta functions ---, once an
inductive splice--decomposition argument is used, we are forced to
enlarge the class of our forms: in (\ref{eq:adj6}) the standard
form decomposes in the splice component into a generalized form of
type $g\omega_0$. Also, one can see that even if we start with a
holomorphic form, the forms on the splice components, usually, are
not holomorphic (that is, they are meromorphic).

(2) Even if we have a precise analytic realization of a diagram $\Gamma$,
it is not clear what the relations are connecting  this analytic structure and the
eventual analytic realizations of $\Gamma_L$ and $\Gamma_R$.
In general, there is no analytic construction known by the authors
 which  would define a
natural analytic  structure  with topology $\Gamma_L$ starting from the original $(X,0)$.

In the presence of functions and forms the situation becomes even more difficult. In that case
it might happen that even if we know that $\Gamma(F,W)$ is analytically realized as $\Gamma_\pi(X,f,\omega)$,
after splicing the two combinatorial packages might not be realized analytically
(for example, the  `correction term' $(i-1)G_L$ is maybe not the divisor of a
function).

Nevertheless, if $\Gamma$ represents $S^3$, i.e. if any realization of $\Gamma$ is smooth,
then all the functions and (meromorphic) forms will exist.
\end{remark}
%\marginpar{???????????}

%\smallskip
%{\bf TO DO: discuss a concrete example throughout the previous part
%(that can be used later also when treating zeta functions). Must
%choose a good one.}
% \smallskip

\begin{example}\labelpar{ex:EXAMPLE} Consider the following diagram $\Gamma(F)$, compare also with
Remark~\ref{re:1234}(4).

\begin{picture}(400,70)(0,-10)

\put(100,25){\circle*{4}} \put(150,25){\circle*{4}}
\put(200,25){\circle*{4}} \put(250,25){\circle*{4}}
\put(300,25){\circle*{4}} \put(150,5){\circle*{4}}
\put(250,5){\circle*{4}} \put(100,25){\line(1,0){200}}
\put(150,25){\line(0,-1){20}} \put(200,25){\vector(0,-1){20}}
%\put(370,5){\circle*{4}}
\put(250,25){\line(0,-1){20}} \put(145,30){\makebox(0,0){\tiny{$2$}}}
\put(195,30){\makebox(0,0){\tiny{$1$}}}\put(155,30){\makebox(0,0){\tiny{$7$}}}
\put(245,30){\makebox(0,0){\tiny{$7$}}}\put(205,30){\makebox(0,0){\tiny{$1$}}}
%\put(375,35){\makebox(0,0){\tiny{$-13$}}}
%\put(375,-3){\makebox(0,0){\tiny{$(N)$}}}
\put(256,30){\makebox(0,0){\tiny{$2$}}}
%\put(425,35){\makebox(0,0){\tiny{$-2$}}}
\put(155,20){\makebox(0,0){\tiny{$3$}}}
\put(255,20){\makebox(0,0){\tiny{$3$}}}%\put(322,5){\makebox(0,0){$-7$}}

\put(100,18){\makebox(0,0){\tiny{$(3)$}}}
\put(145,18){\makebox(0,0){\tiny{$(6)$}}}
\put(206,18){\makebox(0,0){\tiny{$(1)$}}}
\put(245,18){\makebox(0,0){\tiny{$(6)$}}}
\put(300,18){\makebox(0,0){\tiny{$(3)$}}}
\put(150,-2){\makebox(0,0){\tiny{$(2)$}}}
\put(200,-2){\makebox(0,0){\tiny{$(1)$}}}
\put(250,-2){\makebox(0,0){\tiny{$(2)$}}}
\put(150,45){\makebox(0,0){\tiny{$v_1$}}}
\put(200,45){\makebox(0,0){\tiny{$v_0$}}}
\put(250,45){\makebox(0,0){\tiny{$v_1'$}}}

\end{picture}

We denote the nodes by $v_1,\ v_0$ and $v_1'$, and their
$\nu$--numbers by $\nu_1,\ \nu_0$ and $\nu_1'$. Then
$\nu_1=\nu_1'=-13$ and $\nu_0=-2$. Splicing the  diagram
$\Gamma(F,W=0)$ we get the three star--shaped subgraphs
 $\Gamma_1$, $\Gamma_0$ and $\Gamma_1'$:

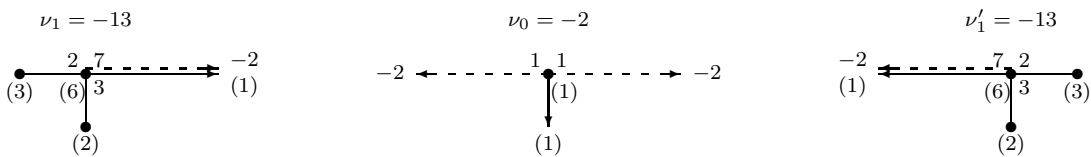
\begin{picture}(400,70)(0,-10)

\put(0,25){\circle*{4}} \put(25,25){\circle*{4}}
\put(200,25){\circle*{4}} \put(375,25){\circle*{4}}
\put(400,25){\circle*{4}} \put(25,5){\circle*{4}}
\put(375,5){\circle*{4}} \put(0,25){\vector(1,0){75}}\put(400,25){\vector(-1,0){75}}
\put(25,25){\line(0,-1){20}} \put(200,25){\vector(0,-1){20}}
\put(375,25){\line(0,-1){20}} \put(20,30){\makebox(0,0){\tiny{$2$}}}
\put(195,30){\makebox(0,0){\tiny{$1$}}}\put(30,30){\makebox(0,0){\tiny{$7$}}}
\put(370,30){\makebox(0,0){\tiny{$7$}}}\put(205,30){\makebox(0,0){\tiny{$1$}}}
%\put(375,35){\makebox(0,0){\tiny{$-13$}}}
%\put(375,-3){\makebox(0,0){\tiny{$(N)$}}}
\put(380,30){\makebox(0,0){\tiny{$2$}}}
%\put(425,35){\makebox(0,0){\tiny{$-2$}}}
\put(30,20){\makebox(0,0){\tiny{$3$}}}
\put(380,20){\makebox(0,0){\tiny{$3$}}}%\put(322,5){\makebox(0,0){$-7$}}

\put(0,18){\makebox(0,0){\tiny{$(3)$}}}
\put(20,18){\makebox(0,0){\tiny{$(6)$}}}
\put(206,18){\makebox(0,0){\tiny{$(1)$}}}
\put(370,18){\makebox(0,0){\tiny{$(6)$}}}
\put(400,18){\makebox(0,0){\tiny{$(3)$}}}
\put(25,-2){\makebox(0,0){\tiny{$(2)$}}}
\put(200,-2){\makebox(0,0){\tiny{$(1)$}}}
\put(375,-2){\makebox(0,0){\tiny{$(2)$}}}

\put(25,45){\makebox(0,0){\tiny{$\nu_1=-13$}}}
\put(200,45){\makebox(0,0){\tiny{$\nu_0=-2$}}}
\put(375,45){\makebox(0,0){\tiny{$\nu_1'=-13$}}}

\dashline[3]{3}(25,27)(70,27)\put(70,27){\vector(1,0){5}}
\dashline[3]{3}(375,27)(330,27)\put(330,27){\vector(-1,0){5}}
\dashline[3]{3}(200,25)(155,25)\put(155,25){\vector(-1,0){5}}
\dashline[3]{3}(200,25)(245,25)\put(245,25){\vector(1,0){5}}

\put(85,20){\makebox(0,0){\tiny{$(1)$}}}
\put(85,30){\makebox(0,0){\tiny{$-2$}}}
\put(140,25){\makebox(0,0){\tiny{$-2$}}}
\put(260,25){\makebox(0,0){\tiny{$-2$}}}
\put(315,20){\makebox(0,0){\tiny{$(1)$}}}
\put(315,30){\makebox(0,0){\tiny{$-2$}}}
\end{picture}

\noindent The zeta function $Z(\Gamma;s)$ associated with
$\Gamma(F,W=0)$, cf. (\ref{I:zetadiagram}), is
$$Z(\Gamma;s)=2\frac{4}{6s-13}+\frac{1}{s-2}(-1+\frac{1}{s+1})+2\frac{1}{(s-2)(6s-13)}.$$

\end{example}

\subsection{Splicing the topological zeta function}\labelpar{ss:splicingzeta}
We will analyze the splicing behavior of  the
topological zeta function of a graph $\Gamma(F,W)$. Let us consider
again the splicing of the diagram $\Gamma(F,W)$ along the edge $e$ as
in (\ref{be:fo}), where we insert the relevant integers
 $(N,\nu-1)$ for both vertices $v_L$ and $v_R$:

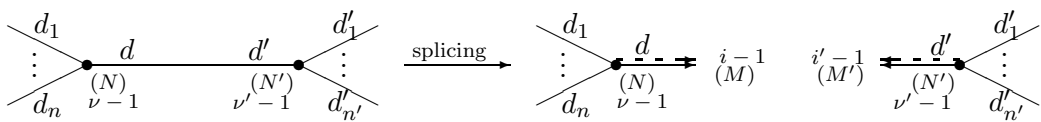
\begin{picture}(100,75)(5,-20)

\put(40,20){\circle*{4}} %\put(80,20){\circle*{4}}
\put(120,20){\circle*{4}} \put(40,20){\line(1,0){80}}
\put(40,20){\line(-2,-1){30}} \put(40,20){\line(-2, 1){30}}
\put(20,23){\makebox(0,0){$\vdots$}} \put(120,20){\line(2,-1){30}}
\put(120,20){\line(2, 1){30}}
\put(137,23){\makebox(0,0){$\vdots$}}
%\put(42,-10){\makebox(0,0){$v_L$}}
%\put(118,-10){\makebox(0,0){$v_R$}}
\put(55,26){\makebox(0,0){$d$}} \put(105,26){\makebox(0,0){$d'$}}
\put(138,35) {\makebox(0,0){$d'_1$}}
\put(138,5){\makebox(0,0){$d'_{n'}$}} \put(25,35)
{\makebox(0,0){$d_1$}} \put(25,5){\makebox(0,0){$d_n$}}
\put(48,13){\makebox(0,0){\tiny{$(N)$}}}\put(50,6){\makebox(0,0){\tiny{$\nu-1$}}}
\put(110,13){\makebox(0,0){\tiny{$(N')$}}}
\put(106,6){\makebox(0,0){\tiny{$\nu'-1$}}}

\put(160,20){\vector(1,0){40}}
\put(163,25){\makebox(0,0)[l]{\tiny{splicing}}}

\put(240,20){\circle*{4}} %\put(270,20){\circle*{4}} \put(340,20){\circle*{4}}
\put(370,20){\circle*{4}}
\put(240,20){\vector(1,0){30}}\put(370,20){\vector(-1,0){30}}
\put(240,20){\line(-2,-1){30}} \put(240,20){\line(-2, 1){30}}
\put(220,23){\makebox(0,0){$\vdots$}}
\put(370,20){\line(2,-1){30}} \put(370,20){\line(2, 1){30}}
\put(387,23){\makebox(0,0){$\vdots$}}
%\put(242,-10){\makebox(0,0){$v_L$}}
%\put(368,-10){\makebox(0,0){$v_R$}}
\put(250,28){\makebox(0,0){$d$}} \put(363,28){\makebox(0,0){$d'$}}
\put(388,35) {\makebox(0,0){$d'_1$}}
\put(388,5){\makebox(0,0){$d'_{n'}$}} \put(225,35)
{\makebox(0,0){$d_1$}} \put(225,5){\makebox(0,0){$d_n$}}

\put(248,13){\makebox(0,0){\tiny{$(N)$}}}\put(250,6){\makebox(0,0){\tiny{$\nu-1$}}}
\put(360,13){\makebox(0,0){\tiny{$(N')$}}}
\put(356,6){\makebox(0,0){\tiny{$\nu'-1$}}}

\put(288,23){\makebox(0,0){\tiny{$i-1$}}}
\put(324,23){\makebox(0,0){\tiny{$i'-1$}}}
\put(285,16){\makebox(0,0){\tiny{$(M)$}}}
\put(325,16){\makebox(0,0){\tiny{$(M')$}}}

\dashline[3]{3}(240,22)(265,22)\put(265,22){\vector(1,0){5}}
\dashline[3]{3}(370,22)(345,22)\put(345,22){\vector(-1,0){5}}

\end{picture}

If $\calA_{F,L}=\emptyset$, then replace this diagram with the
adapted one as in the last diagram of (\ref{be:fo}). In this case
of $\calA_{F,L}=\emptyset$ one always has $M'=0$.

Note that, to be able to define $Z(\Gamma_L)$ and $Z(\Gamma_R)$,
we need that the condition $(N_a,i_a)\neq(0,0)$ for all arrowheads
is also valid after splicing, see (\ref{I:zetadiagram}). For the
moment we just assume this. In the context of allowed divisors we
will show in (\ref{lemma:nonzero}) that it is always true.

The contribution of $e$ to $Z(\Gamma)$ turns out to be the sum of
the contribution of the \lq right leg\rq\ of $\Gamma_L$ to
$Z(\Gamma_L)$ and the contribution of the \lq left leg\rq\ of
$\Gamma_R$ to $Z(\Gamma_R)$, minus an easy correction term, as
shown below. This then yields a simple splicing formula for
topological zeta functions. We start with the following numerical relation.

\begin{lemma}\labelpar{lemma:edge}
We use the notation of (\ref{be:fo}), as indicated on the diagram
above, and put also $q := dd'-(\prod_{j=1}^n d_j)(
\prod_{j=1}^{n'}d'_j)$ for the edge determinant of $e$. If
$\calA_{F,L}=\emptyset$,  set $M'=0$. Then we have the equality
\begin{equation}\label{eq:spliceedge}\begin{split}
\frac q{(\nu+sN)(\nu'+sN')}
= \frac d{(\nu+sN)(i+sM)} + \frac
{d'}{(\nu'+sN')(i'+sM')}\\ - \frac 1{(i+sM)(i'+sM')}\hspace{3.5cm}.
\end{split}\end{equation}
Moreover, if two of the pairs $(\nu,N)$, $(\nu',N')$, $(i,M)$ and
$(i',M')$ are linearly dependent, then any other choice of two
pairs are also linearly dependent.
\end{lemma}

\begin{proof} The equations (\ref{eq:mult}) and
(\ref{eq:mul}), respectively (\ref{eq:nu}) and (\ref{eq:totali}),
imply
 $$\begin{cases} N= (\prod_{j=1}^n d_j)M+dM'
\\ N'=(\prod_{j=1}^{n'}d'_j)M'+d'M
\end{cases} \text{and}\quad \begin{cases} \nu= (\prod_{j=1}^n d_j)i+di' \\ \nu'=
(\prod_{j=1}^{n'} d_j)i'+d'i.
\end{cases}$$
 Hence (as polynomials in $s$)
\begin{equation}\label{eq:s}
\begin{cases}
(\prod_{j=1}^n d_j)(i+sM) = \nu+sN - d(i'+sM') \\
(\prod_{j=1}^{n'} d'_j)(i'+sM') = \nu'+sN' - d'(i+sM).
\end{cases}
\end{equation}
Multiplying the left and right hand sides of (\ref{eq:s}), and
using the defining formula of $q$, we obtain
$$
(\nu+sN)(\nu'+sN') - d{(\nu'+sN')(i'+sM')} - d'{(\nu+sN)(i+sM)}
+q(i+sM)(i'+sM')=0.
$$
This is clearly equivalent to (\ref{eq:spliceedge}). The linear
dependency statements follow easily from (\ref{eq:s}), using that
$q\neq 0$.
\end{proof}

One of the main new results of the article is the next splice decomposition
formula for $Z(\Gamma)$.

\begin{theorem}\labelpar{prop:splicezeta}
(1) Consider the splicing of the diagram $\Gamma$ as in the last
diagram of  (\ref{be:fo}). Again, if $\calA_{F,L}=\emptyset$,  set $M'=0$. Then
$$
\quad Z(\Gamma)=Z(\Gamma_L)+Z(\Gamma_R)-
 \frac 1{(i+sM)(i'+sM')}.
$$
(2) The contribution of $v_L$ in $Z(\Gamma)$ has $-\frac{\nu}N$ as
a pole of order $2$   if and only if the contribution of $v_L$ in
$Z(\Gamma_L)$ has $-\frac{\nu}N$ as a pole  of order $2$.

\noindent (3) Suppose that $-\frac{\nu}N$ is not a pole of order $2$
of $Z(\Gamma)$. Then the contributions of $v_L$ to the residue of
$Z(\Gamma)$ and to the residue of $Z(\Gamma_L)$ at $-\frac{\nu}N$
are exactly the same.
\end{theorem}

\begin{proof}
(1) This is a direct consequence of (\ref{eq:spliceedge}), since
the other contributions to $Z(\Gamma)$ appear in a disjoint way as
the other contributions to $Z(\Gamma_L)$ or $Z(\Gamma_R)$.

(2) First note that the coefficient of any expression
$1/(\nu+sN)^2$ in the topological zeta function formula is always
positive, hence
 there are no cancelations among them. Consequently, the statement follows
immediately  from the linear dependency considerations in Lemma
\ref{lemma:edge}.

(3) The difference between the contributions to both residues is
$$
\frac1N \left( \frac q{\nu'+sN'} - \frac d{i+sM} \right),
$$
evaluated in $s=-\frac{\nu}N$, and this is zero because of
(\ref{eq:spliceedge}).
\end{proof}

\begin{example}\labelpar{ex:EXAMPLE2}
With the notations of Example~\ref{ex:EXAMPLE} one has
$$Z(\Gamma_1;s)=Z(\Gamma'_1;s)=\frac{1}{6s-13}\Big(4+\frac{7}{s-1}\Big), \ \
Z(\Gamma_0;s)=\frac{1}{s-2}\Big(-3+\frac{1}{s+1}\Big). $$

\noindent One verifies that indeed, according to Proposition
\ref{prop:splicezeta}(1),
$$Z(\Gamma;s)=Z(\Gamma_1;s)+Z(\Gamma_0;s)+Z(\Gamma'_1;s)-2\frac{1}{(-1)(s-1)},$$
and in the sum $Z(\Gamma;s)$ the pole $s=1$ `disappears'.
\end{example}

\subsection{Splicing the monodromy zeta function and Alexander
polynomial}\labelpar{ss:splicingalex} Let $f:(X,0)\to  (\C,0)$ be
the germ of a holomorphic function as in (\ref{ss:2}), let $F_0$
be its Milnor fiber, $h_i:H_i(F_0,\C)\to H_i(F_0,\C)$ the
algebraic monodromy  ($i=0,1$), $\Delta_i(t):=\det(tI-h_i)$ the
characteristic polynomial of $h_i$, and finally,
$\zeta(t)=\Delta_1/\Delta_0$ the monodromy zeta function
associated with $f$ at $0$.

It is well--known, cf. \cite[(11.3)]{EN}, that the zeta function
can be computed from the splice diagram $\Gamma=\Gaxf$ as follows:
\begin{equation}\label{eq:zeta}
\zeta(t)=\prod_{v\in\calV(\Gamma)}\, (t^{N_v}-1)^{\delta_v'-2},
\end{equation}
where, for each vertex $v$, $N_v$ denotes its multiplicity and
$\delta _v'$ its valency in $\Gaxf$. The zeta--function is
`almost' multiplicative with respect to the splice decomposition.
In order to have a uniform statement, we consider the {\it
Alexander polynomial} (in one variable),  cf. \cite[(12.1)]{EN},
as follows:
\begin{equation}\label{eq:alex}
\Lambda(t):=\left\{\begin{array}{ll}
\zeta(t) & \ \mbox{if $\#\calA(\Gaxf)\geq 2$,}\\
\Delta_1(t) & \ \mbox{if $\#\calA(\Gaxf)= 1$}.
\end{array}\right.
\end{equation}

The formula (\ref{eq:zeta}) provides $\Lambda(t)$ too, since, if
$\#\calA=1$, then $\Lambda=\zeta\cdot \Delta_0$, and
$\Delta_0(t)=t^{N_a}-1$, where $N_a$ is the multiplicity of the
unique arrowhead $a$. (In general, $F_0$ has $d$ connected
components, hence $\Delta_0(t)=t^d-1$, where $d:=\mbox{gcd}_{a\in
\calA}(N_a)$.) In particular, $\Lambda (t)$ can be recovered from
the diagram $\Gamma=\Gaxf$; let us write $\Lambda_\Gamma(t)$ for
this expression.

Clearly, $\Lambda_\Gamma(t)$ depends only on the divisor of $f$,
hence the above formula defines $\Lambda_\Gamma(t)$ for any
$P$--divisor $F$ and $\Gamma(F)$. Moreover, assume that the splice
diagram $\Gamma(F)$ has the splice decomposition $\Gamma_L$ and
$\Gamma_R$ as in (\ref{be:f}), without considering or asking any
analytic realization. Then, analyzing the splice decompositions of
(\ref{be:f}) and the formula (\ref{eq:zeta}), we easily  get the
following.

\begin{proposition}\labelpar{prop:alex}
$$\Lambda_{\Gamma}(t)=\Lambda_{\Gamma_L}(t)\cdot \Lambda_{\Gamma_R}(t).$$
\end{proposition}

\smallskip
\begin{example}\labelpar{ex:EXAMPLE3}
With the notation of Example~\ref{ex:EXAMPLE} one has
$\Lambda_{\Gamma_1}(t)=\Lambda_{\Gamma'_1}(t)=t^2-t+1$,
$\Lambda_{\Gamma_0}(t)=1$, and their product is indeed
$\Lambda(t)=\Delta_1(t)=(t^2-t+1)^2.$
\end{example}

The advantage (at least in the present paper) of $\Lambda$
compared with $\zeta$ is that $\Lambda $ is a {\it polynomial},
hence we do not have to deal with possible cancelations of the
roots and poles in the multiplicative formula of Proposition
\ref{prop:alex}.

\section{Allowed forms/$P$--divisors}\labelpar{s:3}

\subsection{Definition}\labelpar{ss:defallowed}
In the original setting of a plane curve $f$ we want to pin down a
class of 2-forms $\omega$ such that we can realize the goals of
(\ref{goal:in}) from the introduction. More generally, starting with an effective
divisor $F$ on a IHS germ $(X,0)$, we look for an appropriate
class of Weil divisors $W$.

From the point of view of splicing, our definition of allowed
forms/divisors below is quite natural. It is not difficult to
identify a natural class of divisors $W$ that do the job on a \lq
basic building block\rq, i.e. a star-shaped graph. We use this as
guideline to identify our allowed divisors on a general graph,
just demanding that we obtain allowed divisors on all star-shaped
subgraphs after (repeated) splicing.

Again, the restriction is combinatorial, depending only on the
splicing graph; hence, we will treat {\it allowed $P$--divisors $W$ of graphs
$\Gamma(F)$}.

\begin{definition}\labelpar{def:alloweddiagram}
Let $\Gamma=\Gamma(F)$ be a diagram as in (\ref{ss:2}), see also
(\ref{be:notation}). The set of decorated dashed arrows
$\sum_{a\in\calA_W}(i_a-1)W_a$, i.e., the associated $P$--divisor
$W$, is {\it allowed} for $\Gamma$ (or, the diagram $\Gamma(F,W)$
is allowed), if the following conditions are satisfied:
\begin{enumerate}\label{allow}

\item\label{allow.1}
$i_a\neq 0$ for $a\in\calA_W\setminus \calA_F$,
that is, $(N_a,i_a)\neq(0,0)$ for all $a\in\calA_F\cup \calA_W$.
\item \label{allow.2}
Suppose that $\Gamma$ is star-shaped. Let the central node be
connected to $n$ boundary vertices whose supporting edges have
decorations $\{d_\ell\}_{\ell=1}^n$, and with $r $ other
incident edges connecting with  arrowheads, doubled by dashed arrows or not ($r\geq 1$ always).

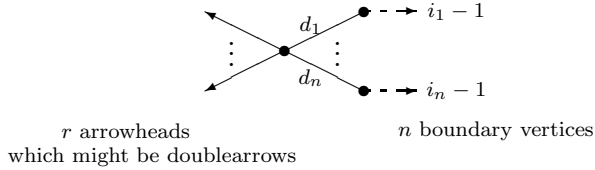
\begin{picture}(310,70)(0,10)

\put(150,50){\circle*{4}} \put(180,65){\circle*{4}} \put(180,35){\circle*{4}}
\put(150,50){\line(2,1){30}}\put(150,50){\line(2,-1){30}}\put(150,50){\vector(-2,-1){30}}
\put(150,50){\vector(-2,1){30}}
\dashline[3]{3}(180,65)(195,65)\put(195,65){\vector(1,0){5}}
\dashline[3]{3}(180,35)(195,35)\put(195,35){\vector(1,0){5}}
\put(170,52){\makebox(0,0){$\vdots$}}
\put(130,52){\makebox(0,0){$\vdots$}}
\put(160,60){\makebox(0,0){\tiny{$d_1$}}}\put(160,40){\makebox(0,0){\tiny{$d_n$}}}
\put(215,65){\makebox(0,0){\tiny{$i_1-1$}}}
\put(215,35){\makebox(0,0){\tiny{$i_n-1$}}}

\put(90,20){\makebox(0,0){\tiny{$r$ arrowheads}}}
\put(100,10){\makebox(0,0){\tiny{which might be doublearrows}}}

\put(230,20){\makebox(0,0){\tiny{$n$ boundary vertices}}}

\end{picture}

\vspace{3mm}

\noindent
Then the decorations $i_1-1,\dots,i_n-1$ of the
dashed arrows at these boundary vertices are subject to the
following restrictions provided that $r=1$ or $r=2$.
\begin{itemize}
\item
$\mathbf  {r=1:}$
if $d_\ell | i_\ell$ for at least $n-1$ indexes
 $\ell\in \{1,\ldots,n\}$, then $i_\ell=d_\ell$ for at least $n-1$
 indexes $\ell\in \{1,\ldots,n\}$;
 % (maybe not the same ones);
\item
$\mathbf  {r=2:}$
 if $d_\ell | i_\ell$ for all  $1\leq \ell\leq n$,
then $i_\ell=d_\ell$ for all these indexes $\ell$.
\end{itemize}
\item\label{allow.3}
For arbitrary $\Gamma$  we require that the induced decorations
on each  star-shaped subdiagram of $\Gamma$, obtained
after repeated splicing as in (\ref{be:fo}), satisfy the
restrictions (\ref{allow.2}).
\end{enumerate}
\end{definition}

\begin{remark}\labelpar{rem:allowed}
(a) In (\ref{allow.2}) there are thus no conditions on the decorations $i_a$
for the arrowheads given by $a\in \calA_F$, that is, those
associated to the other $r$ edges.

(b) The value $i_\ell=1$ is possible. It corresponds to no dashed
arrow, or formally to a dashed arrow with decoration zero. Also,
for the boundary vertices from the right, if $d_\ell=1$ for some $\ell$,
and the corresponding leg is not represented as above, but with its
minimal diagram as in (\ref{ss:3}), then the above definition applies
for these dashed arrowheads too (with $d_\ell=1$).

(c) One can formulate the restrictions in (\ref{allow.2}) simultaneously for
all $r$ as follows:
\begin{quote}
\noindent {\it if \, $d_\ell | i_\ell$ \, for at least $n+r-2$ indexes
$\ell\in\{1,\ldots,n\}$, then $i_\ell=d_\ell$ for at least $n+r-2$ of the
indexes $\ell$.} (This assumption is empty if  $r\geq 3$.)
\end{quote}

(d) We assumed implicitly in (\ref{allow.2}) that $n\geq 1$. When $n=0$ the
conditions are empty.

(e) A priori it is not clear at all that there exist allowed $W$ on a
general graph $\Gamma(F)$. We will construct plenty of them later.
\end{remark}

\begin{definition}\labelpar{def:allowedform}
Let $(X,0)$ be an IHS surface germ, and $F'$ a (non-zero) effective
Weil divisor on it. A Weil divisor $W'$ of $(X,0)$  is {\it allowed} for the
pair $(X,F')$ if there exists an embedded resolution
$\pi:\tilde{X}\to X$ of $F'$ such that the diagram $\GaxFW$ is
allowed.
\end{definition}

This notion is well defined, in the sense that it is invariant
under \lq extra\rq\ blowing-ups, as shown below.

\begin{proposition}\labelpar{prop:blowup}
We use the notation from (\ref{def:allowedform}). Suppose that the
diagram $\Gamma:=\GaxFW$ is allowed. Let
$h:\tilde{X_1}\to\tilde{X}$ be a blowing-up in some point of
$\tilde{X}$. Then the diagram $\Gamma_1:=\Gamma_{\pi\circ
h}(X,F_1,W_1)$ (with obvious notations) is also allowed.
\end{proposition}

\begin{proof}
We only have to investigate the spliced star-shaped subgraphs of
$\Gamma_1$ that are new or different with respect to $\Gamma$. If
the centre $P$ of the blowing-up $h$ is either a point of a
boundary curve, or an intersection point of two components
(exceptional or strict transform), we are done because then
$\Gamma_1=\Gamma$. We are left with the following two cases for
$P$.

\vspace{1mm}

\noindent {\it Case 1.} \ $P$ is  a point of $E_j^\circ$ (that is,
a generic point of $E_j$), where the vertex corresponding to $E_j$
has valency 2 in $G_\pi(X,F,W)$ (so it does not occur explicitly
in $\Gamma$):

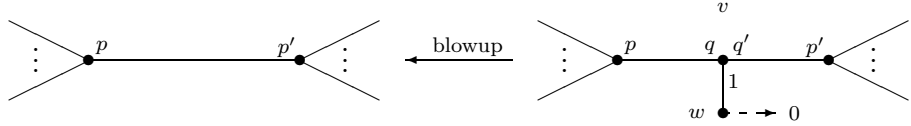
\begin{picture}(100,55)(-25,-10)

\put(40,20){\circle*{4}} %\put(80,20){\circle*{4}}
\put(120,20){\circle*{4}} \put(40,20){\line(1,0){80}}
\put(40,20){\line(-2,-1){30}} \put(40,20){\line(-2, 1){30}}
\put(20,23){\makebox(0,0){$\vdots$}} \put(120,20){\line(2,-1){30}}
\put(120,20){\line(2, 1){30}}
\put(137,23){\makebox(0,0){$\vdots$}}

\put(45,25){\makebox(0,0){\tiny{$p$}}}
\put(115,25){\makebox(0,0){\tiny{$p'$}}}
%\put(55,26){\makebox(0,0){$d$}} \put(105,26){\makebox(0,0){$d'$}}
%\put(138,35) {\makebox(0,0){$d'_1$}}
%\put(138,5){\makebox(0,0){$d'_{n'}$}} \put(25,35)
%{\makebox(0,0){$d_1$}} \put(25,5){\makebox(0,0){$d_n$}}
%\put(80,-10){\makebox(0,0){$e$}}

\put(200,20){\vector(-1,0){40}}
\put(170,25){\makebox(0,0)[l]{\tiny{blowup}}}

\put(245,25){\makebox(0,0){\tiny{$p$}}}
\put(315,25){\makebox(0,0){\tiny{$p'$}}}
\put(275,25){\makebox(0,0){\tiny{$q$}}}
\put(287,26){\makebox(0,0){\tiny{$q'$}}}

\put(240,20){\circle*{4}}
\put(280,20){\circle*{4}}\put(280,0){\circle*{4}}
\put(320,20){\circle*{4}}
\put(240,20){\line(1,0){80}}\put(280,20){\line(0,-1){20}}
\put(240,20){\line(-2,-1){30}} \put(240,20){\line(-2, 1){30}}
\put(220,23){\makebox(0,0){$\vdots$}}
\put(320,20){\line(2,-1){30}} \put(320,20){\line(2, 1){30}}
\put(337,23){\makebox(0,0){$\vdots$}}
\put(284,12){\makebox(0,0){\tiny{1}}}
\put(307,0){\makebox(0,0){\tiny{0}}}
\put(280,40){\makebox(0,0){\tiny{$v$}}}
\put(270,0){\makebox(0,0){\tiny{$w$}}}
%\put(250,26){\makebox(0,0){$d$}} \put(313,26){\makebox(0,0){$d'$}}
%\put(338,35) {\makebox(0,0){$d'_1$}}
%\put(338,5){\makebox(0,0){$d'_{n'}$}} \put(225,35)
%{\makebox(0,0){$d_1$}} \put(225,5){\makebox(0,0){$d_n$}}
\dashline[3]{3}(280,0)(295,0)\put(295,0){\vector(1,0){5}}
\end{picture}

\noindent (Above  we did not insert the information
about $F$; and one of the nodes of the diagram before blowup can
be replaced by a boundary vertex with or without dashed arrows.)
%its arrowheads might divide the above cases in further subcases (see below).

It is not difficult to verify that the spliced star-shaped
subgraphs around the `old' nodes  in $\Gamma$ and $\Gamma_1$ are
the same. Moreover,  the new spliced star-shaped subgraph of
$\Gamma_1$ around $v$ satisfies the definition of allowedness.
Indeed, in both cases $r=1$ or  $r=2$, the fact that the
decorations associated to the boundary vertex $w$ satisfy
$d_1=i_1=1$ finishes the verification.

\vspace{1mm}

\noindent {\it Case 2.} \ $P$ is
a point of $E_v^\circ$, where $v$ is a node in
$\Gamma$.

In this case the only novelty in $\Gamma_1$ is an extra edge  at the
node $v$ supporting a boundary vertex, again with edge decoration $d_\ell=1$ and associated number
$i_\ell=1$. This again does not affect the allowedness condition for
the star-shaped subgraph around $v$.
\end{proof}

\begin{remark}\labelpar{re:blowup}
It is possible that a divisor $W'$ on $(X,0)$ is not allowed in
the diagram associated with the {\it minimal} embedded resolution
$\pi$ of $(X,F')$, but is allowed in some  $\GaxFW$ associated
with some non--minimal $\pi$. Consider for example the situation

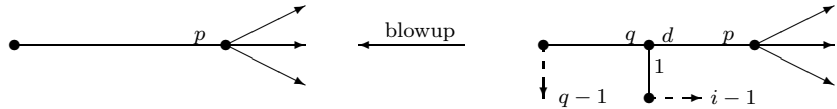
\begin{picture}(100,55)(-25,-10)

\put(40,20){\circle*{4}} %\put(80,20){\circle*{4}}
\put(120,20){\circle*{4}} \put(40,20){\vector(1,0){110}}
%\put(40,20){\line(-2,-1){30}} \put(40,20){\line(-2, 1){30}}
%\put(20,23){\makebox(0,0){$\vdots$}}
\put(120,20){\vector(2,-1){30}}
\put(120,20){\vector(2, 1){30}}
\put(110,23){\makebox(0,0){\tiny{$p$}}}
%\put(42,-10){\makebox(0,0){$v_L$}}
%\put(118,-10){\makebox(0,0){$v_R$}}
%\put(55,26){\makebox(0,0){$d$}} \put(105,26){\makebox(0,0){$d'$}}
%\put(138,35) {\makebox(0,0){$d'_1$}}
%\put(138,5){\makebox(0,0){$d'_{n'}$}} \put(25,35)
%{\makebox(0,0){$d_1$}} \put(25,5){\makebox(0,0){$d_n$}}
%\put(80,-10){\makebox(0,0){$e$}}

\put(210,20){\vector(-1,0){40}}
\put(180,25){\makebox(0,0)[l]{\tiny{blowup}}}

\put(310,23){\makebox(0,0){\tiny{$p$}}}
\put(240,20){\circle*{4}}
\put(280,20){\circle*{4}}\put(280,0){\circle*{4}}
\put(320,20){\circle*{4}}
\put(240,20){\vector(1,0){110}}\put(280,20){\line(0,-1){20}}
%\put(240,20){\line(-2,-1){30}} \put(240,20){\line(-2, 1){30}}
%\put(220,23){\makebox(0,0){$\vdots$}}
\put(320,20){\vector(2,-1){30}} \put(320,20){\vector(2, 1){30}}
%\put(337,23){\makebox(0,0){$\vdots$}}
\put(284,12){\makebox(0,0){\tiny{1}}}
\put(312,0){\makebox(0,0){\tiny{$i-1$}}}\put(255,0){\makebox(0,0){\tiny{$q-1$}}}
\put(273,23){\makebox(0,0){\tiny{$q$}}}\put(287,24){\makebox(0,0){\tiny{$d$}}}
%\put(270,0){\makebox(0,0){\tiny{$w$}}}
%\put(250,26){\makebox(0,0){$d$}} \put(313,26){\makebox(0,0){$d'$}}
%\put(338,35) {\makebox(0,0){$d'_1$}}
%\put(338,5){\makebox(0,0){$d'_{n'}$}} \put(225,35)
%{\makebox(0,0){$d_1$}} \put(225,5){\makebox(0,0){$d_n$}}
\dashline[3]{3}(280,0)(295,0)\put(295,0){\vector(1,0){5}}
\dashline[3]{3}(240,20)(240,5)\put(240,5){\vector(0,-1){5}}
\end{picture}

\noindent where $\pi$ is obtained from the minimal embedded
resolution $\pi_{\min}$ by composing with one blowing-up, and
$i\in\Z_{>1}$. The component of the strict transform of $W'$ with
multiplicity $i-1$ intersects the exceptional divisor of
$\pi_{\min}$ in a component $E_j$ of valency 2, and this was not
permitted, cf. (\ref{ss:3}).
\end{remark}

\begin{remark}\labelpar{re:REMARK}
 Assume that $F$ has only one arrowhead with multiplicity 1. Then for any $W$, the
dashed arrowhead with multiplicity $i-1$ which doubles the
arrowhead of $F$ has an almost irrelevant geometric contribution.
Indeed, its only effect is the following: in any ratio $\nu_v/N_v$
it has a global integral ($i-1$)-shift. In particular, in such a
situation (having connections with monodromy in mind), we might take $i-1=0$ without restricting the
generality of the discussion.
\end{remark}

\begin{example}\labelpar{ex:EXAMPLE4}
Let us continue the discussion of Example~\ref{ex:EXAMPLE}.
Having in mind Remark~\ref{re:REMARK},
the general form $W$ will have the following dashed arrowheads:

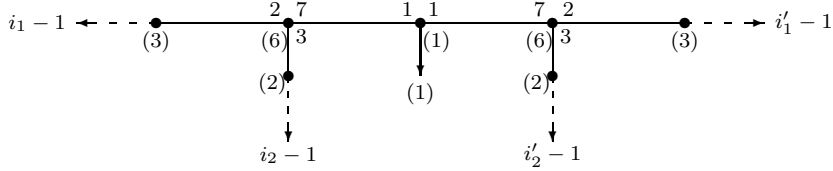
\begin{picture}(400,80)(0,-30)

\put(100,25){\circle*{4}} \put(150,25){\circle*{4}}
\put(200,25){\circle*{4}} \put(250,25){\circle*{4}}
\put(300,25){\circle*{4}} \put(150,5){\circle*{4}}
\put(250,5){\circle*{4}} \put(100,25){\line(1,0){200}}
\put(150,25){\line(0,-1){20}} \put(200,25){\vector(0,-1){20}}
%\put(370,5){\circle*{4}}
\put(250,25){\line(0,-1){20}} \put(145,30){\makebox(0,0){\tiny{$2$}}}
\put(195,30){\makebox(0,0){\tiny{$1$}}}\put(155,30){\makebox(0,0){\tiny{$7$}}}
\put(245,30){\makebox(0,0){\tiny{$7$}}}\put(205,30){\makebox(0,0){\tiny{$1$}}}
%\put(375,35){\makebox(0,0){\tiny{$-13$}}}
%\put(375,-3){\makebox(0,0){\tiny{$(N)$}}}
\put(256,30){\makebox(0,0){\tiny{$2$}}}
%\put(425,35){\makebox(0,0){\tiny{$-2$}}}
\put(155,20){\makebox(0,0){\tiny{$3$}}}
\put(255,20){\makebox(0,0){\tiny{$3$}}}%\put(322,5){\makebox(0,0){$-7$}}

\put(100,18){\makebox(0,0){\tiny{$(3)$}}}
\put(145,18){\makebox(0,0){\tiny{$(6)$}}}
\put(206,18){\makebox(0,0){\tiny{$(1)$}}}
\put(245,18){\makebox(0,0){\tiny{$(6)$}}}
\put(300,18){\makebox(0,0){\tiny{$(3)$}}}
\put(144,2){\makebox(0,0){\tiny{$(2)$}}}
\put(200,-2){\makebox(0,0){\tiny{$(1)$}}}
\put(244,2){\makebox(0,0){\tiny{$(2)$}}}
\put(55,25){\makebox(0,0){\tiny{$i_1-1$}}}
\put(345,25){\makebox(0,0){\tiny{$i'_1-1$}}}
\put(150,-25){\makebox(0,0){\tiny{$i_2-1$}}}
\put(250,-25){\makebox(0,0){\tiny{$i'_2-1$}}}

\dashline[3]{3}(100,25)(75,25)\put(75,25){\vector(-1,0){5}}
\dashline[3]{3}(300,25)(325,25)\put(325,25){\vector(1,0){5}}
\dashline[3]{3}(150,5)(150,-15)\put(150,-15){\vector(0,-1){5}}
\dashline[3]{3}(250,5)(250,-15)\put(250,-15){\vector(0,-1){5}}
\end{picture}

\noindent The splice decomposition provides:

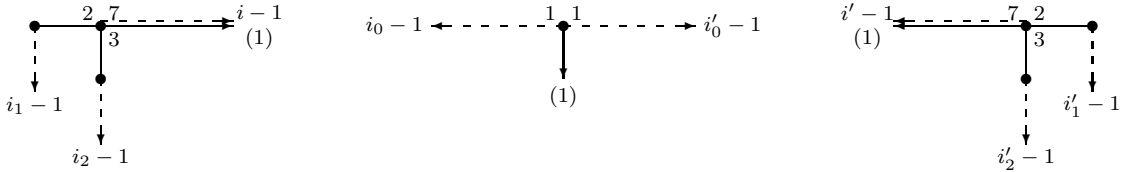
\begin{picture}(400,90)(0,-40)

\put(0,25){\circle*{4}} \put(25,25){\circle*{4}}
\put(200,25){\circle*{4}} \put(375,25){\circle*{4}}
\put(400,25){\circle*{4}} \put(25,5){\circle*{4}}
\put(375,5){\circle*{4}} \put(0,25){\vector(1,0){75}}
\put(400,25){\vector(-1,0){75}}
\put(25,25){\line(0,-1){20}} \put(200,25){\vector(0,-1){20}}
\put(375,25){\line(0,-1){20}} \put(20,30){\makebox(0,0){\tiny{$2$}}}
\put(195,30){\makebox(0,0){\tiny{$1$}}}\put(30,30){\makebox(0,0){\tiny{$7$}}}
\put(370,30){\makebox(0,0){\tiny{$7$}}}\put(205,30){\makebox(0,0){\tiny{$1$}}}
%\put(375,35){\makebox(0,0){\tiny{$-13$}}}
%\put(375,-3){\makebox(0,0){\tiny{$(N)$}}}
\put(380,30){\makebox(0,0){\tiny{$2$}}}
%\put(425,35){\makebox(0,0){\tiny{$-2$}}}
\put(30,20){\makebox(0,0){\tiny{$3$}}}
\put(380,20){\makebox(0,0){\tiny{$3$}}}%\put(322,5){\makebox(0,0){$-7$}}

%\put(0,18){\makebox(0,0){\tiny{$(3)$}}}
%\put(20,18){\makebox(0,0){\tiny{$(6)$}}}
%\put(206,18){\makebox(0,0){\tiny{$(1)$}}}
%\put(370,18){\makebox(0,0){\tiny{$(6)$}}}
%\put(400,18){\makebox(0,0){\tiny{$(3)$}}}
%\put(25,-2){\makebox(0,0){\tiny{$(2)$}}}
\put(200,-2){\makebox(0,0){\tiny{$(1)$}}}
%\put(375,-2){\makebox(0,0){\tiny{$(2)$}}}

%\put(25,45){\makebox(0,0){\tiny{$\nu_1=-13$}}}
%\put(200,45){\makebox(0,0){\tiny{$\nu_0=-2$}}}
%\put(375,45){\makebox(0,0){\tiny{$\nu_1'=-13$}}}

\dashline[3]{3}(25,27)(70,27)\put(70,27){\vector(1,0){5}}
\dashline[3]{3}(375,27)(330,27)\put(330,27){\vector(-1,0){5}}
\dashline[3]{3}(200,25)(155,25)\put(155,25){\vector(-1,0){5}}
\dashline[3]{3}(200,25)(245,25)\put(245,25){\vector(1,0){5}}

\put(85,20){\makebox(0,0){\tiny{$(1)$}}}
\put(85,30){\makebox(0,0){\tiny{$i-1$}}}
\put(136,25){\makebox(0,0){\tiny{$i_0-1$}}}
\put(264,25){\makebox(0,0){\tiny{$i_0'-1$}}}
\put(315,20){\makebox(0,0){\tiny{$(1)$}}}
\put(315,30){\makebox(0,0){\tiny{$i'-1$}}}

\put(0,-5){\makebox(0,0){\tiny{$i_1-1$}}}
\put(400,-5){\makebox(0,0){\tiny{$i'_1-1$}}}
\put(25,-25){\makebox(0,0){\tiny{$i_2-1$}}}
\put(375,-25){\makebox(0,0){\tiny{$i'_2-1$}}}

\dashline[3]{3}(0,25)(0,5)\put(0,5){\vector(0,-1){5}}
\dashline[3]{3}(400,25)(400,5)\put(400,5){\vector(0,-1){5}}
\dashline[3]{3}(25,5)(25,-15)\put(25,-15){\vector(0,-1){5}}
\dashline[3]{3}(375,5)(375,-15)\put(375,-15){\vector(0,-1){5}}
\end{picture}

\noindent In the above picture
$i-1=i_0'-1=-2+3(i_1'-1)+2(i_2'-1)$ and
$i'-1=i_0-1=-2+3(i_1-1)+2(i_2-1)$.

Assume that $W$ is allowed. This imposes the following numerical conditions.

\noindent $\bullet$ \ In the middle graph $\Gamma_0$ we impose: either
$3i_1+2i_2=7$ or  $3i_1'+2i_2'=7$.

\noindent $\bullet$ \ In $\Gamma_1$ one gets: if $2\mid i_1$ or $3\mid i_2$
then either $2=i_1$ or $3=i_2$. Note that if   $3i_1+2i_2=7$ then
$2\nmid i_1$ and $3\nmid i_2$. There is a symmetric restriction in
$\Gamma_1'$ too: if $2\mid i'_1$ or $3\mid i'_2$ then either
$2=i'_1$ or $3=i'_2$.

In particular, the zero form $W=0$ is {\it not allowed}.
\end{example}

\subsection{Restricting and extending allowed divisors}\labelpar{ss:extend}
Consider the splicing of a given diagram $\Gamma(F)$ along a
special edge $e$ as in (\ref{be:f}). A basic idea in the
definition of an allowed $W$ for $\Gamma(F)$ is that the induced
$W_L$ and $W_R$ should be allowed for $\Gamma_L(F_L)$ and
$\Gamma_R(F_R)$, respectively. This is almost clear from the
nature of the definitions.

There is potentially a problem when (say) $\calA_{F,L}=\emptyset$,
since then a new dashed arrowhead  at a boundary vertex of
$\Gamma_R(F_R)$  is created, and it could have associated
decoration $i'=0$, see the last picture in (\ref{be:fo}). Indeed,
in Definition~\ref{def:alloweddiagram}, part (3), we asked for
each star--shaped subdiagram  to satisfy condition (2), but we
didn't ask (1). In the next lemma we will verify that (1) will be
automatically satisfied.

\begin{lemma}\labelpar{lemma:nonzero}
When $\calA_{F,L}=\emptyset$ and $i_a\neq 0$ for all $a\in
\calA_{W,L}$, then $i'\neq 0$. In particular, allowed divisors on
a graph always \lq restrict\rq\ to allowed divisors on spliced
subdiagrams.
\end{lemma}

\begin{proof}
By induction on the number of nodes in $\Gamma_L$ it is sufficient
to prove that $i'\neq 0$ when $\Gamma_L$ is star-shaped. Let
$i_1-1,\dots,i_n-1$ be the multiplicities of the dashed arrowheads
along the edges with decorations $d_1,\dots,d_n$.
Denote $D:=\prod_{\ell=1}^n d_\ell$. Then, using
(\ref{eq:totali}),
$$i'=(1-n)D+\sum_{\ell=1}^n \frac D{d_\ell} i_\ell.$$
Suppose that $i'=0$. Then $d_\ell$
divides $\frac D{d_\ell} i_\ell$, hence divides $i_\ell$ too, for all $\ell$. By
the definition of allowedness we then know that $i_\ell=d_\ell$
for at least $n-1$ of the $i_\ell$; say for $i_2,\dots,i_n$. Thus
$$0=i'=(1-n)D+(n-1)D+ Di_1/d_1$$
and hence $i_1=0$, contradicting the assumptions.
 For the second
statement, note that $i_a=1$ for any $a\in {\calA}_F\setminus
{\calA}_W$, while $i_a\not=0$ for $a\in {\calA}_W\setminus
{\calA}_F$
 since $W$ is allowed.
\end{proof}

A crucial question  in our setting  is the converse: can we \lq
extend\rq\ an allowed $W^\flat$ on $\Gamma_R(F_R)$ to $\Gamma(F)$,
that is, can we construct an allowed $W$ on $\Gamma(F)$ for which
$W_R=W^\flat$?

It is enough to study this question when $\Gamma_L(F_L)$ is
star-shaped, since we can then proceed further inductively.

\begin{proposition}\labelpar{prop:extend}
Let $\Gamma(F)$ be an arbitrary diagram as in (\ref{ss:2}).  Splice
$\Gamma(F)$ along a special edge $e$ such that $\Gamma_L(F_L)$ is
star-shaped. Then the  map
$$
\Psi:\{\mbox{allowed }W\mbox{ for }\Gamma(F)\} \to \{\mbox{allowed
}W^\flat\mbox{ for }\Gamma_R(F_R)\} : W\mapsto W_R
$$
is surjective if there is at least one arrowhead in $\Gamma(F)$ on
the right of $e$, or the number of arrowheads on the left of $e$
is different from one and two.
\end{proposition}

\begin{proof}
Take an allowed $W^\flat$ for $\Gamma_R(F_R)$. Let $i'-1$ be the
decoration of the dashed arrowhead at the left of $v_R$.

Let $i_1-1,\dots, i_n-1 \ (n\geq2)$ be the (still to be
determined) decorations of the dashed arrowheads of some $W$, that
we want to construct in the pre-image of $W^\flat$ by $\Psi$. Let
$j-1$ be the decoration of the new induced dashed arrowhead at
$v_L$ for $\Gamma_L(F_L)$. Note that $j$ is {\it fixed} in the
sense that it is uniquely determined in terms of $W^\flat$ by the
formula (\ref{eq:totali}). (And, when there are no arrowheads on
the right of $e$, we have that $j\neq 0$ by the argument in the
proof of Lemma \ref{lemma:nonzero}.)

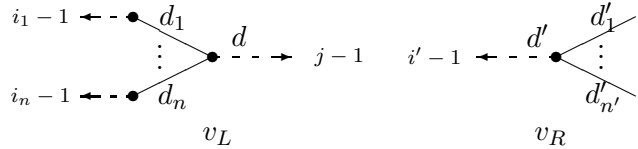
\begin{picture}(200,75)(85,-20)

%\put(40,20){\circle*{4}} %\put(80,20){\circle*{4}}
%\put(120,20){\circle*{4}} \put(40,20){\line(1,0){80}}
%\put(40,20){\line(-2,-1){30}} \put(40,20){\line(-2, 1){30}}
%\put(20,23){\makebox(0,0){$\vdots$}} \put(120,20){\line(2,-1){30}}
%\put(120,20){\line(2, 1){30}}
%\put(137,23){\makebox(0,0){$\vdots$}}
%\put(42,-10){\makebox(0,0){$v_L$}}
%\put(118,-10){\makebox(0,0){$v_R$}}
%\put(55,26){\makebox(0,0){$d$}} \put(105,26){\makebox(0,0){$d'$}}
%\put(138,35) {\makebox(0,0){$d'_1$}}
%\put(138,5){\makebox(0,0){$d'_{n'}$}} \put(25,35)
%{\makebox(0,0){$d_1$}} \put(25,5){\makebox(0,0){$d_n$}}
%\put(42,10){\makebox(0,0){\tiny{$(N)$}}}
%\put(118,10){\makebox(0,0){\tiny{$(N')$}}}

%\put(160,20){\vector(1,0){40}}
%\put(160,25){\makebox(0,0)[l]{splicing}}

\put(240,20){\circle*{4}} %\put(270,20){\circle*{4}} \put(340,20){\circle*{4}}
\put(370,20){\circle*{4}}
%\put(240,20){\vector(1,0){30}}
%\put(370,20){\vector(-1,0){30}}
\put(240,20){\line(-2,-1){30}} \put(240,20){\line(-2, 1){30}}
\put(220,23){\makebox(0,0){$\vdots$}}
\put(370,20){\line(2,-1){30}} \put(370,20){\line(2, 1){30}}
\put(387,23){\makebox(0,0){$\vdots$}}
\put(242,-10){\makebox(0,0){$v_L$}}
\put(368,-10){\makebox(0,0){$v_R$}}
\put(250,28){\makebox(0,0){$d$}} \put(363,28){\makebox(0,0){$d'$}}
\put(388,35) {\makebox(0,0){$d'_1$}}
\put(388,5){\makebox(0,0){$d'_{n'}$}} \put(225,35)
{\makebox(0,0){$d_1$}} \put(225,5){\makebox(0,0){$d_n$}}
%\put(242,10){\makebox(0,0){\tiny{$(N)$}}}
%\put(368,10){\makebox(0,0){\tiny{$(N')$}}}
\put(288,20){\makebox(0,0){\tiny{$j-1$}}}
\put(324,20){\makebox(0,0){\tiny{$i'-1$}}}
\put(175,5){\makebox(0,0){\tiny{$i_n-1$}}}
\put(175,35){\makebox(0,0){\tiny{$i_1-1$}}}

\dashline[3]{3}(240,20)(265,20)\put(265,20){\vector(1,0){5}}
\dashline[3]{3}(370,20)(345,20)\put(345,20){\vector(-1,0){5}}

\put(210,35){\circle*{4}}\put(210,5){\circle*{4}}
\dashline[3]{3}(210,35)(195,35)\put(195,35){\vector(-1,0){5}}
\dashline[3]{3}(210,5)(195,5)\put(195,5){\vector(-1,0){5}}
\end{picture}

\noindent (In this diagram we did not insert the information regarding the divisor
$F$, the corresponding arrowhead positions might determine different cases, see below.)

Denote $D:=\prod_{\ell=1}^n d_\ell$. Then, cf. (\ref{eq:totali}),
we are searching for $i_1,\ldots, i_n$ with
\begin{equation}\label{eq:ext}
i'=(1-n)D+ \sum_{\ell=1}^n \frac D{d_\ell} i_\ell.
\end{equation}
Since $\gcd_\ell \{ D/d_\ell\}=1$, we know that there exist
$i_1,\dots,i_n \in \Z$ satisfying (\ref{eq:ext}). We have to
verify that this can be done compatibly with the restrictions on
the $i_\ell$ and $j$, when {\it exactly one or two} (ordinary)
arrowheads are among the legs in $\Gamma_L(F_L)$.
%In the picture above we did not draw the ordinary arrowheads.

Note first that, by (\ref{eq:ext}), we have for any
$\ell=1,\dots,n$ that $d_\ell|i'$ if and only if $d_\ell|i_\ell$.
With the given assumptions on the arrowheads in $\Gamma(F)$, we
encounter two cases.

\smallskip
(1) There is exactly one arrowhead on $\Gamma_L(F_L)$, and it
coincides with the dashed $(j-1)$--arrowhead. Then we suppose that
$d_\ell|i'$ for at least $n-1$ of the $i_\ell$, say for
$i_1,\dots, i_{n-1}$. (Otherwise nothing has to be verified.)

(2) There are exactly two arrowheads on $\Gamma_L(F_L)$, and they
coincide with the dashed $(j-1)$-- and $(i_n-1)$--arrowheads. Then
we suppose that $d_\ell|i'$ for $\ell=1,\dots,n-1$.

\smallskip
\noindent In each of these cases we take $i_\ell=d_\ell$ for
$\ell=1,\dots,n-1$ and then, in order to satisfy (\ref{eq:ext}),
we take $i_n$ given by $i'=Di_n/d_n$. This way we thus
constructed an allowed $W$ for $\Gamma(F)$ that \lq restricts\rq\
to $W^\flat$.
\end{proof}

\begin{remark}\labelpar{re:extend} \ (a)
Consider in the proof above the excluded cases.
First assume that there are  exactly two
arrowheads on $\Gamma_L(F_L)$, namely when they coincide with the
dashed $(i_{n-1}-1)$-- and $(i_n-1)$--arrowheads. To verify
allowedness we have to suppose that $d|j$ and $d_\ell|i'$ for
$\ell=1,\dots,n-2$. Then, in order to get an allowed extension
 we should have that $d=j$, but this is
not true in general.

Similarly, assume that  $\Gamma_L(F_L)$ has
exactly one arrowhead which
coincides with the dashed $(i_n-1)$--arrowhead. Then in the situation
$d|j$, $d_\ell| i'$ for $\ell\leq n-2$ but $d_{n-1}\nmid i'$
one gets an allowed extension only if $d=j$.

These cases motivate the restrictions of (\ref{prop:extend}).

\noindent
This discussion shows the following {\bf Addendum to
Proposition~\ref{prop:extend}}: {\it with the above notations, in
the following cases an extension is still possible:}
\begin{equation}\label{eq:addendum} \begin{array}{l}
\mbox{{\it (a) \ \  either $d\nmid j$ or $d=j$;} } \\
\mbox{{\it (b) \ \  the only arrowhead
coincides with the dashed $(i_n-1)$--arrowhead }} \\
\mbox{{\it \hspace{8mm} and
$d_\ell| i'$ for $\ell\leq n-1$. }}\end{array}
\end{equation}

(b) Let $\Gamma(F)$ be an arbitrary diagram. Let $\Gamma_\calA$ be
that minimal connected subdiagram of $\Gamma$ which contains those
nodes which  either support at least one arrowhead of $F$, or sit on a (geodesic)
path connecting two arrowheads of $F$, and those boundary vertices
which are supported by these nodes. The connected components of
$\Gamma\setminus \Gamma_\calA$ are denoted by
$\{\Gamma_j\}_{j\in\calJ}$, and each $\Gamma_j$ is connected to
$\Gamma_\calA$ at the vertex $v_j$ of $\Gamma_\calA$. (For
example, if $\Gamma$ is the minimal
 diagram of a plane curve singularity, then $|\calJ|\leq 1$.) Then one has the following
 facts.

\vspace{1mm}

$\bullet$ \ Any allowed $P$--divisor supported on a star--shaped sub--diagram
centered at any node of $\Gamma_\calA$ can be extended by (\ref{prop:extend})
to an allowed $P$--divisor of the whole $\Gamma(F)$. In particular,
\begin{equation*}
\mbox{{\it any $\Gamma(F)$ always admits allowed divisors $W$.}}
\end{equation*}

$\bullet$ \ Any allowed $P$--divisor  on a star--shaped diagram
centered at a vertex $v$ in $\Gamma_j$ can be extended `away from
$v_j$'. In order to extend it `in the direction of $v_j$'
one needs some extra conditions (like in (\ref{eq:addendum})).
\end{remark}

\section{Allowed forms/divisors induce eigenvalues}\labelpar{s:4}

\subsection{} In this section we prove that the poles of the topological
zeta function associated to any $\Gamma(F)$ and allowed divisor
$W$ provide eigenvalues for the monodromy zeta function.
%We start with the following lemma.

\begin{lemma}\labelpar{le:nopole}
Let $S$ be the star-shaped diagram  as in Definition
\ref{def:alloweddiagram} with $r=1$ or $r=2$ ordinary arrowheads,
equipped with decorations as below. We assume that $i_\ell=d_\ell$
for $\ell=1,\dots,n-1$ (and $i_n\neq 0$) if $r=1$, and
$i_\ell=d_\ell$ for $\ell=1,\dots,n$ if $r=2$ . Thus $W$ is
allowed.  When $r=2$, we assume also that $-\frac{\nu_v}{N_v}$ is
not a pole of order $2$ of $Z(S)$, that is, $\frac{\nu_v}{N_v} \neq
\frac{k_1}{N_1}$ and $\frac{\nu_v}{N_v} \neq \frac{k_2}{N_2}$.
Then, in all the above situations,
$-\frac{\nu_v}{N_v}$ is not a pole of $Z(S)$.
\end{lemma}

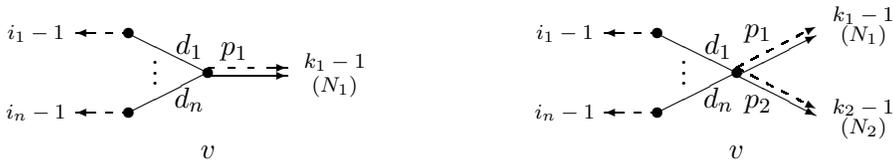
\begin{picture}(200,75)(145,-20)
\put(240,20){\circle*{4}}\put(210,35){\circle*{4}}\put(210,5){\circle*{4}}
\put(240,19){\vector(1,0){30}}
\put(240,20){\line(-2,-1){30}} \put(240,20){\line(-2, 1){30}}
\put(220,23){\makebox(0,0){$\vdots$}}
\put(240,-10){\makebox(0,0){$v$}}
\put(250,28){\makebox(0,0){$p_1$}}
 \put(233,30){\makebox(0,0){$d_1$}} \put(233,10){\makebox(0,0){$d_n$}}
\put(288,15){\makebox(0,0){\tiny{$(N_1)$}}}
\put(288,24){\makebox(0,0){\tiny{$k_1-1$}}}
\put(175,5){\makebox(0,0){\tiny{$i_n-1$}}}
\put(175,35){\makebox(0,0){\tiny{$i_1-1$}}}

\dashline[3]{3}(240,22)(265,22)\put(265,22){\vector(1,0){5}}
\dashline[3]{3}(210,35)(195,35)\put(195,35){\vector(-1,0){5}}
\dashline[3]{3}(210,5)(195,5)\put(195,5){\vector(-1,0){5}}

\put(440,20){\circle*{4}}\put(410,35){\circle*{4}}\put(410,5){\circle*{4}}
\put(440,19){\vector(2,1){30}}\put(440,19){\vector(2,-1){30}}
\put(440,20){\line(-2,-1){30}} \put(440,20){\line(-2, 1){30}}
\put(420,23){\makebox(0,0){$\vdots$}}
\put(440,-10){\makebox(0,0){$v$}}
\put(448,35){\makebox(0,0){$p_1$}}\put(448,8){\makebox(0,0){$p_2$}}
 \put(433,30){\makebox(0,0){$d_1$}} \put(433,10){\makebox(0,0){$d_n$}}
\put(488,34){\makebox(0,0){\tiny{$(N_1)$}}}
\put(488,42){\makebox(0,0){\tiny{$k_1-1$}}}
\put(488,-1){\makebox(0,0){\tiny{$(N_2)$}}}
\put(488,7){\makebox(0,0){\tiny{$k_2-1$}}}
\put(375,5){\makebox(0,0){\tiny{$i_n-1$}}}
\put(375,35){\makebox(0,0){\tiny{$i_1-1$}}}

\dashline[3]{3}(440,22)(465,35)\put(465,35){\vector(2,1){5}}
\dashline[3]{3}(440,22)(465,9)\put(465,9){\vector(2,-1){5}}
\dashline[3]{3}(410,35)(395,35)\put(395,35){\vector(-1,0){5}}
\dashline[3]{3}(410,5)(395,5)\put(395,5){\vector(-1,0){5}}
\end{picture}

\begin{proof}
We consider first the case $r=2$. By (\ref{eq:nu}) and
(\ref{eq:mult}) we have, with $D=\prod_\ell d_\ell$,
\begin{equation}\label{eq:nopole1}
\nu_v=\sum_{\ell=1}^n \frac{Dp_1p_2}{d_\ell}i_\ell -
nDp_1p_2+Dp_2k_1+Dp_1k_2 \quad\mbox{and}\quad N_v=
D(p_2N_1+p_1N_2),
\end{equation}
respectively. With our assumptions this simplifies to
\begin{equation}\label{eq:nopole2}
\nu_v= D(p_2k_1+p_1k_2) \quad\mbox{and}\quad N_v=
D(p_2N_1+p_1N_2).
\end{equation}
The residue of $-\frac{\nu_v}{N_v}$ is (up to a factor $N_v$)
$$
\Big( -n+\sum_{\ell=1}^n \frac{d_\ell}{i_\ell}+
\frac{p_1}{k_1+sN_1}+\frac{p_2}{k_2+sN_2} \Big)
\Big|_{(s=-\frac{\nu_v}{N_v})} = \frac{p_1}{k_1 -\frac{\nu_v}{N_v}
N_1}+\frac{p_2}{k_2 -\frac{\nu_v}{N_v} N_2}.
$$
This expression being zero is equivalent to
$\frac{\nu_v}{N_v}=\frac{p_2k_1+p_1k_2}{p_2N_1+p_1N_2}$, which
follows  from (\ref{eq:nopole2}).

When $r=1$, we have by (\ref{eq:nu}) and (\ref{eq:mult}) that
\begin{equation}\label{eq:nopole3}
\nu_v=\sum_{\ell=1}^n \frac{Dp_1}{d_\ell}i_\ell - (n-1)Dp_1+Dk_1
\quad\mbox{and}\quad N_v= DN_1,
\end{equation}
simplifying with our assumptions to
\begin{equation}\label{eq:nopole4}
\nu_v= D(\frac{p_1}{d_n}i_n + k_1) \quad\mbox{and}\quad N_v= DN_1.
\end{equation}
Thus $\frac{\nu_v}{N_v}\neq\frac{k_1}{N_1}$ (since $i_n\not=0$)
and $-\frac{\nu_v}{N_v}$ is not a pole of order 2. The fact that
its residue is zero is a similar easy computation as above.
\end{proof}

\begin{theorem}\labelpar{thm:monconj} Let $(X,0)$ be an IHS germ,
and $f$ an analytic function on $X$. Let $W$ be an arbitrary
allowed divisor for $(X,f)$. If $s_0$ is a pole of the topological
zeta function $Z(f, W; s)$, then $\exp(2\pi is_0)$ is a monodromy
eigenvalue of $f$ at some point of $\{f=0\}$ (in one of the homology groups).
\end{theorem}

\begin{proof} Fix an embedded resolution $\pi$ of $f$ such that
$W$ is allowed for it. We use the usual notation associated to
$\Gamma:=\Gamma_\pi(X,f,W)$. We consider three subcases.

\smallskip
(1) {\it There is a component $F_a$ of $\{f=0\}$ such that
$s_0=-\frac{i_a}{N_a}$.} Then $\exp(2\pi is_0)$ is an eigenvalue
of $f$ at a point of $F_a$  close to $0$ (since all $N_a$-th roots
of unity are eigenvalues at such a point).

\smallskip
(2) {\it There is no $F_a$ as in (1) and let $s_0$ be  a pole of order
$1$ of  $Z(f, W; s)$.} Then $s_0=-\frac{\nu_v}{N_v}$ for some node $v$, such that the
contribution of $v$ to the residue of $Z(f, W; s)$ at $s_0$ is
non-zero. Consider after repeated splicing the induced
star-shaped diagram  $S$ around $v$; by Proposition
(\ref{prop:splicezeta}) we know that the residue of $Z(S)$ at $s_0$
is exactly this contribution.

\begin{picture}(200,80)(265,-20)
\put(440,20){\circle*{4}}\put(410,35){\circle*{4}}\put(410,5){\circle*{4}}
\put(440,19){\vector(2,1){30}}\put(440,19){\vector(2,-1){30}}
\put(440,20){\line(-2,-1){30}} \put(440,20){\line(-2, 1){30}}
\put(460,23){\makebox(0,0){$\vdots$}}\put(420,23){\makebox(0,0){$\vdots$}}
\put(440,-10){\makebox(0,0){$v$}}
\put(448,35){\makebox(0,0){$p_1$}}\put(448,8){\makebox(0,0){$p_r$}}
 \put(433,30){\makebox(0,0){$d_1$}} \put(433,10){\makebox(0,0){$d_n$}}
\put(488,34){\makebox(0,0){\tiny{$(N_1)$}}}
\put(488,42){\makebox(0,0){\tiny{$k_1-1$}}}
\put(488,-1){\makebox(0,0){\tiny{$(N_r)$}}}
\put(488,7){\makebox(0,0){\tiny{$k_r-1$}}}
\put(375,5){\makebox(0,0){\tiny{$i_n-1$}}}
\put(375,35){\makebox(0,0){\tiny{$i_1-1$}}}
\put(560,20){\makebox(0,0){($r$ arrowheads)}}
\dashline[3]{3}(440,22)(465,35)\put(465,35){\vector(2,1){5}}
\dashline[3]{3}(440,22)(465,9)\put(465,9){\vector(2,-1){5}}
\dashline[3]{3}(410,35)(395,35)\put(395,35){\vector(-1,0){5}}
\dashline[3]{3}(410,5)(395,5)\put(395,5){\vector(-1,0){5}}
\end{picture}

From (\ref{ss:splicingalex}) we compute
\begin{equation}\label{eq:zeta1}
\zeta_S(t)= \frac{(t^{N_v}-1)^{r+n-2}}{\prod_{\ell=1}^n
(t^{N_v/d_\ell}-1)},
\end{equation}
and by (\ref{eq:nu}), denoting $D:=\prod_{\ell=1}^n d_\ell$, we
have
%\marginpar{(...)}
\begin{equation}\label{eq:nu1}
\nu_v=\sum_{\ell=1}^n \frac D{d_\ell}(\prod_{j=1}^r p_j)i_\ell +
D\cdot (\mbox{some integer}).
\end{equation}

We distinguish three possibilities for $r$ in order to show that
$\exp(2\pi is_0)$ is always a root of $\Lambda_S(t)$.

\vspace{2mm}

$\bullet$   $\mathbf{(r\geq 3)}$ \ Then, via (\ref{eq:zeta1}),
 $\exp(2\pi is_0)$ is clearly a root of
$\Lambda_S(t)=\zeta_S(t)$.

\smallskip
$\bullet$  $\mathbf{(r=2)}$ \ Suppose that $\exp(2\pi is_0)$ is
not a  root of $\Lambda_S(t)=\zeta_S(t)$. Then, by
(\ref{eq:zeta1}), $\frac{\nu_v}{d_\ell}$ must be an integer for
all $\ell=1,\dots,n$. By (\ref{eq:nu1}) this is equivalent to
$d_\ell|i_\ell$ for all these $\ell$. By allowedness we conclude
that then $i_\ell=d_\ell$ for all $\ell$. Lemma \ref{le:nopole}
then contradicts that the residue at $s_0$ is non-zero.

\smallskip
$\bullet$  $\mathbf{(r=1)}$ \
 Suppose that $\exp(2\pi is_0)$ is not a root of
$\Lambda_S(t)=\Delta_S(t)$. Then, analogously,
 for at least $n-1$ of
the numbers $i_1,\dots,i_n$ we have that $\frac{\nu_v}{d_\ell}$
must be an integer, or, equivalently, $d_\ell|i_\ell$. Then by
allowedness  $i_\ell=d_\ell$ for $n-1$ of these numbers, and again
Lemma \ref{le:nopole}  contradicts that the residue at $s_0$ is
non-zero.

\smallskip
We conclude that $\exp(2\pi is_0)$ is indeed always a root of
$\Lambda_S(t)$,  hence of  $\Lambda_\Gamma(t)$, and thus that
$\exp(2\pi is_0)$ is a monodromy eigenvalue of $f$ at $0$.

\smallskip
(3) {\it There is no $F_a$ as in (1) and let $s_0$ be a pole of order
$2$.} This implies that $s_0=-\frac{\nu_v}{N_v}=-\frac{\nu_w}{N_w}$
for two nodes $v$ and $w$, connected by a special edge $e$.
Equivalently, cf. (\ref{prop:splicezeta}),
\begin{equation}\label{eq:s0}
s_0=-\frac{\nu_v}{N_v}=-\frac{k}{N}\end{equation}
 for the
central node $v$ and an arrowhead with decorations
$(N_a,i_a)=(N,k)$ in some star-shaped spliced subdiagram of
$\Gamma$.

Let $C$ be the connected part of $\Gamma$ containing $v$ and
consisting of nodes $w$ with $s_0=-\frac{\nu_w}{N_w}$ only. Either
at least one node $w$ in $C$ has at least three attached
(ordinary) arrowheads in the induced star-shaped subgraph $S_w$
after splicing, and then $\exp(2\pi is_0)$ is a root of
$\Lambda_{S_w}(t)=\zeta_{S_w}(t)$; or {\it all} nodes $w$ in $C$
have exactly one or two attached arrowheads in $S_w$. So we are left with
 this second possibility. We distinguish two subcases for
such a node $v$. For the two diagrams and notations, see just
before the proof of Lemma \ref{le:nopole}.

\smallskip
$\bullet$  $\mathbf{(r=1)}$ \
 Using (\ref{eq:mult})  and (\ref{eq:nu}),
see also (\ref{eq:nopole3}), we have from (\ref{eq:s0}) that
$$
\frac{\nu_v}{N_v}=\frac{\sum_{\ell=1}^n \frac{Dp_1}{d_\ell}i_\ell
- (n-1)Dp_1+Dk_1}{DN_1}  =\frac{k_1}{N_1}.
$$
This simplifies to $\sum_{\ell=1}^n \frac{D}{d_\ell}i_\ell -
(n-1)D=0$, implying that $d_\ell|i_\ell$ for all $\ell=1,\dots,n$.
By allowedness we then have $i_\ell=d_\ell$ for say
$\ell=2,\dots,n$. Hence the previous equality reduces to $i_1=0$,
contradicting that $W$ is allowed. Hence this case cannot occur.

\smallskip
$\bullet$  $\mathbf{(r=2)}$ \
 We may assume also that $v$ is an \lq
extremity\rq\ of $C$, that is, that $v$ is only connected to one
other node of $C$, say $\nu_v/N_v = k_1/N_1\not=k_2/N_2$.  Now
(\ref{eq:mult}) and (\ref{eq:nu}), see also (\ref{eq:nopole1}),
together with (\ref{eq:s0}) yield
$$
\frac{\nu_v}{N_v}=\frac{\sum_{\ell=1}^n
\frac{Dp_1p_2}{d_\ell}i_\ell -
nDp_1p_2+Dp_2k_1+Dp_1k_2}{D(p_2N_1+p_1N_2)} =\frac{k_1}{N_1}.
$$
{\it Suppose} that $\exp(2\pi is_0)$ is {\it not} a root of
$\Lambda_{S_v}(t)$. By the same arguments as in case (2)
we get that  $d_\ell|i_\ell$ for all $\ell=1,\dots,n$. By allowedness
we now have that $i_\ell=d_\ell$ for all $\ell$, and then the
previous equality reduces to
$$
\frac{p_2k_1+p_1k_2}{p_2N_1+p_1N_2} =\frac{k_1}{N_1}.
$$
This is equivalent to $\frac{k_1}{N_1}=\frac{k_2}{N_2} (=-s_0)$,
contradicting that $v$ is an extremity of $C$.

We conclude that $\exp(2\pi is_0)$ is a root of $\Lambda_{S_v}(t)$
for some node $v$, and hence also of $\Lambda_\Gamma(t)$.
\end{proof}

\begin{remark}\labelpar{re:F} Although we formulated the previous
Theorem~\ref{thm:monconj} for an analytic function $f$, it has a
purely combinatorial version (with  the same proof) valid for
diagrams.

 Start with a diagram $\Gamma(F)$ and set the possible `eigenvalues of the monodromies
 at different points and in different homologies':
 \begin{equation*}
 Eig:=\{\lambda\,:\, \Delta_1(\lambda)=0\}\cup
 \bigcup_{a\in\calA_F}\{\lambda\,:\, \lambda^{N_a}=1\}.
 \end{equation*}
 Then, for any allowed $W$ and pole $s_0$ of the zeta function
 $Z(F,W;s)$, we have that $\exp(2\pi is_0)$ belongs to $Eig$.
\end{remark}

\begin{example}\labelpar{ex:EXAMPLE5}
Let us continue the main Example~\ref{ex:EXAMPLE4} further.

 Using (\ref{ex:EXAMPLE3}) we get that the eigenvalues are the roots of $(t-1)(t^2-t+1)$.
Theorem~\ref{thm:monconj} says that if $W$ is allowed and $s_0$ is
a pole of $Z(s)$, then $\exp(2\pi is_0)$ is 1 or a primitive 6-th
root of unity.

(a) First we show  that we can find easily non--allowed forms $W$
such that the corresponding pole will not provide an eigenvalue.
This proves that some kind of restriction regarding the divisors
$W$ is necessary.

Consider in (\ref{ex:EXAMPLE}) a general form $W$, not necessarily
allowed. Then
$\nu_1=-13+21(i_1-1)+14(i_2-1)+18(i_1'-1)+12(i_2'-1)$, which is
congruent with $3i_1+2i_2$ modulo $N_{v_1}=6$. We wish to get, for
example, $\nu_1\equiv 3 \ (\!\!\mod 6)$, hence we might take $i_1=1$
and $i_2=6$. Assume also that $i_1'=i_2'=1$. Then $\nu_1=57$,
$\nu_0=8$ and $\nu_1'=47$. Then computing the zeta function we
realize that $s_0=-57/6$ is a pole, but $\exp(2\pi i
s_0)=-1\not\in Eig$.

(b) It is not hard to find divisors $W$ which are not allowed, but
such that nevertheless their poles provide only eigenvalues. Take
for example the trivial form $W=0$.  Then by (\ref{ex:EXAMPLE4})
it is not allowed, but using the expression for $Z(s)$ from
(\ref{ex:EXAMPLE}) we can conclude that the poles provide
eigenvalues.

(c)  In order to emphasize the subtlety of the statement of
Theorem~\ref{thm:monconj} (and of its proof) we `will try to find
a counterexample' of this fact. Namely, let us take $i_1=1$,
$i_2=3$ and write $I'$ for the expression  $3i_1'+2i_2'$. Note
that $W$ is allowed exactly when $I'=7$, cf. (\ref{ex:EXAMPLE4}).

By a computation $\nu_1=6I'-15$, $\nu_0=I'-3$ and $\nu'_1=7I'-24$.
Moreover, the zeta function $Z(s)$ is
$$\frac{2}{6I'-15+6s}+\frac{-1+I'/(i_1'i_2')}{7I'-24+6s}
+\frac{1}{I'-3+s}\big(-1+\frac{1}{s+1}+\frac{1}{6I'-15+6s}+\frac{1}{7I'-24+6s}\Big).$$
 In particular, $s_0=(15-6I')/6$ is a candidate pole of $Z(s)$, such that $\exp(2\pi is_0)=-1\not\in Eig$.

The point is that, for any $I'$, the residue of this candidate pole is zero,
that is, this is a fake candidate pole, not a pole. In particular,
for $I'=7$ we get no contradiction, and for some other special
choices of $I'$ we get plenty of non--allowed forms for which all
the poles provide eigenvalues. E.g.,
 for $I'\equiv 0\ (\!\!\mod 6)$ all the poles are integers.

(d) Let us consider the allowed form as in (c) given by $i_1=1$,
$i_2=3$ and $I'=7$. Then  $Z$ has a pole $s_0$ with $\exp(2\pi
is_0)=\exp(-2\pi i/6)$.

The other  root $\exp(2\pi i/6)$ of $t^2-t+1$ can be realized by
the following allowed form. Consider $i_1=1$, $i_2=1$ and $I'=7$.
Then $\nu_1=-1$, $\nu_0=0$, $\nu_1'=1$ and
$$Z(s)=\frac{4}{6s-1}-\frac{1}{s+1}+\frac{1}{6s+1}
\Big(-1+\frac{7}{i_1'i_2'}\Big)+\frac{12}{(6s-1)(6s+1)}.$$
 One immediately verifies that all candidate poles are indeed
poles. Hence, in fact, the poles of this unique zeta function hit
all eigenvalues of $Eig$.

\end{example}

\section{Plane curves}\labelpar{s:5}

In this section we treat the case when $(X,0)$ is smooth, that is, $f$ is a plane curve
singularity. We will fix some local coordinates $(x,y)$ of $(X,0)$.

Although the results of the section \ref{s:7} generalize some of
the statements of the present section, we prefer to provide some
details in this particular case too, since some of the (much
shorter) arguments might be of interest for specialists of plane
curve germs. Moreover, we also show how the classical situation
($W'=0$) is included in our general treatment.

There is another reason to separate the plane curve case. The
proof of the abundance of the allowed forms for general $(X,0)$
(which allows to realize {\it all} the monodromy eigenvalues) will
be proved under a technical  assumption regarding $\Gamma(F)$
(namely, the semigroup condition). Although this condition is
satisfied by splice diagrams of plane curve singularities, cf.
Remark \ref{re:1234}(1), it is natural to see how the case of
plane curves runs independently of this condition, just using
their standard properties.

\subsection{} In this subsection we verify that the standard form is allowed,
and hence we reprove the `classical monodromy conjecture' for
plane curve singularities.

\begin{proposition} \labelpar{prop:standard}
The standard differential form $dx\wedge dy$ (corresponding to the divisor
$W'=0$ on $X$) is allowed for any (plane) curve singularity  $f$ on $(X,0)=(\C^2,0)$.
\end{proposition}

\begin{proof}  We use the minimal embedded resolution $\pi:\tilde{X}\to X$ of
$f$ and show that the diagram
$\Gamma:=%\Gamma(F=\div(\pi^*f),W=K_\pi) =
\Gamma_\pi(X,F=\div(f),W=0)$ is allowed. Note that thus the
decorations $i_a-1=0$ for all $a\in \calA_W$.

Recall that on any star-shaped subdiagram of $\Gamma$ without
boundary vertex the allowedness condition is trivially satisfied.
If a star-shaped subdiagram contains exactly one boundary vertex,
which by assumption is a boundary vertex of $\Gamma$ too (before
the splice operation), then the corresponding leg decoration
(being $>1$) does not divide the associated $i_a (=1)$, hence the
allowedness condition is satisfied again.

It is well known that other star-shaped subdiagrams can arise from
at most one connected part of $\Gamma$, that has the following
form.

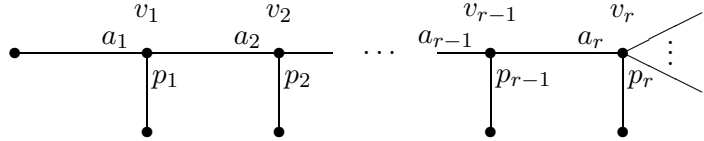
\begin{picture}(400,65)(0,0)
%\put(20,40){\circle*{4}}
\put(70,40){\circle*{4}}
\put(120,40){\circle*{4}} \put(170,40){\circle*{4}}
\put(250,40){\circle*{4}} \put(300,40){\circle*{4}}
%\put(70,10){\circle*{4}}
 \put(120,10){\circle*{4}}
\put(170,10){\circle*{4}} \put(250,10){\circle*{4}}
\put(300,10){\circle*{4}}

\put(70,40){\line(1,0){120}}\put(230,40){\line(1,0){70}}
\put(300,40){\line(2,1){30}}\put(300,40){\line(2,-1){30}}
%\put(70,40){\line(0,-1){30}}
\put(170,40){\line(0,-1){30}}
\put(120,40){\line(0,-1){30}}
\put(250,40){\line(0,-1){30}}
\put(300,40){\line(0,-1){30}}

%\put(58,45){\makebox(0,0){$a_1$}}
\put(108,45){\makebox(0,0){$a_1$}}
\put(158,45){\makebox(0,0){$a_2$}}
\put(233,45){\makebox(0,0){$a_{r-1}$}}
\put(288,45){\makebox(0,0){$a_r$}}

%\put(72,30){\makebox(0,0)[l]{$p_1$}}
\put(122,30){\makebox(0,0)[l]{$p_1$}}
\put(172,30){\makebox(0,0)[l]{$p_2$}}
\put(252,30){\makebox(0,0)[l]{$p_{r-1}$}}
\put(302,30){\makebox(0,0)[l]{$p_r$}}

\put(208,40){\makebox(0,0){$\ldots$}}
\put(318,43){\makebox(0,0){$\vdots$}}
%\put(70,55){\makebox(0,0){$v_1$}}
\put(120,55){\makebox(0,0){$v_1$}}
\put(170,55){\makebox(0,0){$v_2$}}
\put(250,55){\makebox(0,0){$v_{r-1}$}}
\put(300,55){\makebox(0,0){$v_r$}}

\end{picture}

This has the following type of splice sub--diagrams (where $2\leq k\leq r-1$):

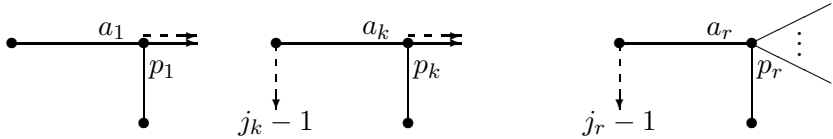
\begin{picture}(400,65)(0,0)
\put(20,40){\circle*{4}}
\put(70,40){\circle*{4}}
\put(120,40){\circle*{4}}
\put(170,40){\circle*{4}}
\put(250,40){\circle*{4}} \put(300,40){\circle*{4}}
\put(70,10){\circle*{4}}
% \put(120,10){\circle*{4}}
\put(170,10){\circle*{4}} %\put(250,10){\circle*{4}}
\put(300,10){\circle*{4}}

\put(20,40){\vector(1,0){70}}\put(250,40){\line(1,0){50}}
\put(300,40){\line(2,1){30}}\put(300,40){\line(2,-1){30}}
\put(70,40){\line(0,-1){30}}
\put(170,40){\line(0,-1){30}}
%\put(120,40){\line(0,-1){30}}
%\put(250,40){\line(0,-1){30}}
\put(300,40){\line(0,-1){30}}
\put(120,40){\vector(1,0){70}}

\put(58,45){\makebox(0,0){$a_1$}}
%\put(108,45){\makebox(0,0){$a_1$}}
\put(158,45){\makebox(0,0){$a_k$}}
%\put(233,45){\makebox(0,0){$a_{r-1}$}}
\put(288,45){\makebox(0,0){$a_r$}}

\put(72,30){\makebox(0,0)[l]{$p_1$}}

\put(172,30){\makebox(0,0)[l]{$p_k$}}
%\put(252,30){\makebox(0,0)[l]{$p_{r-1}$}}
\put(302,30){\makebox(0,0)[l]{$p_r$}}

%\put(208,40){\makebox(0,0){$\ldots$}}
\put(318,43){\makebox(0,0){$\vdots$}}

\put(120,10){\makebox(0,0){$j_k-1$}}
\dashline[3]{3}(120,40)(120,20)\put(120,20){\vector(0,-1){5}}
\put(250,10){\makebox(0,0){$j_r-1$}}
\dashline[3]{3}(250,40)(250,20)\put(250,20){\vector(0,-1){5}}

\dashline[3]{3}(70,43)(85,43)\put(85,43){\vector(1,0){5}}
\dashline[3]{3}(170,43)(185,43)\put(185,43){\vector(1,0){5}}
\end{picture}

Here the leg with decoration $p_r$ is optional. If it does not
occur, we put formally $p_r=1$. When $r=1$ it must occur,
otherwise we are in the situation discussed just before.  We will
use the positivity of the edge determinants, saying in this case
that $a_k> a_{k-1}p_{k-1}p_k$ for $k=2,\dots,r$.

For the first diagram the allowedness is automatically satisfied.

In the spliced star-shaped subdiagrams around the vertices $v_k,
2\leq k \leq r,$ it is a priori possible that $a_k | j_k$ for some
$k$. If furthermore $j_k \neq a_k$, the allowedness condition
would be violated. We will show however that $-a_k<j_k<0$ for
$k=2,\dots,r$, and this will finish the proof.

\smallskip
 More precisely, we verify by induction that $j_k<0$ and
$|j_k| < a_k/p_k$ for $k=2,\dots,r$.  First by (\ref{eq:i}) we
have that $j_2= p_1+a_1-a_1p_1$. Hence $j_2<0$ and $|j_2|< a_1p_1
< a_2/p_2$.

Take now $k\in\{2,\dots,r-1\}$. In this case (\ref{eq:totali})
yields
$$
j_{k+1} = p_kj_k + a_k - a_kp_k.
$$
The induction hypothesis says that $j_k<0$ and $|p_kj_k|<a_k$.
Consequently also $j_{k+1}<0$ and $|j_{k+1}|<a_kp_k <
a_{k+1}/p_{k+1}$.
\end{proof}

This together with Theorem~\ref{thm:monconj} give an alternative
proof of the classical monodromy conjecture  for curves (see
\cite{Lo2} for the original proof of a stronger result in the
context of $p$-adic zeta functions, and \cite{Ro} for a direct
proof).

\begin{corollary} \labelpar{cor:monconj} For any plane curve singularity $f$,
 if $s_0$ is a pole of the topological
zeta function $Z(f; s)$, then $\exp(2\pi is_0)$ is a monodromy
eigenvalue of $f$ at some point of $\{f=0\}$.
\end{corollary}

\subsection{A technical lemma.}\label{ss:??} In the remaining part of this section we show that
any monodromy eigenvalue of a given $f$ can be generated by poles
of different $Z(f,W)$ with $W$ allowed. The proof is given in
several steps. This subsection contains two technical partial
steps, Lemma \ref{le:pole} and Proposition \ref{prop:extend2},
targeting those subdiagrams from where the extension of the
allowed forms is harder (compare with (\ref{re:extend})(b)). The
main result is Theorem (\ref{th:ext}).

We formulate and prove the following lemma in the context of an
arbitrary diagram $\Gamma(F)$ (corresponding to an effective
divisor on an IHS germ $(X,0)$). The proof in the plane curve case
is not easier, and we will need the general statement  in the next section as well.

\begin{lemma} \labelpar{le:pole} Suppose that $\Gamma=\Gamma(F)$
is star-shaped. Let the central node be connected to $n$ boundary
vertices and $r$ arrowheads whose supporting edges have
decorations $\{d_\ell\}_{\ell=1}^n$ and $\{p_\ell\}_{\ell=1}^r$,
respectively.
 Here $r\geq 1$, $n\geq 0$, and $r+n\geq 3$.

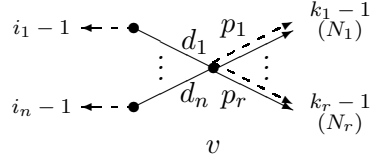
\begin{picture}(200,80)(265,-20)
\put(440,20){\circle*{4}}\put(410,35){\circle*{4}}\put(410,5){\circle*{4}}
\put(440,19){\vector(2,1){30}}\put(440,19){\vector(2,-1){30}}
\put(440,20){\line(-2,-1){30}} \put(440,20){\line(-2, 1){30}}
\put(460,23){\makebox(0,0){$\vdots$}}\put(420,23){\makebox(0,0){$\vdots$}}
\put(440,-10){\makebox(0,0){$v$}}
\put(448,35){\makebox(0,0){$p_1$}}\put(448,8){\makebox(0,0){$p_r$}}
 \put(433,30){\makebox(0,0){$d_1$}} \put(433,10){\makebox(0,0){$d_n$}}
\put(488,34){\makebox(0,0){\tiny{$(N_1)$}}}
\put(488,42){\makebox(0,0){\tiny{$k_1-1$}}}
\put(488,-1){\makebox(0,0){\tiny{$(N_r)$}}}
\put(488,7){\makebox(0,0){\tiny{$k_r-1$}}}
\put(375,5){\makebox(0,0){\tiny{$i_n-1$}}}
\put(375,35){\makebox(0,0){\tiny{$i_1-1$}}}
%\put(560,20){\makebox(0,0){($r$ arrowheads)}}
\dashline[3]{3}(440,22)(465,35)\put(465,35){\vector(2,1){5}}
\dashline[3]{3}(440,22)(465,9)\put(465,9){\vector(2,-1){5}}
\dashline[3]{3}(410,35)(395,35)\put(395,35){\vector(-1,0){5}}
\dashline[3]{3}(410,5)(395,5)\put(395,5){\vector(-1,0){5}}
\end{picture}

Let $\lambda$ be a root of the Alexander polynomial
$\Lambda_\Gamma(t)$ of $\Gamma$. Then there exist infinitely many (even infinitely many {\rm effective})
allowed $P$-divisors $W$ for $\Gamma$ (corresponding to the
decorated dashed arrows on the diagram) admitting a pole $s_0$ of the
topological zeta function $Z(\Gamma(F,W);s)$, such that $\exp(2\pi
is_0)=\lambda$.

\end{lemma}

\begin{proof} %We might use the diagram from  part (2) of the proof of (\ref{thm:monconj}).
Denote $D:=\prod_{\ell=1}^n d_\ell$ and
$P:=\prod_{\ell=1}^r p_\ell$. For any $P$-divisor $W$, that is,
for any set of decorations $k_1,\dots,k_r,i_1,\dots,i_n$ (with the
$i_\ell\neq 0$), we have that $Z(s):=Z(\Gamma(F,W);s)$ has the
form
$$
\frac 1{\nu + sN} \left( 2-r-n+ \sum_{\ell=1}^n \frac
{d_\ell}{i_\ell} + \sum_{\ell=1}^r \frac {p_\ell} {k_\ell+sN_\ell}
\right),
$$
where
\begin{equation}\label{eq:nu2z}
N= D\sum_{\ell=1}^r \frac P{p_\ell} N_\ell
\ \ \ \mbox{and} \ \ \
\nu = D \sum_{\ell=1}^r  \frac P{p_\ell} k_\ell + P\sum_{\ell=1}^n
\frac D{d_\ell} i_\ell - (r+n-2)DP.
\end{equation}
If the candidate pole $-\nu/N$ is equal to some $k_\ell/N_\ell$,
then it is a pole of order two. Otherwise, we consider its
residue, which is (up to a factor $N$) equal to
$$
\calR:= 2-r-n+ \sum_{\ell=1}^n \frac {d_\ell}{i_\ell} +
\sum_{\ell=1}^r \frac {p_\ell}{k_\ell -(\nu/N)N_\ell}.
$$
One easily verifies that this expression is not identically zero
as function in the $r+n (\geq 3)$ variables $i_\ell$ and $k_\ell$.
Hence $-\nu/N$ is a pole of $Z(s)$ as soon as the algebraic
equation $\calR=0$ is not satisfied.

We consider three cases (depending on the value of $r$) for the
roots of $\Lambda_\Gamma(t)$, which is given by
$$
\frac {(t^N-1)^{r+n-2}}{\prod_{\ell=1}^n (t^{N/{d_\ell}}-1)},
$$
except when $r=1$, where we must multiply this expression by
$t^{N_1}-1$.

\smallskip
$\bullet$  $\mathbf{(r\geq 3)}$ \ Then its roots are all the
$N$-th roots of unity.

$\bullet$  $\mathbf{(r=2)}$ \ Then its roots are all $N$-th roots of unity that are {\it
not} $(N/d_\ell)$-th roots of unity simultaneously for all
$\ell=1,\dots,n$. In other words all $\exp(2\pi i \frac uN)$ for
which $u \not\equiv 0 \mod d_\ell$ for at least one $d_\ell$.

$\bullet$  $\mathbf{(r=1)}$ \ Then its roots are all $N$-th roots of unity that are {\it
not} $(N/d_\ell)$-th roots of unity simultaneously for (at least)
$n-1$ indexes $\ell=1,\dots,n$. In other words all $\exp(2\pi i
\frac uN)$ for which $u \not\equiv 0 \mod d_\ell$ for at least two
different $d_\ell$.

\smallskip
Fix a root $\lambda=\exp(2\pi i \frac uN)$ of $\Lambda_\Gamma(t)$.
Since the numbers $p_1,\dots,p_r,d_1,\dots,d_n$ are pairwise
coprime, there exist integers $k_1,\dots,k_r,i_1,\dots,i_n$ (all
positive if we desire so) such that $\nu$ in (\ref{eq:nu2z})
satisfies $\nu \equiv u \mod N$. When $r=2$ or $r=1$, the
restrictions on the given $u$ imply that $d_\ell \nmid i_\ell$ for
at least one or at least two indexes $\ell$, respectively. Hence
in each case the constructed $W$ is allowed for $\Gamma$.

Since we can choose the numbers $k_1,\dots,k_r,i_1,\dots,i_n$
freely$\mod N$, it is clear that we can find infinitely many such
sets (in $\Z$ or in $\Z_{>0}$) that satisfy $\calR \neq 0$, and hence $s_0=-\nu/N$ is then a
pole of $Z(s)$ satisfying $\exp(2\pi is_0)=\lambda$.
\end{proof}

\begin{remark} \labelpar{re:pole}  Assume that above  $r=1$ with $p_1=1$ and
$n=2$. Set $s_0=-\nu/N$. The fact that $\lambda=\exp(2\pi is_0)$ is a root
 of $\Lambda_\Gamma(t)$ is equivalent to $d_1 \nmid i_1$ and $d_2 \nmid i_2$.
 On the other hand, $\calR=0$ if and only if $(d_1-i_1)(d_2 -i_2)=0$.
 Hence, if $\lambda$ is a root, then $s_0$ is a pole  of $Z(s)$ (for any
allowed $W$); see also (3.4) in \cite{NV}.
\end{remark}

\bekezdes Now we return to plane curve singularities and we target
that subdiagram of the minimal splice diagram whose star--shaped
components after splicing have only one ordinary arrowhead; see
also Remark \ref{re:extend}(b).

\begin{proposition} \labelpar{prop:extend2} Consider the subdiagram (given below)
of the minimal embedded resolution diagram $\Gamma=\GaxF$ of a
plane curve germ, determining  a $P$--divisor $F$. Here $r\geq 2$
and the leg with decoration $p_r$ is optional.

Fix $k\in \{1,\dots,r-1\}$, consider the star-shaped subdiagram
$S_k=\Gamma(F_k)$ around $v_k$ and fix a root $\lambda$ of the
Alexander polynomial $\Lambda_{S_k}(t)$.  Then there exist
infinitely many allowed divisors $W_k$ for $S_k$, such that if
$(N_k,\nu_k-1)$ denote the decorations of $v_k$ as above
associated with $S_k$ and $W_k$, then

(0) \ $s_0=-\nu_k/N_k$ is a pole of $Z(F_k,W_k;s)$,

(1) \ $\exp(2\pi i s_0)=\lambda$, and

(2) $W_k$ can be extended to an allowed divisor on the whole
diagram $\Gamma$.
\end{proposition}

\noindent The subdiagram of  $\Gamma$  is

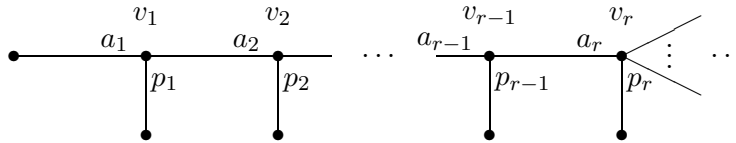
\begin{picture}(400,65)(0,0)
%\put(20,40){\circle*{4}}
\put(70,40){\circle*{4}}
\put(120,40){\circle*{4}} \put(170,40){\circle*{4}}
\put(250,40){\circle*{4}} \put(300,40){\circle*{4}}
%\put(70,10){\circle*{4}}
 \put(120,10){\circle*{4}}
\put(170,10){\circle*{4}} \put(250,10){\circle*{4}}
\put(300,10){\circle*{4}}

\put(70,40){\line(1,0){120}}\put(230,40){\line(1,0){70}}
\put(300,40){\line(2,1){30}}\put(300,40){\line(2,-1){30}}
%\put(70,40){\line(0,-1){30}}
\put(170,40){\line(0,-1){30}}
\put(120,40){\line(0,-1){30}}
\put(250,40){\line(0,-1){30}}
\put(300,40){\line(0,-1){30}}

%\put(58,45){\makebox(0,0){$a_1$}}
\put(108,45){\makebox(0,0){$a_1$}}
\put(158,45){\makebox(0,0){$a_2$}}
\put(233,45){\makebox(0,0){$a_{r-1}$}}
\put(288,45){\makebox(0,0){$a_r$}}

%\put(72,30){\makebox(0,0)[l]{$p_1$}}
\put(122,30){\makebox(0,0)[l]{$p_1$}}
\put(172,30){\makebox(0,0)[l]{$p_2$}}
\put(252,30){\makebox(0,0)[l]{$p_{r-1}$}}
\put(302,30){\makebox(0,0)[l]{$p_r$}}

\put(208,40){\makebox(0,0){$\ldots$}}
\put(318,43){\makebox(0,0){$\vdots$}}
%\put(70,55){\makebox(0,0){$v_1$}}
\put(120,55){\makebox(0,0){$v_1$}}
\put(170,55){\makebox(0,0){$v_2$}}
\put(250,55){\makebox(0,0){$v_{r-1}$}}
\put(300,55){\makebox(0,0){$v_r$}}

\put(340,40){\makebox(0,0){$\ldots$}}

%\dashline[3]{3}(170,10)(185,10)\put(185,10){\vector(1,0){5}}
%\dashline[3]{3}(120,10)(135,10)\put(135,10){\vector(1,0){5}}
%\dashline[3]{3}(250,10)(265,10)\put(265,10){\vector(1,0){5}}
%\dashline[3]{3}(300,10)(315,10)\put(315,10){\vector(1,0){5}}
\end{picture}

\noindent with spliced star--shaped subdiagrams (where we insert the wanted $W_k$'s too)

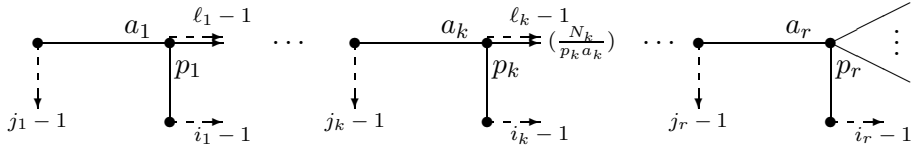
\begin{picture}(400,65)(-20,0)
\put(0,40){\circle*{4}}
\put(50,40){\circle*{4}}
\put(120,40){\circle*{4}}
\put(170,40){\circle*{4}}
\put(250,40){\circle*{4}} \put(300,40){\circle*{4}}
\put(50,10){\circle*{4}}
% \put(120,10){\circle*{4}}
\put(170,10){\circle*{4}} %\put(250,10){\circle*{4}}
\put(300,10){\circle*{4}}

\put(0,40){\vector(1,0){70}}\put(250,40){\line(1,0){50}}
\put(300,40){\line(2,1){30}}\put(300,40){\line(2,-1){30}}
\put(50,40){\line(0,-1){30}}
\put(170,40){\line(0,-1){30}}
%\put(120,40){\line(0,-1){30}}
%\put(250,40){\line(0,-1){30}}
\put(300,40){\line(0,-1){30}}
\put(120,40){\vector(1,0){70}}

\put(38,45){\makebox(0,0){$a_1$}}
%\put(108,45){\makebox(0,0){$a_1$}}
\put(158,45){\makebox(0,0){$a_k$}}
%\put(233,45){\makebox(0,0){$a_{r-1}$}}
\put(288,45){\makebox(0,0){$a_r$}}

\put(52,30){\makebox(0,0)[l]{$p_1$}}

\put(172,30){\makebox(0,0)[l]{$p_k$}}

\put(302,30){\makebox(0,0)[l]{$p_r$}}

\put(325,43){\makebox(0,0){$\vdots$}}

\put(120,10){\makebox(0,0){\tiny{$j_k-1$}}}
\dashline[3]{3}(120,40)(120,20)\put(120,20){\vector(0,-1){5}}
\put(250,10){\makebox(0,0){\tiny{$j_r-1$}}}
\dashline[3]{3}(250,40)(250,20)\put(250,20){\vector(0,-1){5}}

\dashline[3]{3}(50,42)(65,42)\put(65,42){\vector(1,0){5}}
\dashline[3]{3}(170,42)(185,42)\put(185,42){\vector(1,0){5}}

\dashline[3]{3}(50,10)(65,10)\put(65,10){\vector(1,0){5}}
\dashline[3]{3}(170,10)(185,10)\put(185,10){\vector(1,0){5}}
%\dashline[3]{3}(250,10)(265,10)\put(265,10){\vector(1,0){5}}
\dashline[3]{3}(300,10)(315,10)\put(315,10){\vector(1,0){5}}

\put(70,50){\makebox(0,0){\tiny{$\ell_1-1$}}}\put(70,5){\makebox(0,0){\tiny{$i_1-1$}}}
\put(190,50){\makebox(0,0){\tiny{$\ell_k-1$}}}
\put(193,40){\makebox(0,0)[l]{\tiny{$(\frac{N_k}{p_ka_k})$}}}
\put(190,5){\makebox(0,0){\tiny{$i_k-1$}}}
\put(320,5){\makebox(0,0){\tiny{$i_r-1$}}}

%\put(300,40){\vector(2,1){30}}\put(300,40){\vector(2,-1){30}}
%\dashline[3]{3}(300,43)(330,58)\put(325,55){\vector(2,1){5}}
%\dashline[3]{3}(300,43)(330,28)\put(325,31){\vector(2,-1){5}}

\dashline[3]{3}(0,40)(0,20)\put(0,20){\vector(0,-1){5}}
\put(0,10){\makebox(0,0){\tiny{$j_1-1$}}}
\put(95,40){\makebox(0,0){$\ldots$}}\put(235,40){\makebox(0,0){$\ldots$}}
\end{picture}

\begin{proof} From Lemma \ref{le:pole} we know that infinitely many
$W_k$ satisfying (0) and (1) exist, but in order to satisfy also
(2), we will specify choices. Write $\lambda$ as
$\lambda=\exp(2\pi i (u/N_k))$, where (since it is a root of
$\Lambda_{S_k}(t)$)  $a_k \nmid u$ and $p_k \nmid u$. We choose
the decorations $i_k$, $j_k$ and $\ell_k$ such that
$$
\nu_k=i_ka_k + j_kp_k +(\ell_k-1) a_kp_k \equiv u  \hspace{6mm} (\mmod \ N_k),
$$
and, moreover
\begin{equation}\label{eq:625}
0< i_ka_k + j_kp_k - a_kp_k < a_kp_k.
\end{equation}
Note that this is possible since $a_kp_k | N_k$, and that we can
choose $\ell_k$ freely$\mod N_k/(a_kp_k)$. By Remark \ref{re:pole}
we know that $-\nu_k/N_k$ is a pole of $Z(F_k,W_k;s)$.

We claim that we can choose inductively $\{i_m,\ell_m\}$ for
$m=k+1,\dots,r-1$ such that $|j_{m+1}|<a_mp_m$ for
$m=k,\dots,r-1$. Then for these $m$ this yields $|j_{m+1}|<a_m
p_m<a_{m+1}$ and thus $a_{m+1}\nmid j_{m+1}$. By Remark
\ref{re:extend} this ensures that we can extend $W_k$ further to
obtain an allowed divisor $W$ on the whole diagram $\Gamma$.

We now prove the claim. By (\ref{eq:totali}) we have
\begin{equation}\label{eq:ell}
\ell_m=i_{m+1} + p_{m+1}(\ell_{m+1}-1)
\end{equation}
for $m=k,\dots,r-2$, and
\begin{equation}\label{eq:j}
j_{m+1}=i_ma_m + j_mp_m -a_mp_m
\end{equation}
for $m=k,\dots,r-1$. In particular, we know already from
(\ref{eq:625}) that $|j_{k+1}|<a_kp_k$. When some $\{i_m,\ell_m\}$
is constructed we take each time $i_{m+1}$ and $\ell_{m+1}$ in
(\ref{eq:ell}) such that $1\leq i_{m+1} \leq p_{m+1}$. Then it
follows  from (\ref{eq:j}) and the inductive argument (and the
positivity of the edge determinant) that indeed $|j_{m+1}|<a_mp_m$
for all $m=k+1,\dots,r-1$.
\end{proof}

\subsection{}\label{ss:maincurve} Now we are ready to prove the theorem regarding the
abundance of the allowed forms.

\begin{theorem}\labelpar{th:ext}  Let $(X,0)$ be a smooth surface germ and $f$ an analytic function on $X$,
determining a (plane) curve singularity. Let $\lambda$ be a
monodromy eigenvalue of $f$ at a point of $\{f=0\}$. Then there
exist infinitely many {\it allowed} $P$-divisors $W$ for
$(X,\div(f))$, and for each of them a pole $s_0$ of the
topological zeta function $Z(f,W;s)$ such that $\exp(2\pi
is_0)=\lambda$.
\end{theorem}

\begin{proof} Let $\Gamma(F)$ be the diagram of the minimal
embedded resolution of $f$.

(1) Suppose first that $\lambda$ is a monodromy eigenvalue at a
point $b\in \{f=0\}$, with $b\neq 0$. Writing
$\div(f)=\sum_{a\in\calA_F} N_a F_a$, this means that $\lambda$ is
a $N_a$-th root of unity for some $N_a$. Fix such an
$a\in\calA_F$; so $\lambda=\exp(2\pi i (-u/N_a))$ for some (fixed)
$u\in \{1,\dots,N_a\}$.

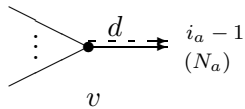
\begin{picture}(100,50)(85,0)
\put(240,20){\circle*{4}}
\put(240,20){\vector(1,0){30}}
\put(240,20){\line(-2,-1){30}}
\put(240,20){\line(-2, 1){30}}
\put(220,23){\makebox(0,0){$\vdots$}}
\put(250,28){\makebox(0,0){$d$}}
\put(242,0){\makebox(0,0){$v$}}
\put(288,25){\makebox(0,0){\tiny{$i_a-1$}}}
\put(285,14){\makebox(0,0){\tiny{$(N_a)$}}}
\dashline[3]{3}(240,22)(265,22)\put(265,22){\vector(1,0){5}}
\end{picture}

\medskip
Consider the star-shaped subdiagram $S_v=\Gamma(F_v)$ around the
vertex $v$, to which the arrowhead $a$ is attached. Choose an
allowed $P$-divisor $W_v$ for $S_v$ with decoration $i_a \equiv u
\mod N_a$, such that moreover $i_a/N_a \neq i_{a'}/N_{a'}$ for all
(eventual) other arrowheads $a'$ on $S_v$. (Here we have
infinitely many such choices.) Then either $-i_a/N_a$ is a pole of
order two of $Z(F_v,W_v;s)$, or the contribution of $a$ to this
zeta function is
$$
\frac d{(\nu_v+sN_v)(i_a+sN_a)}
$$
and hence the residue of $-i_a/N_a$ is nonzero.

We surely can extend $W_v$ to an allowed divisor $W$ on the whole
of $\Gamma(F)$ (see Remark \ref{re:extend}), and doing so we do
not use the value $i_a$. (For this see the proof of Proposition
\ref{prop:extend}; with the notation of that proof, from
$\Gamma_R$ only $i'$ was used.)  If $-i_a/N_a$ is not a pole of
order two and $i_a/N_a$ would be equal to some value $\nu_w/N_w,
w\in (\calN \cup \calA_F) \setminus \{a\},$ then we can add to
$i_a$ some multiple of $N_a$ in order to avoid this. This way we
are sure that $-i_a/N_a$ is a pole of $Z(f,W;s)$.

\smallskip
(2) Suppose now that $\lambda$ is not as in (1); hence it is a
root of the Alexander polynomial $\Lambda_{\Gamma(F)}(t)$. By
Proposition \ref{prop:alex} there is at least one node $w\in\calN$
such that $\lambda$ is a root of $\Lambda_{S_w}(t)$, where
$S_w=\Gamma(F_w)$ is the star-shaped subdiagram around $w$. By
Lemma \ref{le:pole} and Proposition \ref{prop:extend2} there exist
infinitely many allowed $P$-divisors $W_w$ for $S_w$, and for each
of them a pole $s_0$ of $Z(F_w,W_w;s)$ such that $\exp(2\pi
is_0)=\lambda$ and $W_w$ can be extended to an allowed divisor $W$
on the whole diagram $\Gamma(F)$. Indeed, the possible obstruction
to extend $W_w$, as described in Remark \ref{re:extend}, is
removed in Proposition \ref{prop:extend2}.

If $s_0$ is a pole of order two of $Z(F_w,W_w;s)$, then it is a
pole of order two of $Z(s)=Z(f,W;s)$ too, hence we are done.
Otherwise, there is a potential problem when the following
situation occurs: $s_0$ is not a pole of order two,
$s_0=-\nu_w/N_w$ for a subset  $\calN' \subset \calN$ containing
at least two nodes $w$ of $\Gamma(F)$, such that for each
$w\in\calN'$ the local residue-contribution (to the total residue
of $s_0$) $\calR_w \neq 0$,  and $\sum_{w\in\calN'}\calR_w =0$. In
this case $s_0$ is not a pole of $Z(s)$, although it is a pole of
several $Z(F_w,W_w;s)$.

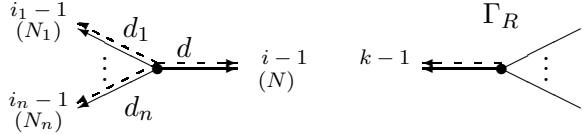
\begin{picture}(100,65)(105,-10)

\put(240,20){\circle*{4}} %\put(270,20){\circle*{4}} \put(340,20){\circle*{4}}
\put(370,20){\circle*{4}}
\put(240,20){\vector(1,0){30}}\put(370,20){\vector(-1,0){30}}
\put(240,20){\vector(-2,-1){30}}
\put(240,20){\vector(-2, 1){30}}
\put(220,23){\makebox(0,0){$\vdots$}}
\put(370,20){\line(2,-1){30}}
\put(370,20){\line(2, 1){30}}
\put(387,23){\makebox(0,0){$\vdots$}}
%\put(242,-10){\makebox(0,0){$v_L$}}
%\put(368,-10){\makebox(0,0){$v_R$}}
\put(250,28){\makebox(0,0){$d$}} %\put(363,28){\makebox(0,0){$d'$}}
%\put(388,35) {\makebox(0,0){$d'_1$}}

\put(370,40){\makebox(0,0){$\Gamma_R$}}

 \put(232,35){\makebox(0,0){$d_1$}} \put(233,5){\makebox(0,0){$d_n$}}
\put(288,24){\makebox(0,0){\tiny{$i-1$}}}
%\put(324,23){\makebox(0,0){\tiny{$i'-1$}}}
\put(285,15){\makebox(0,0){\tiny{$(N)$}}}

\dashline[3]{3}(240,22)(265,22)\put(265,22){\vector(1,0){5}}
\put(326,24){\makebox(0,0){\tiny{$k-1$}}}
\put(195,42){\makebox(0,0){\tiny{$i_1-1$}}}
\put(195,34){\makebox(0,0){\tiny{$(N_1)$}}}

\put(195,8){\makebox(0,0){\tiny{$i_n-1$}}}
\put(195,0){\makebox(0,0){\tiny{$(N_n)$}}}

\dashline[3]{3}(240,22)(210,37)\put(215,35){\vector(-2,1){5}}
\dashline[3]{3}(240,22)(210,7)\put(215,10){\vector(-2,-1){6}}
\dashline[3]{3}(370,22)(345,22)\put(345,22){\vector(-1,0){5}}

\end{picture}

Take an \lq extreme\rq\ node $v \in \calN'$ (the node of the left diagram above),
meaning that it is a boundary vertex of the full subdiagram of $\Gamma$ generated by  $\calN'$.
%(geodesically) connected to only one other node in $\calN'$.
Consider the star-shaped subdiagram $S_v$ of $\Gamma=\Gamma(F,W)$
around $v$, where the edge $e$ with decoration $d$ is in the
direction of the other nodes in $\calN'$, and the diagram
$\Gamma_R$, obtained after splicing $\Gamma$ along $e$. (It is not
necessary to have arrowheads at all the  legs of $S_v$, as it is
indicated in the above diagram; in those cases we put formally
$N_\ell=0$ or $N=0$.)

Denote $D:=\prod_{j=1}^n d_j$. By (\ref{eq:nu}) and
(\ref{eq:totali}) (or from  (\ref{eq:**}) below)
%the proof of Lemma\ref{lemma:edge})
we know that $\nu_v$ depends only on $i$ and
$k=\sum_{j=1}^n (D/d_j)i_j - (n-1)D$, and not on the actual
(separate) values of $i_1,\dots,i_n$. The residue-contribution
$\calR_v$ though depends on these values.

\smallskip
{\it We claim that we can modify $i_1,\dots,i_n$ keeping $k$ (and
hence $\nu_v$) fixed, but changing $\calR_v$, such that the newly
created divisor (determined by these new $i_1,\ldots, i_n$) has
the following properties: it agrees with the old $P$--divisor on
$\Gamma_R$, and can be extended from  $S_v\cup \Gamma_R$ further
`to the left' to a new allowed $P$-divisor $\tilde W$ on the whole
diagram.}

\smallskip
In that way the new value becomes $\sum_{v\in\calN'}\calR_v \neq
0$. If \lq on the left\rq\ there are no nodes $v'$ with \lq new
value\rq\ $\nu_{v'}/N_{v'}=s_0$, we are done since then this sum
is the (total) residue of $s_0$ for $Z(s)$. We are still done if
the sum of $\sum_{v\in\calN'}\calR_v$ and all new
residue-contributions of these $v'$ is nonzero. Otherwise, we
repeat the argument, replacing $v$ in the claim by such a new
(extreme) node $v'$. This process must stop by finiteness of the
diagram.

\smallskip
(3)  We now prove the claim. We start with
two observations.

\smallskip
\noindent (i) Since there are at least two nodes in $\calN'$, we
can always assume that there is an ordinary arrowhead at the leg
of $S_v$ with decoration $d$, that is, that $N\neq 0$. Then, with
the terminology of Remark \ref{re:extend}(b), we can extend any
$P$-divisor on $\Gamma_R$ \lq unconditionally\rq\ to the left.
More precisely, the specific situation/problem of Proposition
\ref{prop:extend2} will never occur; we can always simply follow
the procedure in the proof of Proposition \ref{prop:extend}.

\smallskip
\noindent (ii) We must be sure that we {\it can} modify
$i_1,\dots,i_n$ while extending from $\Gamma_R$, considering the
extension procedure described in the proof of Proposition
\ref{prop:extend}. In this procedure there is no room to modify
$i_1,\dots,i_n$ only if (after renumbering) $i_\ell$ must be
chosen as $i_\ell=d_\ell$ for $\ell \leq n-1$. But in this case we
would have that $\calR_v=0$ (see Lemma \ref{le:nopole}),
contradicting our assumption.

\smallskip
 By (\ref{eq:nu}),
(\ref{eq:totali}) and (\ref{eq:mult}) we have
\begin{equation}\label{eq:**}
\nu_v=\sum_{j=1}^n d(D/d_j)i_j +Di - (n-1)dD = dk+Di
\end{equation}
and
$$
N_v=\sum_{j=1}^n d(D/d_j)N_j +DN.
$$
We have further that
$$
N_v \calR_v= 1-n+\sum_{j=1}^n \frac
{d_j}{i_j-(\nu_v/N_v)N_j}+\frac d{i-(\nu_v/N_v)N}.
$$
{\bf Case $\mathbf{n\geq 3}$.} We replace the triple $(i_1,i_2,i_3)$ by
$(i_1+xd_1,i_2+yd_2,i_3-(x+y)d_3)$ where $x,y\in\Z$. Then $\nu_v$
does not change, but the three corresponding terms in $N_v
\calR_v$ are replaced by
$$
\frac {d_1}{i_1+xd_1-(\nu_v/N_v)N_1}+\frac
{d_2}{i_2+yd_2-(\nu_v/N_v)N_2}+\frac
{d_3}{i_3-(x+y)d_3-(\nu_v/N_v)N_3}.
$$
It is easy to see that this expression is not constant as function
in $x$ and $y$; hence we can choose appropriate $x$ and $y$ in
$\Z$ such that the \lq new\rq\ $\calR_v$ is different from the
original one. (Note that divisibility of $i_\ell$ by $d_\ell$ does
not change, so we don't destroy allowedness.) We then extend this
new $P$-divisor from $S_v\cup \Gamma_R$ further to an allowed
divisor on the whole diagram.

\smallskip
\noindent {\bf  Case $\mathbf{n=2}$.} We replace the pair $(i_1,i_2)$ by
$(i_1+xd_1,i_2-xd_2)$ where $x\in\Z$. Again $\nu_v$ does not
change, and now the two corresponding terms in $N_v \calR_v$ are
replaced by
$$
\frac {d_1}{i_1+xd_1-(\nu_v/N_v)N_1}+\frac
{d_2}{i_2-xd_2-(\nu_v/N_v)N_2}.
$$
When this expression is not constant in $x$, we conclude as above.
It is constant in $x$ if and only if it is identically zero if and
only if
$$
d_1i_2+d_2i_1 - \frac{\nu_v}{N_v}(d_1N_2+d_2N_1) = 0.
$$
Suppose this identity holds. Then the formulas for $\nu_v$ and
$N_v$ above easily yield that $\nu_v/N_v= (i-d)/N$. But then the
(original) $N_v\calR_v$ would be equal to
$$
-1+\frac d{i-(\nu_v/N_v)N} = 0,
$$
contradicting the assumption.
\end{proof}

\begin{remark} \labelpar{re:effective} For plane curve singularity germs $f$, the associated allowed $W$ in Theorem \ref{th:ext} are always divisors of differential forms $\omega$. From the proof of the more general Theorem \ref{thm:allrealized}, we will see that there exist moreover infinitely many \emph{effective} allowed divisors $W$ doing the job in the theorem, corresponding here in the plane curve case to \emph{holomorphic} differential forms $\omega$.
\end{remark}

\section{Diagrams $\Gamma(F)$ with the semigroup condition}\labelpar{s:7}

\subsection{The semigroup condition}\labelpar{ss:SC}  Let us fix a diagram $\Gamma(F)$.
The reader is invited to recall the definition of the {\it
semigroup condition} associated with $\Gamma(F)$ from
(\ref{be:semigro}).

The semigroup condition of $\Gamma(F)$
is equivalent with the following property: for any edge $e$ (as in
the following diagram, see also (\ref{be:fo})), such that  ${\mathcal A}_{F,L}=\emptyset$,
 $d'$ is in the  semigroup ${\mathcal S}_e$
 generated by $\{l_{ew}\}_w$ where the index $w$ runs over
 all the boundary vertices of $\Gamma _L$.
 (For the definition of $l_{ew}$ see (\ref{be:f}).)

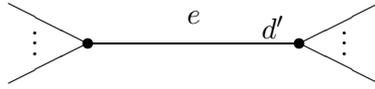
\begin{picture}(200,55)(-100,-5)
\put(40,20){\circle*{4}} %\put(80,20){\circle*{4}}
\put(120,20){\circle*{4}} \put(40,20){\line(1,0){80}}
\put(40,20){\line(-2,-1){30}} \put(40,20){\line(-2, 1){30}}
\put(20,23){\makebox(0,0){$\vdots$}} \put(120,20){\line(2,-1){30}}
\put(120,20){\line(2, 1){30}}
\put(137,23){\makebox(0,0){$\vdots$}}
%\put(120,0){\makebox(0,0){$v$}}
 \put(110,26){\makebox(0,0){$d'$}} %\put(55,26){\makebox(0,0){$d$}}
\put(80,30){\makebox(0,0){$e$}}
\end{picture}

This condition appears naturally in the context of {\it splice
quotient singularities}, introduced by Neumann and Wahl
\cite{NW1,NWuj}. For some special diagrams this condition is
automatically satisfied. For example, if $\Gamma$ represents a
rational germ (which in the context of
 IHS germs is equivalent with the fact that $\Gamma$ represents either the smooth
 or the $E_8$ germ), and the diagram is not necessarily minimal and $F$ is arbitrary, then
 $\Gamma(F)$ has the semigroup condition. Another case is when $\Gamma$ is minimal and it represents a
 minimally elliptic (automatically Gorenstein) singularity. These facts follow from the
 `End Curve Theorem' \cite{NWEC,Ok}.

It is convenient to denote
the {\it subsemigroup} of $\N$ generated by  $g_1,\ldots g_t$ by $\calS\langle g_1,\ldots, g_t\rangle$.

\subsection{Preliminary arithmetical properties.}

Here we gather some arithmetical properties which will be useful in
the proofs of the  main results  of this section (listed in the next subsection).

\begin{lemma} \labelpar{le:elem} Let $d_1,\dots,d_n$ be pairwise coprime positive
integers, and denote $D:=\prod_{j=1}^n d_j$. Then the following two facts hold.

(a)  There exist no  positive integers $m_j$ such  that
$$
\sum_{j=1}^n m_j\frac D{d_j} = (n-1)D.
$$
(b) If $d\in\calS\langle D/d_1,\dots,D/d_n\rangle$, $d>0$ and $d|D$,
then $(D/d_j)\mid d$ for some $j\in\{1,\dots,n\}$.
\end{lemma}

\begin{proof} (a) In such an equality we would have that $d_j|m_j$ for
all $j=1,\dots,n$. But then the left hand side would be at least
$nD$.
In (b), by assumption,  we can write $d$ in the form
\begin{equation}\label{eq:DIVDIV}
d=\sum_{j=1}^n m_j \frac D{d_j},\end{equation}
where all $m_j$ are nonnegative integers, and also
$d=\prod_{j=1}^n \bar{d}_j$
with $\bar{d}_j|d_j$ for all $j$. Since the $d_j$ are pairwise coprime,
(\ref{eq:DIVDIV}) shows that  $\bar{d}_j|m_j$ for all $j$. Writing
$d$ as $\bar{d}_j\prod_{\ell\neq j}\bar{d}_\ell$, we conclude that
$d$ divides $m_j D/d_j$
for all $j$. If at least two of the numbers $m_j$
would be nonzero, say $m_1\neq 0$ and $m_2\neq 0$, we obtain the
contradiction
$$
d\geq m_1 \frac D{d_1} + m_2 \frac D{d_2} \geq d+d.
$$
Hence exactly one $m_j$ is nonzero, implying then that
$(D/d_j)\mid d$.
\end{proof}

Recall that a diagram $\Gamma$ is called {\it minimal} if all the
decorations $d_{ve}$ are strictly greater than 1, provided that
$e$ connects the node $v$ with a boundary vertex.

\begin{proposition} \labelpar{prop:div} Let $\Gamma(F)$ be a splice diagram as
in (\ref{ss:2}) (hence with $W'=0$) and minimal in the above
sense. Let  $e$ be an edge connecting two nodes such that $\calA_{F,L}=\emptyset$, and set
$i'-1$ the multiplicity of the induced dashed arrowhead at $v_R$
after splicing $\Gamma$ along $e$ as in (\ref{be:fo}).

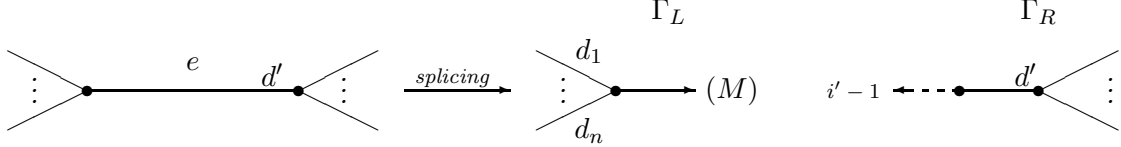
\begin{picture}(400,65)(10,-10)

\put(40,20){\circle*{4}} %\put(80,20){\circle*{4}}
\put(120,20){\circle*{4}} \put(40,20){\line(1,0){80}}
\put(40,20){\line(-2,-1){30}} \put(40,20){\line(-2, 1){30}}
\put(20,23){\makebox(0,0){$\vdots$}} \put(120,20){\line(2,-1){30}}
\put(120,20){\line(2, 1){30}}
\put(137,23){\makebox(0,0){$\vdots$}}
%\put(120,0){\makebox(0,0){$v$}}
 \put(110,26){\makebox(0,0){$d'$}} %\put(55,26){\makebox(0,0){$d$}}
\put(80,30){\makebox(0,0){$e$}}

\put(160,20){\vector(1,0){40}}
\put(164,25){\makebox(0,0)[l]{\tiny{splicing}}}

\put(240,20){\circle*{4}}
\put(370,20){\circle*{4}}
\put(240,20){\vector(1,0){30}}
\put(240,20){\line(-2,-1){30}} \put(240,20){\line(-2, 1){30}}
\put(220,23){\makebox(0,0){$\vdots$}}
%\put(400,0){\makebox(0,0){$v$}}
\put(395,26){\makebox(0,0){$d'$}}
\put(330,20) {\makebox(0,0){\tiny{$i'-1$}}}

\dashline[3]{3}(350,20)(370,20)\put(350,20){\vector(-1,0){5}}

\put(400,20){\circle*{4}}
\put(400,20){\line(2,-1){30}} \put(400,20){\line(2, 1){30}}
\put(284,20){\makebox(0,0){$(M)$}}
\put(427,23){\makebox(0,0){$\vdots$}}
\put(400,20){\line(-1,0){30}}

\put(400,50){\makebox(0,0){$\Gamma_R$}}
\put(260,50){\makebox(0,0){$\Gamma_L$}}
\put(230,35) {\makebox(0,0){$d_1$}}
\put(230,5){\makebox(0,0){$d_{n}$}}

\end{picture}

 (1) Then $i'<0$.

 (2)  Assume that the semigroup condition is satisfied in $\Gamma(F)$ (at least for the edge $e$ and
 the edges $e_L$ sitting in  $\Gamma_L$).  (This means that  $d'$ is in the  semigroup ${\mathcal S}_e$
 generated by $\{l_{ew}\}_w$ where the index $w$ runs over
 all the boundary vertices of $\Gamma _L$, and there are similar inclusions for all edges $e_L$ of
 $\Gamma_L$.)   Then
 $-i'\not\in {\mathcal S}_e$. Hence, $d'\nmid i'$.
\end{proposition}

\begin{proof} (1)  Set $D:=\prod_{j=1}^n d_j$. We proceed by induction on the number of nodes in
$\Gamma_L$.

%\begin{picture}(100,55)(200,-5)
%\put(370,20){\circle*{4}}
%\put(370,20){\vector(1,0){30}}
%\put(370,0){\makebox(0,0){$v_R$}}\put(300,20){\makebox(0,0){$\Gamma_R:$}}
%\put(370,20){\line(-2,-1){30}} \put(370,20){\line(-2, 1){30}}
%\put(387,23){\makebox(0,0){$\vdots$}}\put(410,20){\makebox(0,0){$\ldots$}}
%\dashline[3]{3}(240,22)(265,22)\put(265,22){\vector(1,0){5}}
%\dashline[3]{3}(370,22)(395,22)\put(395,22){\vector(1,0){5}}
%\end{picture}

 Suppose first that $v_L$ is the only
node of $\Gamma_L$. By (\ref{eq:i}) we have that $i'=\sum_{j=1}^n
D/d_j - (n-1)D$. When $n=2$, this is $d_1+d_2-d_1d_2$ and thus
negative.  When $n>2$, then by minimality $d_j\geq 2$, hence $\sum
1/d_j\leq n/2<n-1$, therefore $i'<0$ again.
%$$
%i=d_r\left( \sum_{j=1}^{r-1}(\prod_{\ell=1}^{r-1}d_\ell)/d_j -
%(r-2) (\prod_{\ell=1}^{r-1}d_\ell)\right) + (D/d_r - D).
%$$
%Since both terms are negative (the first one by induction on $r$),
%so is $i$.

We  suppose now  that $\Gamma_L$ contains at least two nodes. From
(\ref{eq:i}) we can write $i'$ as
$$
i'=\sum_{j=1}^n \frac{D}{d_j}\left(\sum_{w\in\calV_j} (2-\delta_w)
\ell_{e_jw} \right) - (n-1)D,
$$
where for $j=1,\dots,n$ the set $\calV_j$ consists of the vertices
of $\Gamma_L$ connected (geodesically) to $v_L$ through the edge
$e_j$ with weight $d_j$, and $\ell_{e_jw}$ is the product of all
the decorations adjacent to, but not on, the path from $w$ to
$e_j$. For all $j$ this sum is either equal to 1 (when $e_j$ ends
at a boundary vertex), or negative by induction. Since at least
one sum is negative, we conclude that $i'<0$.

(2) Denote by $\calB_L$ and $\calN_L$ the boundary vertices and
nodes, respectively, in  $\Gamma_L$.  We will
show the following claim. {\it Let
$$
I:=\sum_{w\in\calB_L}m_w \ell_{ew} + \sum_{w\in\calN_L}(2-\delta_w)
\ell_{ew},
$$
where all $m_w \in \Z_{>0}$. Then $I\neq 0$ and, if $I < 0$, then
$-I\not\in {\mathcal S}_e$.}

Since $i'=I$ when all $m_w=1$, and $i'<0$ by part (1), the statement
then follows.

\medskip
We now prove the claim, again by induction on the cardinality of
$\calN_L$. If $\calN_L=\{v_L\}$ then
$$
I=\sum_{j=1}^n m_j \frac D{d_j} - (n-1)D,
$$
and this is nonzero by Lemma \ref{le:elem}(a). If $I<0$ and $-I\in{\mathcal S}_e$, then
%$\lambda d= -\sum_{j=1}^r m_w \frac D{d_j} + (r-1)D$ for some
%$\lambda \in \Z_{> 0}$. Since $d$ belongs to the semigroup
%generated by $D/d_1,\dots,D/d_r$, we conclude that
$
\sum_{j=1}^n k_j \frac D{d_j} = (n-1)D
$
for some positive integers $k_j$, contradicting again Lemma \ref{le:elem}(a).

Let now $\calN_L$ have at least two elements. Suppose again that
$I< 0$ and $-I\in{\mathcal S}_e$. Then analogously
%since $d$ belongs to the semigroup generated by the $\ell_w, w\in \calB_R,$
we get that
$$
\sum_{w\in\calB_L} k_w\ell_{ew} = \sum_{w\in\calN_L}
(\delta_w-2)\ell_{ew}
$$
for some positive integers $k_w$. We separate $v_L$ (with
$\delta_{v_L}=n+1$ and $\ell_{ev_L}=D$) on the right hand side, and
rewrite this equality  as
\begin{equation}\label{eq:ind}
\sum_{j=1}^n  \frac D{d_j}\left[ \sum_{w\in\calB_L^{(j)}} k_w
\ell_{e_jw} +  \sum_{w\in\calN_L^{(j)}} (2-\delta_w)
\ell_{e_jw}\right] = (n-1)D,
\end{equation}
where $\calB_L^{(j)}\cup\calN_L^{(j)} =\calV_j$. By induction the
square bracket is non--zero. Moreover, (\ref{eq:ind}) shows that
 $d_j$ divides the $j$-th  square bracket for all $j$.
 Applying the induction hypothesis on all these terms
 (and the assumption $d_j\in {\mathcal S}_{e_j}$)
yields that all these square brackets  are positive. But this contradicts Lemma
\ref{le:elem}(a).

We still have to show that $I\neq 0$. Assuming that $I=0$ yields
the same expression as in (\ref{eq:ind}), with the $k_w$ replaced
by the original $m_w$. And then we obtain a contradiction by the
same argument.
\end{proof}

\begin{remark}\labelpar{re:Milnornumber} Let us give the `Milnor number interpretation'
of the statement (\ref{prop:div})(1). Consider the splice diagram
$\Gamma_L$, but replace the multiplicity $M$ of the unique
arrowhead by 1. This represents a fibrable knot; let  $S$ be its
fiber. It is a connected punctured Riemann surface. Let its first
Betti number be $\mu$ (the Milnor number). Clearly, $\mu$ is even.
Since $\Gamma_L$ is minimal and non--empty, $\mu\not=0$ (its proof
is basically our proof of (1)). On the other hand, by \cite{AC},
$i'=\chi(S)=1-\mu$, hence $i'<0$.

The second part also has some `classical' interpretation. Start
again with the fact $-i'=\mu-1$, and assume that the above diagram
represents a plane curve singularity. Then  ${\mathcal S}_e$ is
exactly the semigroup ${\mathcal S}$ of the plane curve, and it is
a classical fact that $\mu-1$ is the largest integer not in
${\mathcal S}$.

The point is that in any  generalization  of $\mu-1\not\in{\mathcal S}$
for more general $\Gamma_L$ (as our (2) does)
 one needs some restriction about $\Gamma_L$: for example,  if we have two nodes,
the second one in ${\mathcal V}_1$, and $d_1=1$, then ${\mathcal
S}_e=\mathbb{N}$.
\end{remark}

\subsection{$W=0$ is allowed.}

\begin{theorem} \labelpar{thm:W} Let $F$ be a (nonzero) effective divisor on an
IHS germ $(X,0)$, such that the minimal embedded resolution
diagram $\GaxF$ satisfies the semigroup condition. Then the
divisor $W=0$ is allowed for the pair $(X,F)$.
\end{theorem}

\begin{proof}
Denoting by $\pi$ this minimal embedded resolution, we will show
that the diagram $\Gamma= \Gamma_\pi(X,F,W=0)$ is allowed. Note
that thus the decorations $i_a-1=0$ for all $a\in \calA_W$.

Recall again that on any star-shaped subdiagram of $\Gamma$
without boundary vertex the allowedness condition is trivially
satisfied. If a star-shaped subdiagram contains a boundary vertex,
that is an original boundary vertex of $\Gamma$, then the
corresponding leg decoration (being $>1$ by minimality) does not
divide the associated $i_a(=1)$.

If a star-shaped subdiagram contains a boundary vertex, that is
created after splicing, that diagram looks like the right diagram
in the statement of Proposition \ref{prop:div}, where $i'-1$ is
the decoration of the constructed dashed arrow attached to that
boundary vertex, and $d'$ is the corresponding edge weight. Since
we showed in Proposition \ref{prop:div} that $d'\nmid i'$, the
allowedness condition is verified in this case too (hence
everywhere)
 by  Addendum (\ref{eq:addendum}).
\end{proof}

\begin{example}\labelpar{ex:EXAMPLE6}
 Recall that in Example~\ref{ex:EXAMPLE4}
 we presented a minimal diagram $\Gamma(F)$ for which  $W=0$ is not allowed.
 Hence some kind of restriction is indeed necessary in order to guarantee the allowedness
 of $W=0$.
\end{example}

\begin{corollary} Let $(X,0)$ be a Gorenstein IHS germ, with
nowhere vanishing 2-form $\omega_0$ on $X\setminus\{0\}$. Let $f$
be a function germ on $(X,0)$ such that the minimal embedded
resolution diagram $\Gaxf$ satisfies the semigroup condition.

If $s_0$ is a pole of the topological zeta function
$Z(f;s)=Z(\div(f),W=0;s)$, then $\exp(2\pi is_0)$ is a monodromy
eigenvalue of $f$ at some point of $\{f=0\}$.
\end{corollary}

\begin{proof} Immediate from Theorem \ref{thm:W} and Theorem
\ref{thm:monconj}.
\end{proof}

\begin{remark}
In Theorem 2.2 of \cite{Ro}, Rodrigues showed without requiring the semigroup condition
that even in the singular setting $\exp(2\pi is_0)$ is a monodromy eigenvalue
of $f$ provided that the pole $s_0$ satisfies $s_0 \leq 0$. We  wish to emphasize that this is a rather strong assumption
in the context of singular ambient spaces. Indeed, if we consider a {\em minimal}
resolution of a non--canonical surface singularity, then  the canonical cycle $K$ is nef, and all its
 coefficients $\nu-1$ are (strictly) negative, hence the corresponding values $-\nu/N$ are all non-negative.
 We can create negative poles when we have to blow up the minimal resolution in order to get a good resolution
 of $(X,F)$, so by subgraphs which behave like graphs of plane curve singularities.
 Usually there are only a few poles like this (although,  for plane curves all of them are negative
 by the very same argument).
\end{remark}

\subsection{All eigenvalues are realized by poles}
The main result of the subsection is based on the following technical proposition.

\begin{proposition}\labelpar{prop:TECH}
Let $(X,0)$ be an IHS germ and $f$ an analytic function on $X$,
such that the minimal embedded resolution diagram
$\Gamma=\Gamma_\pi(X,f)$ satisfies the semigroup condition. With
the notation of Remark~\ref{re:extend}(b), let us fix a node $v_0$
in $\Gamma_\calA$,  and another node  $v_m$ not in  $\Gamma_\calA$
at `distance' $m\geq 1$ from $v_0$. Consider the diagram
$\Gamma_0$ given below  obtained from $\Gamma$ by cutting via
splice--decomposition all the nodes not sitting on the geodesic
path connecting $v_0$ with $v_m$. Here the legs with decorations
$d_{0,2},\ldots, d_{0,n_0}$ are optional, and all boundary
vertices are either original boundary vertices of\, $\Gamma$, or
are obtained after splicing.

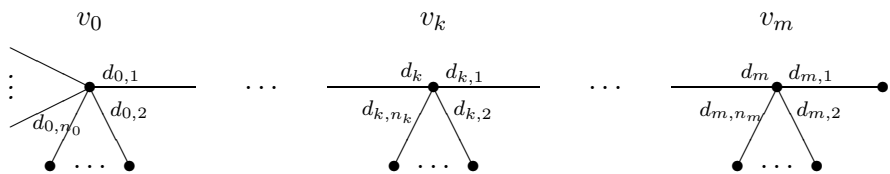
\begin{picture}(400,85)(-30,0)
\put(40,40){\circle*{4}} %\put(120,40){\circle*{4}}
\put(170,40){\circle*{4}} %\put(250,40){\circle*{4}}
\put(300,40){\circle*{4}}
\put(340,40){\circle*{4}}

% \put(105,10){\circle*{4}}
\put(155,10){\circle*{4}}
%\put(235,10){\circle*{4}}
\put(285,10){\circle*{4}}
\put(25,10){\circle*{4}}

% \put(135,10){\circle*{4}}
\put(185,10){\circle*{4}} %\put(265,10){\circle*{4}}
\put(315,10){\circle*{4}}
\put(55,10){\circle*{4}}

\put(40,40){\line(1,0){40}}\put(130,40){\line(1,0){80}}
\put(260,40){\line(1,0){80}}
\put(40,40){\line(-2,1){30}}\put(40,40){\line(-2,-1){30}}

\put(40,40){\line(-1,-2){15}}
\put(170,40){\line(-1,-2){15}} %\put(120,40){\line(1,-2){15}}
%\put(250,40){\line(-1,-2){15}}
\put(300,40){\line(-1,-2){15}}

\put(40,40){\line(1,-2){15}}
\put(170,40){\line(1,-2){15}} %\put(120,40){\line(-1,-2){15}}
%\put(250,40){\line(1,-2){15}}
\put(300,40){\line(1,-2){15}}

\put(52,45){\makebox(0,0){\tiny{$d_{0,1}$}}}
\put(28,26){\makebox(0,0){\tiny{$d_{0,n_0}$}}}
\put(55,31){\makebox(0,0){\tiny{$d_{0,2}$}}}

\put(313,45){\makebox(0,0){\tiny{$d_{m,1}$}}}
\put(283,31){\makebox(0,0){\tiny{$d_{m,n_m}$}}}
\put(316,31){\makebox(0,0){\tiny{$d_{m,2}$}}}
\put(292,46){\makebox(0,0){\tiny{$d_m$}}}

\put(182,45){\makebox(0,0){\tiny{$d_{k,1}$}}}
\put(153,31){\makebox(0,0){\tiny{$d_{k,n_k}$}}}
\put(185,31){\makebox(0,0){\tiny{$d_{k,2}$}}}
\put(162,46){\makebox(0,0){\tiny{$d_k$}}}

%\put(208,40){\makebox(0,0){$\ldots$}}
\put(10,43){\makebox(0,0){$\vdots$}}
\put(40,65){\makebox(0,0){$v_0$}}
%\put(120,65){\makebox(0,0){$v_1$}}
\put(170,65){\makebox(0,0){$v_k$}}
%\put(250,65){\makebox(0,0){$v_{m-1}$}}
\put(300,65){\makebox(0,0){$v_m$}}

\put(40,10){\makebox(0,0){$\ldots$}}
\put(105,40){\makebox(0,0){$\ldots$}}
\put(170,10){\makebox(0,0){$\ldots$}}
\put(235,40){\makebox(0,0){$\ldots$}}
\put(300,10){\makebox(0,0){$\ldots$}}
\end{picture}

Splice this diagram at the edge $(v_{k-1}v_{k}), 0<k\leq m,$ and
denote that splice component which contains $v_k,\ldots, v_m$ by
$\Gamma_k=\Gamma_k(F_k)$. Let the decoration of the
 dashed arrowhead  not in $\Gamma_1$, associated with the splicing   along $(v_0v_1)$, be
 $i_{0,1}-1$ (see the picture below).

Fix a root $\lambda$
of the Alexander polynomial $\Lambda_{\Gamma_m}(t)$ associated with the star--shaped diagram
$\Gamma_m$.   Then there exist infinitely many allowed divisors $W_m$ for
$\Gamma_m$, such that if $(N_m,\nu_m-1)$ denote the decorations of $v_m$ as above
associated with $\Gamma_m$ and $W_m$, then

\vspace{1mm}

(1) \ $s_0=-\nu_m/N_m$ is a pole of $Z(F_m,W_m;s)$, with
$\exp(2\pi i s_0)=\lambda$, and

(2) \ $W_m$ extends to an allowed divisor on $\Gamma_1$, such that
 $d_{0,1}\nmid i_{0,1}$.

\vspace{1mm}
\noindent Moreover, infinitely many of these allowed (extended) $W_m$ on $\Gamma_1$, as well as their further extensions (in the sense of Remark \ref{re:extend}(b)) on the whole of  $\Gamma$, may be chosen to be effective.

\end{proposition}
Note that the above additional non--divisibility property (2) is the key
assumption in  the  Addendum (\ref{eq:addendum}).

\begin{proof} We proceed in several steps. During the proof $\lambda $ is fixed.

\bekezdes\labelpar{A}
 We fix notations for the wanted $W_m$ and its extensions on the spliced
star--shaped subdiagrams of $\Gamma_1$.
Moreover, we  also consider the decoration  $i_{0,1}$ which is part of a potential extension to
$\Gamma_0$, but it is completely determined by the extension on $\Gamma_1$.

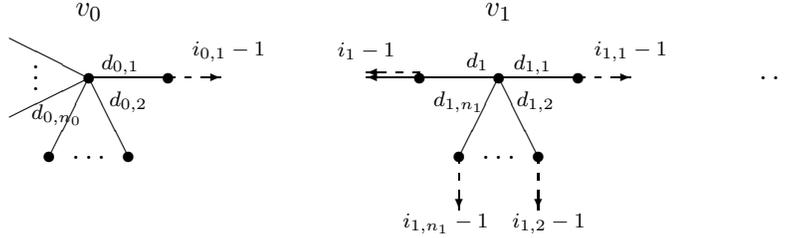
\begin{picture}(450,110)(0,-25)
\put(70,40){\circle*{4}} \put(100,40){\circle*{4}}
\put(70,40){\line(1,0){30}}
\put(70,40){\line(-2,1){30}}\put(70,40){\line(-2,-1){30}}
\put(70,40){\line(-1,-2){15}}
\put(70,40){\line(1,-2){15}}
\put(70,10){\makebox(0,0){$\ldots$}}
\put(50,43){\makebox(0,0){$\vdots$}}
\dashline[3]{3}(100,40)(115,40)\put(115,40){\vector(1,0){5}}
\put(123,50){\makebox(0,0){\tiny{$i_{0,1}-1$}}}
\put(55,10){\circle*{4}}
\put(85,10){\circle*{4}}
\put(82,45){\makebox(0,0){\tiny{$d_{0,1}$}}}
\put(58,26){\makebox(0,0){\tiny{$d_{0,n_0}$}}}
\put(85,31){\makebox(0,0){\tiny{$d_{0,2}$}}}
\put(70,65){\makebox(0,0){$v_0$}}

\put(225,40){\circle*{4}}\put(195,40){\circle*{4}}
\put(255,40){\circle*{4}}
 \put(210,10){\circle*{4}}
 \put(240,10){\circle*{4}}
\put(225,40){\line(1,-2){15}}
\put(225,40){\line(-1,-2){15}}
\put(238,45){\makebox(0,0){\tiny{$d_{1,1}$}}}
\put(210,31){\makebox(0,0){\tiny{$d_{1,n_1}$}}}
\put(239,31){\makebox(0,0){\tiny{$d_{1,2}$}}}
\put(217,46){\makebox(0,0){\tiny{$d_1$}}}
\put(225,65){\makebox(0,0){$v_1$}}
\put(225,10){\makebox(0,0){$\ldots$}}
\put(255,40){\vector(-1,0){80}}
\dashline[3]{3}(195,42)(180,42)\put(180,42){\vector(-1,0){5}}
\dashline[3]{3}(210,10)(210,-5)\put(210,-5){\vector(0,-1){5}}
\dashline[3]{3}(255,40)(270,40)\put(270,40){\vector(1,0){5}}
\dashline[3]{3}(240,10)(240,-5)\put(240,-5){\vector(0,-1){5}}
\put(175,50){\makebox(0,0){\tiny{$i_1-1$}}}
\put(275,50){\makebox(0,0){\tiny{$i_{1,1}-1$}}}
\put(205,-15){\makebox(0,0){\tiny{$i_{1,n_1}-1$}}}\put(244,-15){\makebox(0,0){\tiny{$i_{1,2}-1$}}}

%\put(300,40){\circle*{4}}\put(270,40){\circle*{4}}
%\put(330,40){\circle*{4}}
% \put(285,10){\circle*{4}}
% \put(315,10){\circle*{4}}
%\put(300,40){\line(1,-2){15}} \put(300,40){\line(-1,-2){15}}
%\put(310,45){\makebox(0,0){\tiny{$d_{2,1}$}}}
%\put(285,31){\makebox(0,0){\tiny{$d_{2,n_2}$}}}
%\put(314,31){\makebox(0,0){\tiny{$d_{2,2}$}}}
%\put(292,46){\makebox(0,0){\tiny{$d_2$}}}
%\put(300,65){\makebox(0,0){$v_2$}}
%\put(300,10){\makebox(0,0){$\ldots$}}
%\put(330,40){\vector(-1,0){80}}
%\dashline[3]{3}(270,42)(255,42)\put(255,42){\vector(-1,0){5}}
%\dashline[3]{3}(285,10)(285,-5)\put(285,-5){\vector(0,-1){5}}
%\dashline[3]{3}(330,40)(345,40)\put(345,40){\vector(1,0){5}}
%\dashline[3]{3}(315,10)(315,-5)\put(315,-5){\vector(0,-1){5}}
%\put(260,50){\makebox(0,0){\tiny{$i_2-1$}}}
%\put(340,50){\makebox(0,0){\tiny{$i_{2,1}-1$}}}
%\put(280,-15){\makebox(0,0){\tiny{$i_{2,n_2}-1$}}}\put(319,-15){\makebox(0,0){\tiny{$i_{2,2}-1$}}}

%\put(415,40){\circle*{4}}\put(385,40){\circle*{4}}
% \put(400,10){\circle*{4}}
% \put(430,10){\circle*{4}}
%\put(415,40){\line(1,-2){15}}
%\put(415,40){\line(-1,-2){15}}
%\put(415,65){\makebox(0,0){$v_3$}}
%\put(415,10){\makebox(0,0){$\ldots$}}
%\put(425,40){\vector(-1,0){60}}
%\dashline[3]{3}(385,42)(370,42)\put(370,42){\vector(-1,0){5}}
%\put(375,50){\makebox(0,0){\tiny{$i_3-1$}}}

\put(330,40){\makebox(0,0){$\ldots$}}
\end{picture}

\noindent
We also set $D_k:=\prod_{\ell=1}^{n_k} d_{k,\ell}$
and $D_k^*:=\prod_{\ell=2}^{n_k} d_{k,\ell}$ ($1\leq k\leq m$). Note that
$n_k\geq 2$. %  by minimality of the diagram.
%It is convenient to denote
%the {\it subgroup} of $\Z$ generated by $g_1,\ldots g_t$ by $\Z\langle g_1,\ldots, g_t\rangle$, while
%the {\it subsemigroup} of $\N$ generated by  $g_1,\ldots g_t$ by $\calS\langle g_1,\ldots, g_t\rangle$.

The semigroup condition for $\Gamma$ implies that for any $k>0$ one has
\begin{equation}\label{eq:*}
d_{k-1,1}\in \calS\Big\langle
\frac{D_k}{d_{k,2}}, \ldots, \frac{D_k}{d_{k,n_k}},
\frac{D_k^*D_{k+1}}{d_{k+1,2}},\ldots ,
 \frac{D_k^*D_{k+1}}{d_{k+1,n_{k+1}}},
\frac{D_k^*D_{k+1}^*D_{k+2}}{d_{k+2,2}}, \cdots
\Big\rangle
\end{equation}
%This implies  several simplified inclusions, the first (whose generators involve only $v_k$) is
%\begin{equation}\label{eq:*b}
%d_{k-1,1}\in \calS\Big\langle
%\frac{D_k}{d_{k,2}}, \ldots, \frac{D_k}{d_{k,n_k}},
%D_k^*
%\Big\rangle=
%\calS\Big\langle
%\frac{D_k}{d_{k,1}}, \frac{D_k}{d_{k,2}}, \ldots, \frac{D_k}{d_{k,n_k}}
%\Big\rangle,
%\end{equation}
%and there is a relation for each `segment' $(v_k, v_{k+1},\ldots, v_{k'})$.

The wanted divisor will be constructed by induction.
 From (\ref{le:pole}) we know that infinitely many $W_m$, even infinitely many effective $W_m$, satisfying (1)
exist (see also  \ref{C}). Here
$W_m$ identifies $\nu_m$ by
\begin{equation}\label{eq:nu2}
\nu_m=i_{m-1,1} d_m+i_mD_m.
\end{equation}
Then, we analyze how an allowed divisor $W_{k+1}$  from
$\Gamma_{k+1}$ can be extended over $\Gamma_k$. Along this
procedure we will use the following identities `around $v_k$'
satisfied by any extension:
\begin{equation}\label{eq:h}
i_{k+1}=-(n_k-1)d_kD_k^*+i_kD_k^*+\sum _{\ell\geq 2}i_{k,\ell}
d_kD^*_k/d_{k,\ell},
\end{equation}
\begin{equation}\label{eq:j1a}
i_{k-1,1}=-(n_k-1)D_k+\sum_{\ell\geq 1} i_{k,\ell} D_k/d_{k,\ell}.\end{equation}
In this procedure we need a deeper understanding of the extensions (compared with
(\ref{ss:extend})), and
we need to consider divisors with some special properties, we will call them `strict'.
The decorations of their nodes satisfy some additional conditions, as it is explained next.

Assume that $W_k$ is an allowed  divisor on $\Gamma_k$ for some
$k\geq 1$. The decorations of $W_k$ will distinguish  the nodes as
follows. For some $k'\in \{k,\ldots, m\}$, the node  $v_{k'} $ is
called {\it flexible} if there are at least two indexes
$\ell\in\{1,\ldots, n_{k'}\}$ for which $d_{k',\ell}\nmid
i_{k',\ell}$. For $k'\in \{k,\ldots, m-1\}$, if $v_{k'}$ is not
flexible, but $d_{k',\ell}=  i_{k',\ell}$  for all
$\ell\in\{2,\ldots, n_{k'}\}$ then  it is called {\it rigid}. Note
that not all non--flexible nodes are rigid (see the cases
discussed in (\ref{re:extend})).

In this proof the nodes of all allowed divisors will be  either flexible or rigid.
\bekezdes\labelpar{C}  First we construct an allowed divisor $W_m$ which satisfies (1) and is flexible at $v_m$.
We search for $i_m,i_{m,1},\ldots, i_{1,n_m}$ such that they satisfy the allowedness at $v_m$,
(\ref{eq:nu2}), (\ref{eq:j1a}) for $k=m$,
and $\exp(-2\pi i\nu_m/N_m)$ is root
of the Alexander polynomial $\Lambda_{\Gamma_m}(t)$. Since $D_m\mid N_m$,
the last condition implies that
\begin{equation}\label{eq:LAM}
d_{m,\ell}\nmid \nu_m\ \ \mbox{for at least two indexes
$\ell\in\{1,\ldots,n_m\}$}.
\end{equation}
We proceed as follows. For any  $\nu_m$ with (\ref{eq:LAM}) we
find  $i_{m-1,1}$ and $i_m$ satisfying (\ref{eq:nu2}). This is
possible since $\gcd(d_m,D_m)=1$. Then we find integers
$\{i_{m,\ell}\}_{\ell=1}^{n_m}$ satisfying
 (\ref{eq:j1a}). This, again, is possible since
$\gcd_\ell(\frac{D_m}{d_{m,\ell}})=1$. Since $d_{m,\ell}\mid \nu_m
\Leftrightarrow d_{m,\ell} \mid i_{m-1,1} \Leftrightarrow
d_{m,\ell}\mid i_{m,\ell}$,
 by (\ref{eq:LAM}) we have that $d_{m,\ell}\nmid i_{m,\ell}$ for at least two indexes, hence
$W_m$ is flexible at $v_m$.

%\bekezdes\labelpar{D}  Consider the situation of (\ref{C}) and the divisor constructed there.
%Then (\ref{B}) shows that it can be replaced by another one which additionally satisfies
%$d_{m-1,1}\nmid i_{m-1,1}$ too.
%In particular,  this
% also proves the main statement for $m=1$. % , as the starting step of the induction.

\bekezdes\labelpar{E}  Next we analyze the possibilities  how one
can extend divisors. Consider an allowed divisor $W_{k+1}$ on
$\Gamma_{k+1}$ ($1\leq k < m$). Note that it also determines
$i_{k,1}$ by (\ref{eq:j1a}). Extending over $v_k$ means that  we
already know everything over $\Gamma_{k+1}$ and $i_{k,1}$, and we
are searching for $i_k$ and $\{i_{k,\ell}\}_{\ell\geq 2}$ which
satisfy the  allowedness condition at $v_k$ and the identity
(\ref{eq:h}).

The divisor $W_{k+1}$ and the decorations of $\Gamma_0(F)$ contain
all the divisibility information, like $d_{k,\ell}$ divides
$i_{k,\ell}$ or not, for any extension $W_k$ on $\Gamma_k$.
Indeed, $i_{k,1}$ and $d_{k,1}$  are determined by  $\Gamma_0$ and
$W_{k+1}$, and the divisibility conditions $d_{k,\ell}\nmid
i_{k,\ell}$ ($\ell\geq 2$) are determined by (\ref{eq:h}), since
$d_{k,\ell}\nmid i_{k,\ell} \Leftrightarrow d_{k,\ell}\nmid
i_{k+1}$. Hence, several crucial divisibility properties of an
extension $W_k$ on $\Gamma_k$ are already decided at the level of
its restriction $W_m$ on $\Gamma_m$. This makes the inductive
construction of $W_k$, staring from $W_m$ `global' and difficult.

In order to guarantee the existence of such an extension $W_k$, we
will use two types of criteria: $W_{k+1}$ satisfies  either
$d_{k,1}\nmid i_{k,1}$ or $D^*_k\mid i_{k+1}$ (see
(\ref{eq:addendum})).

If $D^*_k\mid i_{k+1}$ then it has no flexible extension (but, it
might happen that it has several allowed extensions);  we take
always that unique extension for which $v_k$ will be rigid:
$d_{k,\ell}=i_{k,\ell}$ for $\ell\geq 2$. Moreover,  (\ref{eq:h})
and (\ref{eq:j1a})  read as
\begin{equation}\label{eq:rigidh}
i_{k}=i_{k+1}/D_k^* \ \ \mbox{and } \ \ i_{k-1,1}=i_{k,1}D_k^*.\end{equation}

If  $d_{k,1}\nmid i_{k,1}$, then for  any extension $W_k$  the
node $v_k$ is  either flexible or rigid (and the type is decided
already at the level of $W_{k+1}$);
%If it is flexible, then by (\ref{B}) it can be replaced by $\bar{W}_{k+1}$ which has the additional property
%$d_{k-1,1}\nmid i_{k-1,1}$. Hence in this case the inductive construction can be continued.
%(This for an arbitrary $k$ reads as follows. If  $d_{k,1}\nmid d_{k,1}$
%is satisfied by $W_{k+1}$, then it can be extended to a $W_k$; if $v_k$ for this one is
%flexible, then we can assume via (B)  that  $d_{k-1,1}\nmid d_{k-1,1}$ too, hence the induction runs.)
%For the extension $W_k$ the node
$v_k$  is rigid if and only if additionally $D_k^*\mid i_{k+1}$, the case discussed before.
If $v_k$ is flexible, then the extension is not unique, it can be modified if it is necessary (and we will
do this intensively).

Next, we have to check if the extension has one of the two criteria which guarantee the further extension.
We show that if we `modify $W_k$ at the closest flexible node', it will satisfy the inductive
criteria $d_{k-1,1}\nmid i_{k-1,1}$, provided that the tower of extensions was carefully constructed
from the beginning.   The careful choice of the sequence of flexible/rigid nodes and the family of modifications is described in the next part.

\bekezdes\labelpar{E2}
We define the class of {\it strict} allowed divisors $W_k$ on $\Gamma_k$ inductively as follows.

Assume first that $v_k$ is rigid, but at least one node of $(\Gamma_k,W_k)$
is flexible. Let $k'>k$ be that flexible node for which
$v_k,\ldots, v_{k'-1}$ are all rigid. We modify $W_k$ such that we keep unmodified
the restriction on $W_{k'+1}$ and $i_{k',1}$. We fix some $\ell\geq 2$ such that
$d_{k',\ell} \nmid i_{k',\ell}$.
Then we replace $i_{k',\ell}$ into $i_{k',\ell}+td_{k',\ell}$, $t\in\Z$, but keep all other
$i_{k',\ell}$'s. Moreover, modify $i_{k'}\mapsto i_{k'}-td_{k'}$ too.
Then $i_{k'+1}$ and $i_{k',1}$ will stay fixed.

This is the set of modifications we will refer to,
and for {\it strict divisors} we impose the following properties. First,  we assume that
for all the possible modifications, the value  $i_{k'}$ is multiple of
$D^*_{k'-1}$. Then, all these modifications can be extended by a rigid $v_{k'-1}$ to $\Gamma_{k'-1}$.
Then we run again all the modifications (at $v_{k'}$) and we assume that for all of them
$D^*_{k'-2}\mid i_{k'-1}$. Then, again, all of them can be extended. We continue this, at the very end asking $D^*_{k}\mid i_{k+1}$ for all the modifications. If all these conditions are satisfied for $W_k$
then in all its modifications $\bar{W}_k$ the nodes $v_k,\ldots, v_{k'-1}$ will be rigid, and
we call $W_k$ strict. The strictness guarantees that when we run all the modifications at the level
of $v_{k'}$, all the divisors can be extended to some $W_k$. (Otherwise it might happen that
for some modification and at some vertex  both  $d_{k'',1}\nmid i_{k'',1}$ and
 $D^*_{k''}\mid i_{k''+1}$ fail.)

From (\ref{eq:rigidh}) we get
\begin{equation}\label{eq:rigidhh}
i_{k}=\frac{i_{k'}}{D^*_{k}\cdots D^*_{k'-1}}, \ \ \ \mbox{and } \ \ \
i_{k-1,1}=i_{k'-1,1}D_{k}^*\cdots D_{k'-1}^*.
\end{equation}
Since $i_{k'-1,1}\mapsto i_{k'-1,1}+ t D_{k'}$, the modifications induce
\begin{equation}\label{eq:rigidhhh}
i_{k}\mapsto i_{k}-\frac{td_{k'}}{D^*_{k}\cdots D^*_{k'-1}}, \ \ \ \mbox{and } \ \ \
i_{k-1,1}\mapsto i_{k-1,1}+tD_{k}^*\cdots D_{k'-1}^*D_{k'}.
\end{equation}
%The stricness condition then implies for any $k<k''\leq k'$:
%\begin{equation}\label{eq:STRICT}
%D^*_{k''-1}\,\Big| \,\frac{d_{k'}}{D^*_{k''}\cdots D^*_{k'-1}}
%\end{equation}

If $v_k$ is flexible then $W_k$ is strict by definition. In fact
the above discussion is valid in this case too with $k'=k$. In
particular, the set of modifications is given by
$i_{k,\ell}\mapsto i_{k,\ell}+td_{k,\ell}$ for the chosen $\ell$
and keeping the other $i_{k,\ell}$'s,  $i_{k}\mapsto
i_{k}-td_{k}$, $i_{k-1,1}\mapsto i_{k-1,1}+tD_k$.
%Its set of modifications are exactly those  used in (\ref{B}).

If we run the above modification for the divisors $W_m$
constructed in (\ref{C}), then $\nu_m$ stays stable, hence if the
restriction of some $W_k$ to $\Gamma_m$ satisfies (1), then all
its modifications keep satisfying (1).

In our procedure we  consider only strict allowed divisors. They
will be constructed inductively starting from the strict divisors
$W_m$ constructed in (\ref{C}). The inductive statement we prove
is the following: {\it
 for any $1\leq k\leq m$ there exists a strict allowed divisor $W_k$ on $\Gamma_k$
 satisfying $d_{k-1,1}\nmid i_{k-1,1}$, and (1) on $\Gamma_m$.  }

\vspace{2mm}

The proof of the inductive step breaks into two parts.

\vspace{2mm}

(a) If the above properties are  true for some strict $W_{k+1}$ on $\Gamma_{k+1}$
then definitely it can be extended
to an allowed divisor $W_k$, but this is not necessarily strict.
We prove that by a  good choice of one of its modifications, that divisor  has a strict   extension
(not necessarily satisfying $d_{k-1,1}\nmid i_{k-1,1}$).

(b) If $W_k$ is strict and its restriction  satisfies (1), then it can be replaced (by the above moves) by another strict divisor  which satisfies
both (1) and $d_{k-1,1}\nmid i_{k-1,1}$.

\vspace{2mm}

Note that part (b) provides the main inductive statement for $m=1$ too. Indeed, by (\ref{C}) a
strict divisor $W_m$ with (1) exists, which by  (b)  can be replaced by a wanted one.

\bekezdes\labelpar{F} Here we prove part (a) of the inductive step (\ref{E2}).

Assume that $W_{k+1}$ is a strict divisor on $\Gamma_{k+1}$
satisfying (1) and $d_{k,1}\nmid i_{k,1}$. We consider all the
modifications $\bar{W}_{k+1}$ of $W_{k+1}$ as in (\ref{E2}), and
we distinguish the next two cases.

First, suppose that there is no
 $\bar{W}_{k+1}$ (with or without $d_{k,1}\nmid i_{k,1}$) for which
$D_k^*\nmid i_{k+1}$. Then we  extend $W_{k+1}$ by a
rigid node. The extended divisor $W_k$ will be strict.

Second, we assume that
 there exist some   $\bar{W}_{k+1}$ with  $D_k^*\nmid i_{k+1}$.
The problem is that it might happen that in the new situation
$d_{k,1}\nmid i_{k,1}$ fails, and the extension is not guaranteed.
We claim  that the two conditions $d_{k,1}\nmid i_{k,1}$
and $D_k^*\nmid i_{k+1}$ can be obtained simultaneously by some $\bar{W}_{k+1}$.
Then we extend this new $\bar{W}_{k+1}$
to get a strict $W_k$ with flexible $v_k$.

Let us prove now the above claim.

Recall that $d_{k,1}\nmid i_{k,1}$. If $D^*_k\nmid i_{k+1}$ for $W_{k+1}$ then we are done.
Similarly, if $d_{k,1}\nmid i_{k,1}$ for $\bar{W}_{k+1}$ then again we are done. Otherwise,
by (\ref{eq:rigidhhh}) we must have
\begin{equation}\label{eq:DIV}
d_{k,1}\nmid \delta:=D_{k+1}^*\cdots D_{k'-1}^*D_{k'} \ \ \ \mbox{and } \ \ \
D^*_k\nmid \Delta:=\frac{d_{k'}}{D^*_{k+1}\cdots D^*_{k'-1}}.
\end{equation}
We consider the modifications for $t=1,2,3$. Then either we get a wanted pair
or we will have simultaneously
\begin{equation*}
\left\{\begin{array}{l} d_{k,1}\nmid i_{k,1} \\ D^*_k\mid i_{k+1}\\
\end{array}\right. \hspace{5mm}
\left\{\begin{array}{l} d_{k,1}\mid i_{k,1}-\delta \\ D^*_k\nmid i_{k+1}+\Delta\\
\end{array}\right. \hspace{5mm}
\left\{\begin{array}{l} d_{k,1}\nmid i_{k,1}-2\delta \\ D^*_k\mid i_{k+1}+2\Delta\\
\end{array}\right. \hspace{5mm}
\left\{\begin{array}{l} d_{k,1}\mid i_{k,1}-3\delta \\ D^*_k\nmid i_{k+1}+3\Delta.\\
\end{array}\right. \hspace{5mm}
\end{equation*}
This implies $d_{k,1}\mid 2\delta$ and $D^*_k\mid 2\Delta$. This together with
(\ref{eq:DIV}) implies that both $d_{k,1}$ and $D^*_k$ should be  even. This
is not possible since $d_{k,1}$ and $D^*_k$ are relative prime.

\bekezdes\labelpar{B} Finally  we prove part (b) of the inductive step (\ref{E2}).

Assume that $W_k$ is a strict  allowed divisor on $\Gamma_k$ such that its restriction
 satisfies (1). If $v_k$ is rigid we will use all the notations of (\ref{E2}), where $v_{k'}$ is the closest
 flexible node to $v_k$. In fact, these notations can also  be used when $v_k$ is flexible, with the
 convention $k'=k$.

We have to show that for some modification of $W_k$ one has $d_{k-1,1}\nmid i_{k-1,1}$.
We assume that this is not the case, that is, for all modifications of $W_k$ at $v_{k'}$ one has
\begin{equation}\label{eq:STRICT2}
d_{k-1,1}\mid i_{k-1,1}=
(i_{k'-1,1}+tD_{k'})\cdot D^*_{k}\cdots D^*_{k'-1},
%\ \ \ \mbox{and} \ \ \
%d_{k-1,1}\mid D^*_{k}\cdots D^*_{k'-1} D_{k'}.
\end{equation}
%$d_{k-1,1}\mid i_{k-1,1}=i_{k'-1,1}D^*_{k}\cdots D^*_{k'-1} $
 and we wish to get a contradiction.
%If  for all the modifications, then again by
%(\ref{eq:rigidhh})--(\ref{eq:rigidhhh}) we get

For $\ell\in \{2,\ldots, n_k\}$ set $a_\ell:=\gcd(d_{k-1,1},
d_{k,\ell})$, $A^*:=\prod_{\ell \geq 2}a_l$ and $d'_{k-1,
1}:=d_{k-1,1}/A^*$. About $d'_{k-1, 1}$ we wish to prove two
facts. First, clearly
\begin{equation}\label{eq:DBAR}
d'_{k-1,1}\mid (i_{k'-1,1}+tD_{k'})\cdot D^*_{k+1}\cdots D^*_{k'-1}.
\end{equation}
The second one is
\begin{equation}\label{eq:*b}
d'_{k-1,1}\in \calS\Big\langle
\frac{D_{k+1}}{d_{k+1,2}},\ldots ,
 \frac{D_{k+1}}{d_{k+1,n_{k+1}}},
\frac{D_{k+1}^*D_{k+2}}{d_{k+2,2}}, \ldots,\frac{D_{k+1}^*D_{k+2}}{d_{k+2,n_{k+2}}}, \cdots
\Big\rangle.
\end{equation}
The semigroup involved above is the semigroup associated with that diagram which is obtained from
$\Gamma_0$ by deleting the star--shaped subdiagram around $v_k$. (In fact, (\ref{eq:DBAR}) can also
be interpreted in this way.)
The proof of (\ref{eq:*b}) runs as follows.  Write
$$d_{k-1,1}=\sum_{\ell\geq 2} m_\ell \frac{D_k}{d_{k,\ell}} +D^*_k\cdot \sum_{\ell\geq 2}n_\ell
\frac{D_{k+1}}{d_{k+1,\ell}}+D^*_kD^*_{k+1}\cdot \sum_{\ell\geq
2}n'_\ell \frac{D_{k+2}}{d_{k+2,\ell}}+\cdots.$$ Then $a_\ell\mid
m_\ell$, hence $A^*\mid m_\ell D^*_k/d_{k,\ell} $ too, for any
$\ell\geq 2$. In particular, $d'_{k-1,1}$ belongs to
$$\calS\langle d_{k,1},
D_{k+1}/d_{k+2,1},\ldots,D_{k+1}/d_{k+1,n_{k+1}},
D^*_{k+1}D_{k+2}/d_{k+2,1},\ldots,D^*_{k+1}D_{k+2}/d_{k+2,n_{k+2}},
\cdots \rangle.$$
 But $d_{k,1}$ is in the semigroup generated by the others, cf. (\ref{eq:*}), thus we get (\ref{eq:*b}).

The step how we get $d'_{k-1,1}$ from $d_{k-1,1}$ can be
continued. In the second step we set
$d''_{k-1,1}:=d'_{k-1,1}/\gcd(d'_{k-1,1}, D^*_{k+1})$. Dividing
consecutively by the corresponding divisor of $D^*_k, \ldots,
D^*_{k'-1}$, from $d_{k-1,1}$ we get  $\bar{d}_{k-1,1}$ with the
following properties:
\begin{equation}\label{eq:DBAR2}
\bar{d}_{k-1,1}\mid i_{k'-1,1}+tD_{k'}, \ \ \mbox{or,
equivalently, $\bar{d}_{k-1,1}$ divides both $i_{k'-1,1}$ and
$D_{k'}$},
\end{equation}
\begin{equation}\label{eq:DBAR3}
\bar{d}_{k-1,1}\in \calS\Big\langle
\frac{D_{k'}}{d_{k',1}},\ldots ,
 \frac{D_{k'}}{d_{k',n_{k'}}}
\Big\rangle.
\end{equation}

Now (\ref{eq:DBAR3}) together with $\bar{d}_{k-1,1}\mid D_{k'}$,
 via Lemma~\ref{le:elem}(b),
implies that for some $\ell_0\in \{1,\ldots, n_{k'}\}$ one has
$D_{k'}/d_{k',\ell_0}\mid \bar{d}_{k'-1,1}$.  This with
$\bar{d}_{k-1,1}\mid i_{k'-1,1}$ implies that
$D_{k'}/d_{k',\ell_0}\mid i_{k'-1,1}$. Then the formula
(\ref{eq:j1a}) for $i_{k'-1,1}$ implies that $d_{k',\ell}\mid
i_{k',\ell}$ for $\ell\in\{1,\ldots,n_{k'}\}\setminus \{\ell_0\}$,
which
 contradicts the fact that $v_{k'}$ is flexible.

\bekezdes\labelpar{I} Now we verify  that the above construction provides infinitely many
divisors $W_k$ at each step $k$. Indeed, when $k=m$ then  in (\ref{C})
 there are infinitely many possibilities for $i_{m,1}$ to
realize  a desired $W_m$. Furthermore, in the extension procedure,  this
initially chosen $W_m$ is modified, but the original $i_{m,1}$ is never touched.

Moreover, we can obtain this way infinitely many (extended) \emph{effective} divisors  $W_m$. Indeed, in (\ref{C}) we can choose the value $\nu_m$ freely modulo $N_m$, in particular positive and large enough with respect to all decorations along the edges of $\Gamma$. Then also $i_{m-1,1}$ and further $\{i_{m,\ell}\}_{\ell=1}^{n_m}$ can be chosen \lq large\rq. In fact, if $\nu_m$ is large enough, then all constructed (and modified) multiplicities along dashed arrows, on $\Gamma_1$ and further on the whole of $\Gamma$,  will be positive. This ends the proof of Proposition  \ref{prop:TECH}.
% So at the
%end we will be sure to have obtained infinitely many $W_m$
%satisfying the statement.
\end{proof}

\begin{theorem} \labelpar{thm:allrealized}
Let $(X,0)$ be an IHS germ and $f$ an analytic function on $X$,
such that the minimal embedded resolution diagram
$\Gamma_\pi(X,f)$ satisfies the semigroup condition. Let $\lambda$
be a monodromy eigenvalue of $f$ at a point of $\{f=0\}$. Then
there exist infinitely many (effective) allowed $P$-divisors $W$ for
$(X,\div(f))$, and for each of them a pole $s_0$ of the
topological zeta function $Z(f,W;s)$ such that $\exp(2\pi i
s_0)=\lambda$.
\end{theorem}

\begin{proof} The proof of Theorem \ref{th:ext} is still valid here,
replacing the use of Proposition \ref{prop:extend2} by Proposition
\ref{prop:TECH}.
\end{proof}

\begin{example}\labelpar{ex:EXAMPLE7} We provide an example where
the semigroup condition is not satisfied, and where a given
eigenvalue {\it cannot} be induced by a pole of a zeta function
associated to {\em any} allowed divisor.

We re-consider Example \ref{ex:EXAMPLE}, but taking $F$ as the
divisor corresponding to the unique arrowhead with multiplicity
$N$ (instead of multiplicity 1).

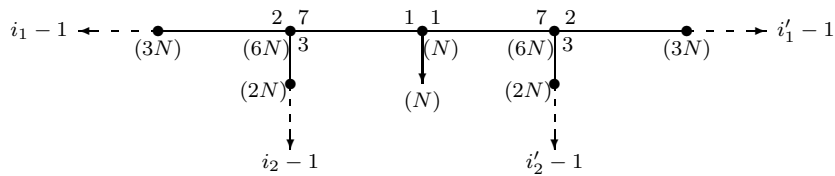
\begin{picture}(400,70)(0,-30)

\put(100,25){\circle*{4}} \put(150,25){\circle*{4}}
\put(200,25){\circle*{4}} \put(250,25){\circle*{4}}
\put(300,25){\circle*{4}} \put(150,5){\circle*{4}}
\put(250,5){\circle*{4}} \put(100,25){\line(1,0){200}}
\put(150,25){\line(0,-1){20}} \put(200,25){\vector(0,-1){20}}
%\put(370,5){\circle*{4}}
\put(250,25){\line(0,-1){20}}
\put(145,30){\makebox(0,0){\tiny{$2$}}}
\put(195,30){\makebox(0,0){\tiny{$1$}}}\put(155,30){\makebox(0,0){\tiny{$7$}}}
\put(245,30){\makebox(0,0){\tiny{$7$}}}\put(205,30){\makebox(0,0){\tiny{$1$}}}
%\put(375,35){\makebox(0,0){\tiny{$-13$}}}
%\put(375,-3){\makebox(0,0){\tiny{$(N)$}}}
\put(256,30){\makebox(0,0){\tiny{$2$}}}
%\put(425,35){\makebox(0,0){\tiny{$-2$}}}
\put(155,20){\makebox(0,0){\tiny{$3$}}}
\put(255,20){\makebox(0,0){\tiny{$3$}}}%\put(322,5){\makebox(0,0){$-7$}}

\put(100,18){\makebox(0,0){\tiny{$(3N)$}}}
\put(141,18){\makebox(0,0){\tiny{$(6N)$}}}
\put(207,18){\makebox(0,0){\tiny{$(N)$}}}
\put(241,18){\makebox(0,0){\tiny{$(6N)$}}}
\put(300,18){\makebox(0,0){\tiny{$(3N)$}}}
\put(140,2){\makebox(0,0){\tiny{$(2N)$}}}
\put(200,-2){\makebox(0,0){\tiny{$(N)$}}}
\put(240,2){\makebox(0,0){\tiny{$(2N)$}}}
\put(55,25){\makebox(0,0){\tiny{$i_1-1$}}}
\put(345,25){\makebox(0,0){\tiny{$i'_1-1$}}}
\put(150,-25){\makebox(0,0){\tiny{$i_2-1$}}}
\put(250,-25){\makebox(0,0){\tiny{$i'_2-1$}}}

\dashline[3]{3}(100,25)(75,25)\put(75,25){\vector(-1,0){5}}
\dashline[3]{3}(300,25)(325,25)\put(325,25){\vector(1,0){5}}
\dashline[3]{3}(150,5)(150,-15)\put(150,-15){\vector(0,-1){5}}
\dashline[3]{3}(250,5)(250,-15)\put(250,-15){\vector(0,-1){5}}
\end{picture}

\noindent Now we have that the Alexander polynomial is
$\Lambda(t)=\Delta_1(t)=(t^{2N}-t^N+1)^2$. We take more
specifically $N=7$, then $N_1=42$, and we pick $\lambda :=
\exp(2\pi i(-5/42)) \in Eig$.

Recall that, cf. Example \ref{ex:EXAMPLE5}, $\nu_1= -78 +7I+6I'$,
where $I=3i_1+2i_2$ and $I'=3i'_1+2i'_2$. We search for an allowed
$W$ such that $-\nu_1/42$ is a pole of $Z(F,W;s)$ and $\nu_1
\equiv 5 $ (mod 42). This last condition is equivalent to $I
\equiv 5$ (mod 6) and $I' \equiv 1$ (mod 7). But for $W$ to be
allowed we need (cf. Example \ref{ex:EXAMPLE4}) that $I=7$ or
$I'=7$. So such an allowed divisor $W$ does not exist. Note that
also the node $v'_1$ cannot induce $\lambda$ by the symmetric
argument, and that the node $v_0$ cannot induce primitive 42-th
roots of unity.
\end{example}

\smallskip
\begin{example}\label{ex:unimod} \

(a) We recall Example (3.5) of B. Rodrigues \cite{Ro}.
Consider the following resolution graph (the right graph  below):

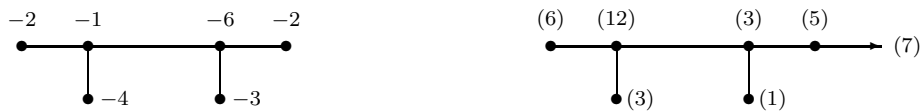
\begin{picture}(400,55)(100,-5)

\put(125,25){\circle*{4}} \put(150,25){\circle*{4}}
%\put(175,25){\circle*{4}}
 \put(200,25){\circle*{4}}
\put(225,25){\circle*{4}} \put(150,5){\circle*{4}}
\put(200,5){\circle*{4}} \put(125,25){\line(1,0){100}}
\put(150,25){\line(0,-1){20}} %\put(175,25){\vector(0,-1){20}}
%\put(370,5){\circle*{4}}
\put(200,25){\line(0,-1){20}} \put(125,35){\makebox(0,0){\tiny{$-2$}}}
\put(150,35){\makebox(0,0){\tiny{$-1$}}}
%\put(175,35){\makebox(0,0){\tiny{$-13$}}}%\put(175,-3){\makebox(0,0){\tiny{$(N)$}}}
\put(200,35){\makebox(0,0){\tiny{$-6$}}}
\put(225,35){\makebox(0,0){\tiny{$-2$}}} \put(160,5){\makebox(0,0){\tiny{$-4$}}}
\put(210,5){\makebox(0,0){\tiny{$-3$}}}%\put(322,5){\makebox(0,0){$-7$}}

\put(325,25){\circle*{4}} \put(350,25){\circle*{4}}
%\put(375,25){\circle*{4}}
 \put(400,25){\circle*{4}}
\put(425,25){\circle*{4}} \put(350,5){\circle*{4}}
\put(400,5){\circle*{4}} \put(325,25){\line(1,0){100}}
\put(350,25){\line(0,-1){20}}
%\put(370,5){\circle*{4}}
\put(400,25){\line(0,-1){20}}
% \put(440,20){\makebox(0,0){\tiny{$\vdots$}}}

\put(325,35){\makebox(0,0){\tiny{$(6)$}}}\put(350,35){\makebox(0,0){\tiny{$(12)$}}}
\put(400,35){\makebox(0,0){\tiny{$(3)$}}}\put(425,35){\makebox(0,0){\tiny{$(5)$}}}
%\put(375,35){\makebox(0,0){\tiny{$-13$}}}
%\put(375,-3){\makebox(0,0){\tiny{$(N)$}}}
\put(359,5){\makebox(0,0){\tiny{$(3)$}}}
%\put(483,15){\makebox(0,0){\tiny{$7$ arrows}}}
\put(409,5){\makebox(0,0){\tiny{$(1)$}}}
\put(460,25){\makebox(0,0){\tiny{$(7)$}}}
%\put(460,5){\makebox(0,0){\tiny{$(1)$}}}
\put(425,25){\vector(1,0){25}}
%\put(425,25){\vector(1,-1){20}}

\end{picture}

\vspace{3mm}

It is easy to verify that it is a numerically Gorenstein elliptic graph with length of the
elliptic sequence two (for terminology, see e.g. \cite{Nweakly}). It was not mentioned in \cite{Ro}, but this graph can be realized by a
hypersurface isolated singularity with multiplicity $3$ and geometric genus $2$ (see also \cite{Y},
case (24) in Table 4). We consider the nowhere vanishing form $\omega_0$ on $X\setminus\{0\}$. The computation in \cite{Ro} shows that, for the indicated Weil divisor $F$, the zeta function
$Z(F,\omega_0;s)$ has $s_0=1/3 $ as a simple pole, but that $\exp(2\pi i/3)$ is not a root of the involved Alexander polynomial.

\smallskip
Note that this example {\em does} satisfy the (analogue of) the
semigroup condition.
Let us explain what we mean by this. Even if a graph is not unimodular,
one can associate with it a splice diagram (by the very similar way
as in (\ref{ss:1})),   and one can impose the semigroup condition in the same way
as above read from the splice diagram, cf. \cite{NW1}.
For example, the `culpable' node with decoration $-6$ (which provides the counterexample
to the `naively generalized' Monodromy Conjecture)  satisfies the semigroup condition,
since the determinant of the $(-2,-1,-4)$ string is 2 which is included in the
semigroup generated by 2 and 4.
Nevertheless, for this graph  the combination of Goal (1) and (2) fails.
The point is that this graph is not unimodular, hence our
main result does not apply to it.

This also shows that in our discussions the IHS--restriction is essential:
any generalization of our main results to non--IHS germs requires the replacement of
the semigroup condition by a much stronger assumption.

\smallskip
(b) One can ask if there is any unimodular graph providing a
 counterexample to the `naively  generalized' Monodromy Conjecture.
Here is one, again, in a combinatorial setting. The form is the standard Gorenstein form,
 whose analytic realization can easily be checked; the function--multiplicities are listed on the
  second diagram, where the  analytic realization of the function is not guaranteed, and
$N$ is a positive integer.
  The example  shows that
in our combinatorial arguments from this section the semigroup assumption cannot be eliminated.

 
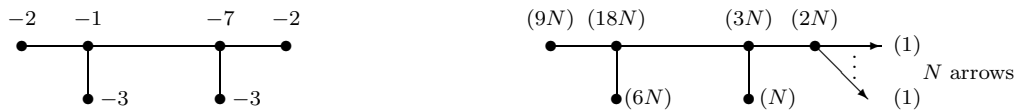
\begin{picture}(400,55)(100,-5)

\put(125,25){\circle*{4}} \put(150,25){\circle*{4}}
%\put(175,25){\circle*{4}}
 \put(200,25){\circle*{4}}
\put(225,25){\circle*{4}} \put(150,5){\circle*{4}}
\put(200,5){\circle*{4}} \put(125,25){\line(1,0){100}}
\put(150,25){\line(0,-1){20}} %\put(175,25){\vector(0,-1){20}}
%\put(370,5){\circle*{4}}
\put(200,25){\line(0,-1){20}} \put(125,35){\makebox(0,0){\tiny{$-2$}}}
\put(150,35){\makebox(0,0){\tiny{$-1$}}}
%\put(175,35){\makebox(0,0){\tiny{$-13$}}}%\put(175,-3){\makebox(0,0){\tiny{$(N)$}}}
\put(200,35){\makebox(0,0){\tiny{$-7$}}}
\put(225,35){\makebox(0,0){\tiny{$-2$}}} \put(160,5){\makebox(0,0){\tiny{$-3$}}}
\put(210,5){\makebox(0,0){\tiny{$-3$}}}%\put(322,5){\makebox(0,0){$-7$}}

\put(325,25){\circle*{4}} \put(350,25){\circle*{4}}
%\put(375,25){\circle*{4}}
 \put(400,25){\circle*{4}}
\put(425,25){\circle*{4}} \put(350,5){\circle*{4}}
\put(400,5){\circle*{4}} \put(325,25){\line(1,0){100}}
\put(350,25){\line(0,-1){20}}
%\put(370,5){\circle*{4}}
\put(400,25){\line(0,-1){20}} \put(440,20){\makebox(0,0){\tiny{$\vdots$}}}

\put(325,35){\makebox(0,0){\tiny{$(9N)$}}}\put(350,35){\makebox(0,0){\tiny{$(18N)$}}}
\put(400,35){\makebox(0,0){\tiny{$(3N)$}}}\put(425,35){\makebox(0,0){\tiny{$(2N)$}}}
%\put(375,35){\makebox(0,0){\tiny{$-13$}}}
%\put(375,-3){\makebox(0,0){\tiny{$(N)$}}}
\put(362,5){\makebox(0,0){\tiny{$(6N)$}}}
\put(483,15){\makebox(0,0){\tiny{$N$ arrows}}}
\put(411,5){\makebox(0,0){\tiny{$(N)$}}}
\put(460,25){\makebox(0,0){\tiny{$(1)$}}}\put(460,5){\makebox(0,0){\tiny{$(1)$}}}
\put(425,25){\vector(1,0){25}}\put(425,25){\vector(1,-1){20}}
\end{picture}

\vspace{2mm}

Clearly, the semigroup condition at the vertex with decoration $-7$  is not satisfied.
(Indeed, the determinant of the $(-2,-1,-3)$ string is 1, which is not an element of
$\calS\langle 2,3\rangle$.)  By a computation one gets that
$7/3N$ is a pole of $Z(s)$, but $\exp(14\pi i/3N) $ is not a root of
 $$\Delta_1(t)=\frac{(t^{9N}+1)(t^{2N}-1)^{N-1}(t-1)}{(t^{3N}+1)(t^N-1)}.$$
%$\Delta_1(t)=t^{6N}-t^{3N}+1$.
\end{example}

\subsection{Final remarks.} \
(a) {\bf (The definition of allowed forms revisited.)}
 \ There is a crucial feature regarding the
definition of the allowed divisors: it does not use the multiplicity system
of the divisor $F$, only {\it its support}. This has the following positive output:
the family of allowed divisors can be
defined uniformly for {\it all} divisors $F$ with the same
support, and all the results we prove are valid uniformly for all
these divisors $F$ (or, functions $f$ with the same support). To
exemplify, let us rewrite Theorem~\ref{thm:monconj} in the following way.

\begin{theorem}\labelpar{thm:monconj7} Let $(X,0)$ be an IHS germ,
and $F'$ a reduced Weil divisor on $X$. Consider an  allowed
divisor $W$  associated with $(X,F')$. Then, for any function $g$
which has  set--theoretical  vanishing set $g^{-1}(0)=F'$, and any
pole $s_0$ of the topological zeta function $Z(g, W; s)$,
$\exp(2\pi is_0)$ is a monodromy eigenvalue of $g$ at some point
of $\{g=0\}$.
\end{theorem}

Note that the zeta--function  $Z(F,W;s)$ and the Alexander
polynomial $\Lambda_{\Gamma(F)}(t)$  do depend essentially on the
multiplicities of $F$.

The above new version (\ref{thm:monconj7}) is definitely a much
stronger statement than the  original (\ref{thm:monconj}). The
interested reader is invited to rewrite all the other results,
especially Theorems \ref{thm:W} and \ref{thm:allrealized} in the
corresponding new versions. Of course, in order to do this, we have to observe that
the  definition of the semigroup condition associated with
$\Gamma_\pi(X,F)$ too {\it depends only on the support of $F$}.

\smallskip

(b) \ {\bf (The restriction (\ref{ss:3})(2) of $W$ revisited.)} \
The restriction (\ref{ss:3})(2) (see also (\ref{bek:INTR3}))
was very convenient in  the computations of arithmetical and
numerical invariants, and additionally created a strong link
between the supports of $F$ and $W$. Moreover, in that choice, we
had in mind the analytic realization of the divisor $W$ too, that is,
the applicability of the main results. More precisely, in general,
it is a rather hard question to determine the analytic realization
of some topologically identified arrowheads/divisors. For example
\cite{NNP} shows that simultaneous realization of some arrowheads
is strongly obstructed.  On the other hand, there is a `natural'
family of analytic singularities for which the analytic
realization of the class of arrowheads considered in
(\ref{ss:3})(2) (arrowheads supported by boundary vertices) is
automatically guaranteed. This is the class of `splice
singularities', cf. \cite{NW1,NWuj}. In is worth to mention that  the
analytic realization of these germs is guaranteed by an
arithmetical property of the graph $\Gamma$ (see End Curve
Theorem  in \cite{NWEC,Ok}), which is nothing else  but the semigroup
condition (\ref{be:semi}).

In this way, the simultaneous appearance of the restriction
(\ref{ss:3})(2) regarding the divisors $W$, and of the semigroup
condition might be natural. Moreover, for a considerably large
class of examples, when the analytic realization of all the forms
$W$ is guaranteed, the semigroup condition too will be satisfied
(compare also with subsection \ref{ss:SC}). This supports strongly
the results of this section, and motivates once again the semigroup
condition, showing that its appearance is not just a technical
necessity (compare also with (\ref{ex:unimod})(b)).

\end{document}